\documentclass[preprint,3p,12pt]{elsarticle}

\usepackage{amssymb}
\usepackage{amsmath}
\usepackage{amsthm}
\usepackage{bbm}
\usepackage{mathtools}
\usepackage[normalem]{ulem}
\usepackage{tikz}

\theoremstyle{definition}
\newtheorem{example}{Example}
\newtheorem{definition}{Definition}
\newtheorem{proposition}{Proposition}
\newtheorem{lemma}{Lemma}
\newtheorem*{remark}{Remark}
\newtheorem{assumption}{Assumption}
\newtheorem{theorem}{Theorem}
\newtheorem{corollary}{Corollary}[theorem]

\newcommand{\Var}{\operatorname{Var}}

\newcommand{\A}[1][]{A_t^{#1}(\epsilon)}
\newcommand{\Z}{\boldsymbol{Z}_{t, T-t, s}^{i}}
\newcommand{\z}[1][i_0, i]{Z_t^{#1}}
\newcommand{\Y}[1][]{%
  \ifx\relax#1\relax
    \boldsymbol{Y}_{t-1, 1, s}^{i}%
  \else
    \boldsymbol{Y}_{t-1, 1, s_{#1}}^{i_{#1}}%
  \fi
}
\newcommand{\V}[1]{V_{#1}}
\newcommand{\assumptions}{%
 \ref{infinite_matrix}--\ref{ass:E_W}
}
\newcommand{\assumptionsHS}{%
 \ref{infinite_matrix}--\ref{ass:extinct}
}

\usepackage{hyperref}

\usepackage{subcaption}
\usepackage{algorithm}
\usepackage{algorithmic}
\usepackage{multirow}

\usepackage{xcolor}

\usepackage{lineno}

\begin{document}

\begin{frontmatter}

\title{On the Coalescence Time Distribution in Multi-type Supercritical Branching Processes}

\author[warwick]{Janique Krasnowska} 
\author[warwick,warwick-cs]{Paul A.\ Jenkins}
\author[warwick]{Adam M.\ Johansen}

\affiliation[warwick]{organization={Department of Statistics, University of Warwick},%Department and Organization
            %addressline={}, 
            city={Coventry},
            postcode={CV4 7AL}, 
            country={United Kingdom}}
\affiliation[warwick-cs]{organization={Department of Computer Science, University of Warwick},%Department and Organization
            %addressline={}, 
            city={Coventry},
            postcode={CV4 7AL}, 
            country={United Kingdom}}

%% Abstract
\begin{abstract}
Consider a population evolving as a discrete-time supercritical multi-type Galton--Watson process. Suppose we run the process for $T$ generations, then sample $k$ individuals uniformly at generation $T$ and trace their genealogy backwards in time. In the limiting regime as $T \rightarrow \infty$, the expected behaviour of the sample's ancestry has been analysed extensively in the single-type case and, more recently, for multi-type processes in the critical case. In this paper, we present a formula for the distribution function of the generation $t$ of the most recent common ancestor in terms of the limiting distribution of the normalised population size. In addition, we provide effective bounds for the decay of this distribution function to 1 in terms of the harmonic moments of the population size at generation $t$. 
In order to better understand the behaviour of these harmonic moments, we use a multi-type generalisation of the Harris--Sevastyanov transformation to express harmonic moments at generation $t$ in terms of moments of the transformed process at the first generation. We present numerical results demonstrating that it is possible to approximate the coalescence time distribution effectively in practical settings.
\end{abstract}
 
%% Keywords
\begin{keyword}
Discrete-time branching process \sep Coalescent Process \sep Genealogy \sep 
Galton--Watson Trees \sep Harmonic moments \sep Harris--Sevastyanov Transformation 
\end{keyword}

\end{frontmatter}
\section{Introduction}
\thispagestyle{empty}
Branching processes constitute a broad class of stochastic models which find applications in areas such as biology \cite{cell_cycle, luria, colon, complexity}, ecology \cite{lambert}, epidemiology \cite{epidemiology}, physics \cite{harris1963}, and computational statistics \cite{multilevel}. 
In the simplest possible setting, consider a discrete-time single-type branching process where each individual reproduces independently, following an identical offspring distribution, and the offspring of all individuals at generation $t$ forms the next generation. This model and variations thereof are often referred to as a \emph{Galton--Watson process}, which can for example be used as a suitable model for populations of cells, genes, or biomolecules \cite{branching_biology}. 
A natural generalisation of the Galton--Watson process is to include a variety of types of individuals with different offspring distributions. Such \emph{multi-type} branching processes find applications in theoretical biology \cite{cell_cycle}, evolutionary theory \cite{colon, complexity}, applied ecology \cite{lambert}, and genetics \cite{luria}, yet these models are less well understood theoretically. In this paper, we study a multi-type branching process in which the number of types may be finite or countably infinite. Our results apply to both settings, though the countably infinite case requires slightly more restrictive assumptions.

Branching processes are endowed with a natural notion of genealogy. Understanding the ancestry of a sample backwards in time is a fundamental question in evolutionary biology. Phenomena such as genetic drift or the selective advantage of an allele in an evolving population can be investigated in a more computationally efficient way by simulating an ancestral tree backwards instead of running the process forwards in time \cite{maddison, Rosenberg2002}. On longer timescales, genealogical trees are also essential to the study of speciation and extinction \cite{coalescent, speciation, popvic}. Results have recently been obtained in the case of non-neutral populations of fixed size (see \cite{koskela2024genealogical} and references therein) but these techniques cannot easily be generalised to the setting of branching processes. The genealogy of a branching process has been studied extensively in the single-type case \cite{lambert, Athreya_2012, athreya_super, harris2020, johnston}. 

Branching processes are classified according to their expected rate of growth. If the expected number of offspring is greater than $1$, then the process is said to be \emph{supercritical} and the population size either grows geometrically or becomes zero (the latter is termed \emph{extinction} and occurs with a probability denoted by $q \in [0, 1)$). 
For a critical, multi-type, continuous-time branching process, its genealogy has been recently described in \cite{osvaldo}. Supercritical multi-type branching process are considered by \cite{Hong2015} who finds the distribution of the generation of the most recent common ancestor (MRCA) in a model with finitely many types that cannot die out.

In this paper, we obtain an expression for the distribution function of the generation, $t$, of the MRCA, for a sample of $k \ge 2$ individuals drawn uniformly without replacement from generation $T$. The expression is given in terms of the limiting distribution of the normalised population size and generalises the result of \cite{Hong2015} to a process that can become extinct and can have a countably infinite number of types. 
In addition, we provide upper and lower bounds with explicit constants for this distribution function. Initially, the bounds depend on the harmonic means of the population size at time $t$, which may not be easily computable. In order to obtain practically useful bounds, we use a generalisation of the Harris--Sevastyanov transformation to a multi-type branching process, which allows us to express harmonic moments of the process at generation $t$ in terms of harmonic moments of the transformed process at the first generation. 

The outline of this paper is as follows. In Section \ref{prelim} we recall some fundamental results about branching processes. We then state our main results in Section \ref{main_results}: Theorem \ref{exact_thm} relates the distribution function of the coalescence time probability in terms of a collection of random variables which characterise the limiting growth of the population; Theorem \ref{harmonic_thm} provides upper and lower bounds on the same probability in terms of the harmonic moments of the total population size conditioned on non-extinction; and Theorem \ref{hs_thm} relates those moments to those of a simpler transformed process that cannot die out. Proofs of these results are given in each of the following three sections. Section \ref{sec:numerical} provides numerical simulations illustrating the practical utility of our results, while Section \ref{sec:discussion} briefly concludes.

\section{Preliminaries}\label{prelim}

Let $\mathbb{N}_0 := \{0, 1, 2,\dots\}$, $\mathbb{N}_+ := \mathbb{N}_0 \setminus \{0\}$, and for integer $n \in \mathbb{N}_+$ denote by $[n]$ the set $\{1,\dots,n\}$. The falling factorial of $x \in \mathbb{R}$ will be denoted by $(x)_k = \prod_{j=0}^{k-1} (x - j)$ for $k \in \mathbb{N}_+$. For $d \in \mathbb{N}_+ \cup \{\infty\}$, $\boldsymbol{e}_i$ will denote the $i$th canonical basis column vector in $\mathbb{R}^d$ endowed with the $\ell^2$ inner product, and $\boldsymbol{0} := (0,0,\dots)^\top$ and $\boldsymbol{1} := (1,1,\dots)^\top$ denote the vectors of all zeros and all ones respectively. The inner product of any two vectors $\boldsymbol{u}$, $\boldsymbol{v}$ will be denoted by $\boldsymbol{u} \cdot \boldsymbol{v}$, and it will be convenient to use $\vert \boldsymbol{u} \vert$ to denote the 1-norm. The element-wise product and quotient of two vectors will be represented by $\boldsymbol{u}\odot\boldsymbol{v}$ and $\boldsymbol{u}\oslash\boldsymbol{v}$ respectively. All (in)equalities between random variables are assumed to hold almost surely unless otherwise specified.

Consider a multi-type Galton--Watson process
\begin{equation*}
    \boldsymbol{Z}^{i_0} = \left(\boldsymbol{Z}_t^{i_0} = \begin{bmatrix}
        Z_t^{i_0, 1} \\
        Z_t^{i_0, 2}\\ 
        \vdots 
    \end{bmatrix} :\: t=0,1,\dots\right)
\end{equation*}
with type-space $S \subseteq\mathbb{N}_+$ started from a single individual of type $i_0\in S$ and with $Z_t^{i_0,i}$ descendants of type $i$ at time $t$. We let $d := |S|$ denote the number (possibly infinite) of types, and without loss of generality we enumerate types so that either $S=\mathbb{N}_+$ or $S=[d]$. The process $\boldsymbol{Z}^{i_0}$ is a time-homogeneous discrete-time Markov process on $\mathbb{N}_0^{d}$ such that 
\[
\boldsymbol{Z}_{t+1}^{i_0}\bigm| \boldsymbol{Z}_{t}^{i_0} = \sum_{i \in S} \sum_{j=1}^{Z_{t}^{i_0, i}} \boldsymbol{\xi}^i_{t,j},
\]
where $(\boldsymbol{\xi}^i_{t,j}: j=1,\dots,Z_{t}^{i_0,i})$ are i.i.d.\ $\mathbb{N}_0^{d}$-valued random vectors representing the offspring associated with parent $j$ of type $i$ in generation $t$. The distribution of $ \boldsymbol{\xi}^i_{t,j}$ will be denoted
\[
p_i(\boldsymbol{v}) \coloneq \mathbb{P}(\boldsymbol{\xi}^i_{t,j} = \boldsymbol{v}), \qquad \boldsymbol{v} \in \mathbb{N}_0^d.
\]
If the type of the root is omitted in the superscript of $\boldsymbol{Z}_t$, we assume the root ancestor to be of type $i_0$.

The branching property permits us to decompose $\boldsymbol{Z}^{i_0}_T$, for any fixed $t$, as
\begin{equation}\label{eq:Zdecomposition}
\boldsymbol{Z}^{i_0}_T = \sum_{i\in S}\sum_{s=1}^{Z_t^{i_0, i}} \boldsymbol{Z}^i_{t,T-t,s},
\end{equation}
where $\boldsymbol{Z}_{t, T-t, s}^i$, $s=1,\dots,Z_t^{i_0,i}$ are independent branching processes started at generation $t$ from an individual of type $i$ and evolving for a further $T-t$ generations; that is, $\boldsymbol{Z}_{t, T-t, s}^i$ represents the descendants of the $s$th individual of type $i$ who existed in generation $t$ and were themselves a descendant of the original founder of type $i_0$. The law of each $\boldsymbol{Z}_{t, T-t, s}^i$ is the same as that of $\boldsymbol{Z}^i_{T-t}$. We make extensive use of \eqref{eq:Zdecomposition} later.

The long-term behaviour of $\boldsymbol{Z}^{i_0}$ can be described in terms of the mean offspring matrix $\boldsymbol{M}=(m_{ij})_{i,j \in S}:=(\mathbb{E}(Z_1^{i, j}))_{i,j \in S}$. For the powers $\boldsymbol{M}^r=(m_{ij}^{(r)})$, $r \in \mathbb{N}_+$, to be well defined, we assume that $\boldsymbol{M}$ is such that all $\boldsymbol{M}^r$ are elementwise finite.

$\boldsymbol{M}$ is said to be \emph{irreducible} if for any pair of indices $i,j \in S$ there exists some positive integer $k$ such that $m_{ij}^{(k)} > 0$. The assumption of irreducibility allows for the existence of a general limiting distribution for the system. Without it, different subsets of communicating types can exhibit distinct limiting behaviours.

\begin{assumption}\label{infinite_matrix}
    $\boldsymbol{M}$ is irreducible and all $m_{ij}^{(r)}$, $r \in \mathbb{N}_+$, $i,j \in S$, are finite.
\end{assumption}

For a pair $i,j \in S$, define the power series
\begin{equation*}
    M_{ij}(z) := \sum_{n=0}^{\infty} m_{ij}^{(n)} z^n, \qquad z \in \mathbb{C}.
\end{equation*}
Under Assumption \ref{infinite_matrix}, $M_{ij}(z)$ has a common radius of convergence $R \in [0, \infty)$ for all $i,j \in S$ (Chapter 6.1 of \cite{seneta}). If for all $i, j \in S$ we have that $M_{ij}(R) < \infty$, then for any $i,j \in S$, the sequence of rescaled population sizes $R^T Z_T^{i, j}$ converges in mean to 0 as $T \to \infty$. Hence, in order to obtain a non-trivial limit, we are interested in the case where for any $i,j \in S$
\begin{equation}
    M_{ij}(R) = \infty. \label{eq:unique}
\end{equation}
A non-negative and non-zero $\boldsymbol{\nu}^{\top}=((\nu_i)_{i \in S})^{\top} \in \mathbb{R}^{d}$ is called an \textit{$R$-invariant measure} (with the obvious abuse of terminology) if it satisfies
\begin{equation*}
    R \boldsymbol{\nu} \boldsymbol{M} = \boldsymbol{\nu},
\end{equation*}
and a non-negative and non-zero $\boldsymbol{u}=(u_i)_{i \in S} \in \mathbb{R}^{d}$ is called an \textit{$R$-invariant vector} if it satisfies

\begin{equation*}
    R \boldsymbol{M} \boldsymbol{u} = \boldsymbol{u}.
\end{equation*}
In the case when equation \eqref{eq:unique} holds, $\boldsymbol{u}$ and $\boldsymbol{\nu}$ exist, are strictly positive and are unique up to a constant multiple (see for example \cite{moy1967}). Under the additional assumption that $\boldsymbol{u}\cdot \boldsymbol{\nu} < \infty$, we can normalise $\boldsymbol{\nu}$ and $\boldsymbol{u}$ so that they satisfy $ \boldsymbol{u} \cdot\boldsymbol{\nu} = 1$ and $\boldsymbol{\nu} \cdot \boldsymbol{1}=1$.

We collect the additional regularity assumption imposed upon our process in the following:
\begin{assumption}\label{R-positive}
  $M_{ij}(R) = \infty$ for any pair $i,j \in S$, $\boldsymbol{u} \cdot\boldsymbol{\nu} < \infty$, and $R \in (0, 1)$ (the process is supercritical).
\end{assumption}
Under Assumptions \ref{infinite_matrix} and \ref{R-positive}, the population size of any type grows asymptotically exponentially with a random initial scaling factor. We state this more formally below.
\begin{theorem}[Theorem 1 of \cite{moy1967}] \label{moy}
Under Assumptions~\ref{infinite_matrix} and \ref{R-positive}, for all possible founding ancestor types $i_0 \in S$, there exists a real, non-negative random variable $W^{(i_0)}$ with $\mathbb{E}(W^{(i_0)}) = u_{i_0}$ such that as $T \rightarrow \infty$,
\begin{equation*}
    R^T Z_T^{i_0,j} \xrightarrow{L^2} \nu_j W^{(i_0)},
\end{equation*}
and, hence, the full population size satisfies
\begin{equation*}
   R^T |\boldsymbol{Z}_T^{i_0}| \xrightarrow{L^2} W^{(i_0)}, \qquad T \to \infty.
\end{equation*}
\end{theorem}
For $i_0 \in S$, define the \emph{generating function} of $\boldsymbol{Z}_1^{i_0}$, $f^{i_0}: [0,1]^{d} \rightarrow \mathbb{R}$, as
\begin{equation*}
    f^{i_0}(\boldsymbol{s}) := \sum_{\boldsymbol{w}\in \mathbb{N}_0^{d}}\left( p_{i_0}(\boldsymbol{w}) \prod_{j \in S}s_j^{w_j}\right),
\end{equation*}
where we adopt the convention that $0^0=1$ so that $f^{i_0}(\boldsymbol{0})=p_{i_0}(\boldsymbol{0})$. For $t \in \mathbb{N}_+$, the generating function of $\boldsymbol{Z}_t^{i_0}$ is given by iterates of $f^{i_0}$ (see for example Chapter 2.2 of \cite{harris1963}):
\begin{equation*}
    f^{i_0}_{t+1}(\boldsymbol{s}) = f^{i_0}(f_t^1(\boldsymbol{s}), \dots, f_t^d(\boldsymbol{s}))
\end{equation*}
with the convention that $f_1^{i_0}(\boldsymbol{s}) = f^{i_0}(\boldsymbol{s})$ and $f_0^{i_0}(\boldsymbol{s}) = s_{i_0}$. 

    Let us introduce the notion of extinction: the process becomes \emph{extinct} at time $t$ if $|\boldsymbol{Z}_t|=0$ and $|\boldsymbol{Z}_{t-1}|>0$; in this case $|\boldsymbol{Z}_s|=0$ for all $s \geq t$. For each $i \in S$, let $q_i$ be the probability of eventual extinction of the process initiated with a single particle of type $i$. Define a vector of all extinction probabilities $\boldsymbol{q}:= (q_i)_{i \in S}$.

    \begin{proposition}[Theorem 7.1 in \cite{harris1963}]\label{prop:compute_extinct}
        Under Assumptions \ref{infinite_matrix} and \ref{R-positive}, $\boldsymbol{q}$ satisfies the equation
\begin{equation*}
\boldsymbol{q} = \boldsymbol{f}(\boldsymbol{q})
\end{equation*}
where $\boldsymbol{f}(\boldsymbol{s})=(f^i(\boldsymbol{s}))_{i \in S}$.
    \end{proposition}
    \begin{remark}
        Note that Proposition \ref{prop:compute_extinct} implies that for any $t \in \mathbb{N}_+$
        \begin{equation*}
            \boldsymbol{q} = \boldsymbol{f}_t(\boldsymbol{q}),
        \end{equation*}
        where $\boldsymbol{f}_t(\boldsymbol{s})=(f^i_t(\boldsymbol{s}))_{i \in S}$.
    \end{remark}

We end this section with a brief example illustrating these concepts.   
\begin{example}\label{deterministic2}
    Consider a branching process with $S=\mathbb{N}_+$ where for $p \in (1/2, 1)$
    \begin{equation*}
        p_i(2\boldsymbol{e}_{i+1}) = 1-p, \qquad p_i(2\boldsymbol{e}_{i-1}) = p, \quad i = 2,3,\dots,
    \end{equation*}
    and for $i=1$, $ p_1(2\boldsymbol{e}_2) = 1$. The process satisfies Assumptions \ref{infinite_matrix} and \ref{R-positive}. The radius of convergence of this system is $R=\tfrac{1}{2}$ (\ref{radius_convergence}). This aligns with the intuition that the system is experiencing a deterministic growth factor of two in each generation.

The $R$-invariant measure $\boldsymbol{\nu}$ of $\boldsymbol{M}$ satisfies 
\begin{equation*}
    \nu_1 = p\nu_2, \quad \nu_1 + p \nu_3 = \nu_2,
\end{equation*}
as well as the recurrence relation
\begin{equation*}
    (1-p)\nu_{i-1} + p \nu_{i+1} = \nu_i, \qquad i=3,4,\dots
\end{equation*}
with an appropriate choice of $\nu_1$ to satisfy the normalisation condition. 
This allows us to arrive at the formula
\begin{equation*}
    \nu_i =\left(\frac{1-p}{p}\right)^{i-2}\frac{1}{p} \nu_1, \qquad i=2,3,\dots.
\end{equation*}
In addition, solving for the invariant vector of $\boldsymbol{M}$ gives 
\begin{equation}
p u_{i-1}+ (1-p)u_{i+1} = u_i, \qquad i=2,3,\dots
\end{equation}
with $u_1=u_2$. Consequently, we can say that all elements of $\boldsymbol{u}$ are equal, which under the normalisation conditions implies that $\boldsymbol{u}=\boldsymbol{1}$. 
\end{example}

\section{Main Results}\label{main_results}
The main results of this paper characterise the distribution of the time of the MRCA of a collection of individuals drawn from the population at time $T$ in the limit $T \to \infty$. Consider sampling $k \ge 2$ individuals uniformly without replacement from generation $T$. Denote the sampled indices by $\boldsymbol{x}=\{x_i\}_{i=1}^k$, $x_i \in [|\boldsymbol{Z}_T^{i_0}|]$, with all elements of $\boldsymbol{x}$ distinct. Consider the random variable $X_{T, k}$ describing the generation of the MRCA of all individuals in $\boldsymbol{x}$. An example of the behaviour of $X_{T, k}$ for $k=2$ is illustrated in Figure \ref{fig:branching}. 
    \begin{figure}
\centering
    \begin{tikzpicture}
% Line
    \draw[thick,->] (0,0) -- (10,0);
    % Marks
    \draw[thin] (0, -0.2) -- (0, 0.2);
        \draw[thin] (1, -0.2) -- (1, 0.2);
    \draw[thin] (8, -0.2) -- (8, 0.2);
    \draw[thin] (4, -0.2) -- (4, 0.2);
    % Text
    \draw (0,-0.5) node {$0$};
        \draw (1,-0.5) node {$1$};
    \draw (4, -0.5)  node {$t$};
    \draw (8, -0.5)  node {$T$};

    \draw [dashed, opacity=0.7] (0,-0.1) -- (0,4);
        \draw [dashed, opacity=0.7] (1,-0.1) -- (1,4);
            \draw [dashed, opacity=0.7] (4,-0.1) -- (4,4);
                \draw [dashed, opacity=0.7] (8,-0.1) -- (8,4);
    % Gen 0
    \node[shape=circle, fill=blue, text=blue](one1) at (0,2) {o};
    % Gen 1
    \node[shape=rectangle, fill=red, text=red](two1) at (1,2.5) {O};
    \node[shape=circle, fill=blue, text=blue](two2) at (1,1.5) {o};
    % Gen t
    \node[shape=rectangle, fill=red, text=red](three1) at (4,3) {O};
    \node[shape=circle, fill=blue, text=blue](three2) at (4,2) {o};
    \node[shape=circle, fill=blue, text=blue](three3) at (4,1) {o};
    % Gen t+1
    \node[shape=rectangle, fill=red, text=red](four1) at (6,3.5) {O};
    \node[shape=circle, fill=blue, text=blue](four2) at (6,2.5) {o};
    \node[shape=rectangle, fill=red, text=red](four3) at (6,1.5) {O};
    \node[shape=circle, fill=blue, text=blue](four4) at (6,0.5) {o};
    % Gen T
    \node[shape=rectangle, fill=red, text=red](five1) at (8,3.5) {O};
    \node[shape=circle, fill=blue, text=blue](five2) at (8,2.75) {o};
    \node[shape=rectangle, fill=red, text=red](five3) at (8,2) {O};
    \node[shape=rectangle, fill=red, text=red](five4) at (8,1.25) {O};
    \node[shape=circle, fill=blue, text=blue](five5) at (8,0.5) {o};
    % Arrows
    % Gen 0
    \draw[-] (one1) -- (two1) ;
    \draw[-] (one1) -- (two2) ;
    % Gen t
    \draw[-] (three1) -- (four1) ;
    \draw[-] (three1) -- (four2) ;
    \draw[-] (three2) -- (four3) ;
    \draw[-] (three3) -- (four4) ;
    % Gen T
    \draw[-] (four1) -- (five1) ;
    \draw[-] (four1) -- (five2) ;
    \draw[-] (four2) -- (five3) ;
   \draw[-, purple] (four4) -- (five4) ;
    \draw[-, purple] (four4) -- (five5) ;
    
    \node at (2.5, 2) {\dots};
    % highlights for genealogy
    \draw[purple, dashed] (7.5,0.1) rectangle (8.5,1.6);
\end{tikzpicture}
\vspace{0.2cm}
\caption{An example realisation of a branching process evolving from generation $0$ to generation $T$ where for the bottom two individuals at time $T$ $X_{T, 2} > t$, whereas for the top and bottom individuals $X_{T, 2} < t$.}
\label{fig:branching}
\end{figure}
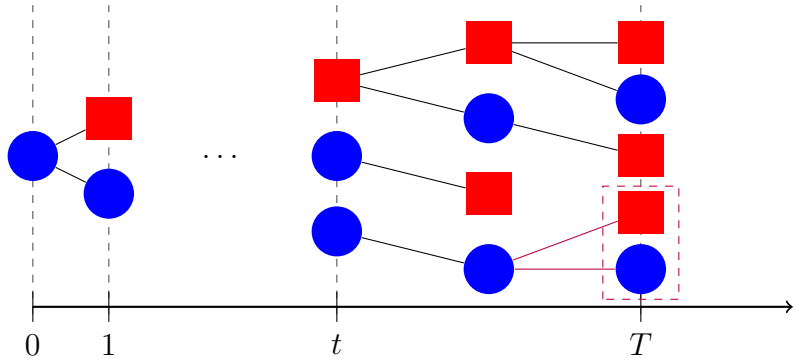

Now, consider the probability $\mathbb{P}(X_{T,k} < t \bigm| |\boldsymbol{Z}_T|\ge k)$ that $k$ individuals sampled at generation $T$ have as their MRCA an individual from a generation older than $t \in (0, T]$ conditioned on the population size at generation $T$ being at least $k$. We can express the asymptotic behaviour of this probability as $T\to\infty$ in terms of the expected value of an expression dependent on the limit of the normalised population size. To state the result formally, we introduce some further notation. Let $\{|\boldsymbol{Z}_{\infty}| > 0\} = \cap_{T\in\mathbb{N}_+} \{|\boldsymbol{Z}_T| > 0\}$ be the event that extinction does not occur at any time $T \in \mathbb{N}_0$, and let $W^{(i)}_{t,s}:= \lim_{T\rightarrow \infty} R^{t-T}|\boldsymbol{Z}^i_{t, T-t, s}| $ denote the limiting rescaled full population size of a process started by an ancestor of type $i$ with index $s$ in generation $t$ (recall \eqref{eq:Zdecomposition}). The law of $W^{(i)}_{t,s}$ does not depend upon $s$ or $t$. 

Ref.\ \cite{Hong2015} considers a process with a finite number of types where individuals of all types produce at least one offspring, i.e.\ the probability of extinction is $0$ regardless of the type of the root ancestor. 
\begin{theorem}[Theorem 2.1 in \cite{Hong2015}]\label{thm_hong}
Under Assumption \ref{infinite_matrix}, for a process with $\boldsymbol{q} = \boldsymbol{0}$ and $d < \infty$, it holds that
    \begin{equation*}
         \lim_{T \rightarrow \infty}\mathbb{P}(X_{T,k} < t) % \label{lim_prob}
    = 1 - \mathbb{E}\left(\frac{\sum_{i=1}^d\sum_{s=1}^{\z} (W^{(i)}_{t,s})^k }
    % denominator
    {(\sum_{i=1}^d\sum_{s=1}^{\z}W^{(i)}_{t,s})^k}  \right).
    \end{equation*}
\end{theorem}
Our Theorem \ref{exact_thm} extends Theorem \ref{thm_hong} to a process with a non-zero probability of extinction and a possibly countably infinite number of types. This extension is non-trivial owing to more delicate convergence conditions when the type space is countably infinite; this is reflected in the slightly more extensive assumptions needed in this case. Furthermore, working with a process that can experience extinction requires additional technical arguments. 
\begin{theorem} \label{exact_thm}
 Under Assumptions \ref{infinite_matrix} and \ref{R-positive}, we have
\begin{equation}
    \lim_{T \rightarrow \infty}\mathbb{P}(X_{T,k} < t\bigm| |\boldsymbol{Z}_T| \ge k ) % \label{lim_prob}
    = 1 - \mathbb{E}\left(\frac{\sum_{i \in S}\sum_{s=1}^{\z} (W^{(i)}_{t,s})^k }
    % denominator
    {(\sum_{i \in S}\sum_{s=1}^{\z}W^{(i)}_{t, s})^k}\bigm| |\boldsymbol{Z}_{\infty}| > 0  \right) .\label{equal_expression}
\end{equation}
\end{theorem}
The right hand-side of (\ref{equal_expression}) is difficult to estimate for large but fixed values of $t$, prompting an analysis of the rate of decay of this quantity to 0 as a function of $t$. To that end, it is convenient to impose further assumptions on $\{W^{(i)}\}_{i \in S}$ and on the extinction probabilities.

\begin{assumption}\label{moments}
    For all $i \in S$, $\mathbb{E}((W^{(i)})^{2k}) < \infty$.
\end{assumption}

When $d < \infty$, Assumption \ref{moments} can be checked more easily by investigating the finiteness of the $2k$th moment of the offspring distribution.

\begin{proposition} \label{prop:W_vs_Z1}
Under Assumptions \ref{infinite_matrix} and \ref{R-positive}, if $d < \infty$ and if for all $i,j \in S$ we have that
\begin{equation}
    \mathbb{E}\left(\left(Z_1^{i,j}\right)^{2k}\right) < \infty \label{eq:finite_moments},
\end{equation}
    then Assumption \ref{moments} holds. %it is true that for all $i \in S$
%    \begin{equation*}
%        \mathbb{E}((W^{(i)})^{2k}) < \infty.
%    \end{equation*}
\end{proposition}

The proof of Proposition \ref{prop:W_vs_Z1} can be found in \ref{W_vs_Z1}. The relationship between the moments of the offspring distribution and the limit of the rescaled population size is well known in the single-type case \cite{bingham}, but, to the best of our knowledge, has not previously been proved explicitly in the multi-type case.

Condition \ref{eq:finite_moments} is satisfied by a lot of standard offspring distributions. Any distribution with a bounded support satisfies the finiteness of moments assumption. For distributions with unbounded support, consider a system where each type of parent $i \in S$ gives birth to a number of children of type $j \in S$ independently of the number of children of other types. In this setting, many commonly chosen offspring distributions---such as linear-fractional, geometric, binomial, and Poisson (refer to \cite{branching_biology} and references therein)---possess finite moments.

 All results can be significantly simplified if we assume that $\sup_{i \in S} q_i < 1$. In both the finite-type and countably infinite case, $q_i < 1$ for all $i \in S$ if $\boldsymbol{M}$ is positive regular and supercritical \cite{hautphenne}. We want to avoid the situation where some $q_i$ are arbitrarily close to 1, i.e.\ that $\sup_{i \in S} q_i = 1$.
\begin{assumption}\label{ass:extinct}
    $\sup_{i \in S} q_i < 1$.
\end{assumption}
Similarly, we want to avoid the situation where, even conditioned on non-extinction, populations can grow subexponentially, i.e.\ abnormally compared to the expected exponential growth. 
\begin{assumption}\label{ass:E_W}
$ \inf_{i \in S} u_i =\inf_{i \in S} \mathbb{E}(W^{(i)}) > 0$.
\end{assumption}
\begin{remark}
    Note that $\inf_{i \in S} \mathbb{E}(W^{(i)}) > 0$ implies that, for any $j \in \mathbb{N}_+$, $\inf_{i \in S} \mathbb{E}((W^{(i)})^j)>0$.
\end{remark}
\begin{remark}
    For a supercritical, irreducible branching process with a \emph{finite} number of types satisfying Assumption \ref{moments}, Assumptions \ref{ass:extinct} and \ref{ass:E_W} are redundant since for all $i \in S$, $q_i > 0$ and $\mathbb{E}(W^{(i)}) > 0$. 
\end{remark}
Expression (\ref{equal_expression}) can be bounded in terms of the harmonic moments of $|\boldsymbol{Z}_t|$ which provide a route to tractable numerical approximations. 
\begin{theorem} \label{harmonic_thm}
Under Assumptions \assumptions, for any $\epsilon \in (0, \inf_{i \in S, j \in \{1,k\}}{\mathbb{E}[(W^{(i)})^j]})$, it holds that
    \begin{align*}
&1 -(C_1^{\epsilon}+C_3^{\epsilon})\mathbb{E}\left(\frac{1}
    % Denominator
    {|\boldsymbol{Z}_t|}\bigm| |\boldsymbol{Z}_{\infty}| > 0\right)
\le \lim_{T \rightarrow \infty}\mathbb{P}(X_{T,k} < t| |\boldsymbol{Z}_T| \ge k ) \\
& \le
    1 - C_2^{\epsilon}\left(\mathbb{E}\left(\frac{1}{|\boldsymbol{Z}_t|^{k-1}}\bigm||\boldsymbol{Z}_{\infty}| > 0\right) - C_3^{\epsilon} \mathbb{E}\left(\frac{1}{|\boldsymbol{Z}_t|^{k}}\bigm||\boldsymbol{Z}_{\infty}| > 0\right)\right),
\end{align*}
where $C_1^{\epsilon} := \frac{\sup_{i \in S}\mathbb{E}((W^{(i)})^k)+\epsilon}
    % Denominator
    {(\inf_{i \in S}\mathbb{E}(W^{(i)})-\epsilon)^k}$, $C_2^{\epsilon} := \frac{\inf_{i \in S}\mathbb{E}((W^{(i)})^k)-\epsilon}{(\sup_{i \in S}\mathbb{E}(W^{(i)})+\epsilon)^k}$, and $C_3^{\epsilon} := 
    % unlikely event bit
    \sum_{j \in \{1, k\}}\frac{\sup_{i \in S}{\Var((W^{(i)})^j)}}{ \epsilon^2(1 - \sup_{i \in S}q_i)}$.
\end{theorem}

The following lemma concerns the computation of harmonic moments and generalises equation (1) from \cite{athreya2003} to the multi-type setting with conditioning.
\begin{lemma}\label{gamma_lemma}
For any $r \in \mathbb{N}_+$, the $r$th harmonic moment of a supercritical process $|\boldsymbol{Z}_t|$ initiated with a single type $i_0 \in S$ can be computed as
\begin{equation}
 \mathbb{E}\left(\frac{1}{|\boldsymbol{Z}_t|^r} \bigm| |\boldsymbol{Z}_t| > 0\right) = \frac{1}{\Gamma(r)}\int_0^{\infty} u^{r-1} \frac{f_t^{i_0}(e^{-u}\boldsymbol{1}) - f_t^{i_0}(\boldsymbol{0})}{1-f_t^{i_0}(\boldsymbol{0})} du  \label{eq:harmonic_moment_formula}.
\end{equation}
\end{lemma}
The proof can be found in \ref{gamma}.

The evaluation of \eqref{eq:harmonic_moment_formula} requires the $t$th iterate of the generating function, which can be hard to compute numerically. To further analyse the decay of the harmonic moments for large values of $t$, it is going to be helpful to transform the original process into a process that cannot become extinct. In the single-type case, this transformation is called the Harris--Sevastyanov transformation and is used quite commonly (refer to \cite{harris1948, Bingham_1988}). The generalisation of this transformation to a multi-type process seems less commonly known and has properties that differ non-trivially from the single-type case. The only explicit statement of the general transformation and its properties known to us is found in \cite{jones} where it is used to analyse the left tail limit of the normed population size by considering the events of extinction and unusually slow growth separately. In addition, \cite{sagitov2013} considers this transformation specifically for the linear-fractional offspring distribution, while \cite{hs_mention} and \cite{hs_mention2} mention the transformation but do not state it explicitly.

The Harris--Sevastyanov transformation lets us transform the original supercritical process $(\boldsymbol{Z}_t)$ into a related process with null probability of extinction regardless of the type of the root ancestor. 
\begin{definition}[Multi-type Harris--Sevastyanov transformation]
  For a supercritical, multi-type branching process $(\boldsymbol{Z}_t)$, define a new process $(\boldsymbol{Y}_t)$ by setting $\boldsymbol{Y_0}=\boldsymbol{Z}_0$ and specifying its generating functions $\boldsymbol{F}(\boldsymbol{s})=(F^i(\boldsymbol{s}))_{i \in S}$, $F^i : [0,1]^{d}\rightarrow \mathbb{R}$ as
\begin{equation*}
\boldsymbol{F}(\boldsymbol{s}) = (\boldsymbol{f}(\boldsymbol{s} \odot (\boldsymbol{1}-\boldsymbol{q})+\boldsymbol{q}) - \boldsymbol{q})\oslash(\boldsymbol{1}-\boldsymbol{q}).
    % \boldsymbol{F}(\boldsymbol{s}) = \frac{\boldsymbol{f}(\boldsymbol{s}(\boldsymbol{1}-\boldsymbol{q})+\boldsymbol{q})- \boldsymbol{q}}{\boldsymbol{1}-\boldsymbol{q}}
\end{equation*}
\end{definition}
\begin{lemma}\label{lemma:iterates}
 The iterates of generating functions are related in an analogous way:
\begin{equation*}
    \boldsymbol{F}_t(\boldsymbol{s}) = (\boldsymbol{f}_t(\boldsymbol{s} \odot (\boldsymbol{1}-\boldsymbol{q})+\boldsymbol{q}) - \boldsymbol{q})\oslash(\boldsymbol{1}-\boldsymbol{q}).
\end{equation*}
\end{lemma}
\begin{lemma}\label{lemma:expected}
The expected values of $\boldsymbol{Z}_t$ and $\boldsymbol{Y}_t$ satisfy
\begin{equation*}
    \mathbb{E}(\boldsymbol{Y}_t^{i_0}) = \frac{\boldsymbol{1}-\boldsymbol{q}}{1-q_{i_0}}\odot\mathbb{E}(\boldsymbol{Z}_t^{i_0}).
\end{equation*}
 \end{lemma}
 The proof of Lemmata \ref{lemma:iterates} and \ref{lemma:expected} can be found in \ref{iterates} and \ref{expect_hs}, respectively.
In the following theorem, we show how $\mathbb{E}\left(|\boldsymbol{Z}_t|^{-1}\bigm| |\boldsymbol{Z}_{\infty}| > 0\right)$ can be bounded in terms of quantities which depend only upon $\boldsymbol{Y}_1$. These can be computed much more easily than the harmonic moment of $|\boldsymbol{Z}_t|$.

\begin{theorem} \label{hs_thm}
    Under Assumptions \assumptions, for any $r \in \mathbb{N}_+$,
 \begin{equation*}
         (1-\sup_{i \in S} q_i)^r \left( \mathbb{E}\left(\sup_{i \in S} |\boldsymbol{Y}_1^i|\right) \right)^{-tr} \le 
    \mathbb{E}\left(\frac{1}{|\boldsymbol{Z}_t|^{r}}\bigm| |\boldsymbol{Z}_{\infty}| > 0\right) 
    \le (1-\sup_{i \in S} q_i)^{-2} \left(\sup_{i \in S} \mathbb{E}\left(\frac{1}{|\boldsymbol{Y}_1^i|}\right)\right)^{tr}.
 \end{equation*}
\end{theorem}
Finally, combining the bounds from Theorem \ref{hs_thm} with the result of Theorem \ref{harmonic_thm}, in the following corollary we clearly see the exponential rate at which the probability of coalescing before generation $t$ converges to $1$.
\begin{corollary}\label{corr:mainresult}
    Under Assumptions \assumptions, for any $ \epsilon \in (0, \inf_{i \in S, j \in \{1,k\}}{\mathbb{E}((W^{(i)})^j)})$, it holds that
    \begin{align*}
1 -C_4^{\epsilon}\left(\mathbb{E}\left( \sup_{i \in S}\frac{1}{|\boldsymbol{Y}_1^i|}\right)\right)^t
&\le \lim_{T \rightarrow \infty}\mathbb{P}(X_{T,k} < t| |\boldsymbol{Z}_T| \ge k ) \\
&\le
    1 - C_5^{\epsilon}\left(\mathbb{E}\left(\sup_{i \in S} |\boldsymbol{Y}_1^i|\right) \right)^{-t(k-1)}  + C_6^{\epsilon}\left(\mathbb{E}\left(\sup_{i \in S}\frac{1}{|\boldsymbol{Y}_1^i|}\right)\right)^{tk}
\end{align*}
where $C_4^{\epsilon} := (C_1^{\epsilon}+C_3^{\epsilon}) (1-\sup_{i \in S} q_i)^{-2}$, $C_5^{\epsilon} := C_2^{\epsilon} (1-\sup_{i \in S} q_i)^k 
    $ and $C_6^{\epsilon} := C_2^{\epsilon} C_3^{\epsilon} (1-\sup_{i \in S} q_i)^{-2}$ and $C_1^{\epsilon}, C_2^{\epsilon}$ and $C_3^{\epsilon}$ are as in Theorem \ref{harmonic_thm}.
\end{corollary}
\section{Proof of Theorem~\ref{exact_thm}}\label{exact_section}
The structure of the proof follows \cite{Hong2015}, but accounts for the possibility of the process becoming extinct and having a possibly  countably infinite number of types. The proof relies on decomposing the population at time $T$ in terms of a sum of families started at an earlier time $t$---recall \eqref{eq:Zdecomposition}. When considering the limiting behaviour of the coalescence probability, the details rely on the continuous mapping theorem and a modified version of the conditional dominated convergence theorem to express $\mathbb{P}(X_{T,k} > t \bigm|  |\boldsymbol{Z}_T| \ge k)$ in terms of the limits of every independent branching process started from each individual alive at generation $t$. 

Consider $\mathbb{P}(X_{T,k} \ge t \bigm| |\boldsymbol{Z}_T| \ge k)$, the probability that all sampled individuals coalesced in the interval $[t, T)$. If the sampled individuals share a MRCA on $[t, T)$, then they are descended from the same ancestor at generation $t$. 
Since the sampled individuals can be descended from any of the families started at generation $t$, we can express $\mathbb{P}(X_{T,k} \ge t \bigm| |\boldsymbol{Z}_T| \ge k)$ as a sum over each possible ancestor  at time $t$ of the probability that all $k$ individuals sampled at time $T$ are descended from that ancestor.

Conditional upon a realisation of the process until time $T$, for a type $i \in S$ and an individual $s \in [\z]$, the probability of the first sampled individual originating from ancestor with index $s$ of type $i$ is exactly the proportion of individuals alive at time $T$ descended from that ancestor,
%\begin{equation*}
    ${|\Z|}/{|\boldsymbol{Z}_T|}$. 
%\end{equation*}
Analogously, the probability of the second sampled individual originating from the same ancestor is again equal of the size of the family started by individual $s$ of type $i$ divided by the total population size after removing the already sampled individual from the population:
%\begin{equation*}
$({|\Z|-1})/{(|\boldsymbol{Z}_t|-1)}$.
%\end{equation*}
Thus $\mathbb{P}(X_{T,k} \ge t \bigm| |\boldsymbol{Z}_T| \ge k)$ can be expressed as the expected value, with respect to the law of the process, of these quantities. 
%Denote by $\boldsymbol{Z}_{t, T-t, s}^i$ the branching process started from an ancestor of type $i$ with index $s$ in generation $t$, and allowed to evolve for $T-t$ generations.
We can write, recalling that the Pochhammer symbol, $(\cdot)_k$, denotes the $k$th falling factorial:
\begin{align}
\mathbb{P}(X_{T,k} \ge t \bigm| |\boldsymbol{Z}_T| \ge k) %\\
    &=\mathbb{E}\left(\frac{\sum_{i \in S}\sum_{s=1}^{\z} (|\Z|)_k }
    % denominator 
    {(|\boldsymbol{Z}_T|)_k} \bigm| |\boldsymbol{Z}_T| \ge k\right) \nonumber \\
    &=\mathbb{E}\left(\frac{\sum_{i \in S}\sum_{s=1}^{\z}  (|\Z|)_k}
    % denominator
    {(\sum_{i \in S} \sum_{s=1}^{\z} |\Z|)_k }\bigm| |\boldsymbol{Z}_T| \ge k\right), \label{last_is_sum}
\end{align}
where expression (\ref{last_is_sum}) is obtained by decomposing the denominator using \eqref{eq:Zdecomposition}. 
%the population at time $T$ in terms of the offspring of individuals at time $t$ 
%\begin{equation} \label{eq:ZT}
%|\boldsymbol{Z}_T|=    \sum_{i \in S} \sum_{s=1}^{\z} |\Z|.
%\end{equation}
Letting $T\to\infty$ we find
\begin{align}
    \lim_{T \rightarrow \infty}\mathbb{P}(X_{T,k} \ge t\bigm| |\boldsymbol{Z}_T| \ge k) %\label{first_line}\\
    %&= \lim_{T \rightarrow \infty} \mathbb{E}\left(\frac{ \sum_{i \in S}\sum_{s=1}^{\z}  (|\Z|)_k}
    % denominator
    %{(\sum_{i \in S} \sum_{s=1}^{\z} |\Z|)_k}\bigm| |\boldsymbol{Z}_T| \ge k\right) \nonumber  \\
    &=   \mathbb{E}\left(\lim_{T \rightarrow \infty}\frac{ \sum_{i \in S}\sum_{s=1}^{\z}  (|\Z|)_k}
    % denominator
    {(\sum_{i \in S} \sum_{s=1}^{\z} |\Z|)_k}\bigm| |\boldsymbol{Z}_{\infty}| >0\right),\label{eq:limit_expectation}
\end{align}
where the exchange of limits and the conditioning event in \eqref{eq:limit_expectation} is possible due to a modified version of the conditional dominated convergence theorem, a proof of which can be found in \ref{dominated}. We rescale the population size in  \eqref{eq:limit_expectation} by $(R^{T-t})^k$ to find
\begin{align}
\lim_{T \rightarrow \infty}\mathbb{P}(X_{T,k} \ge t\bigm| |\boldsymbol{Z}_T| \ge k) 
%&\mathbb{E}\left(\lim_{T \rightarrow \infty}\frac{ \sum_{i \in S}\sum_{s=1}^{\z}  (|\Z|)_k}
    % denominator
    %{(\sum_{i \in S} \sum_{s=1}^{\z} |\Z|)_k}\bigm| |\boldsymbol{Z}_{\infty}| >0\right) \notag\\
    & = \mathbb{E}\left(\lim_{T \rightarrow \infty}\frac{ \sum_{i \in S}\sum_{s=1}^{\z} \prod_{j=0}^{k-1} R^{T-t}(|\Z|-j)}
    % denominator
    {\prod_{j=0}^{k-1}\sum_{i \in S} \sum_{s=1}^{\z} R^{T-t}(|\Z| - j)}\bigm| |\boldsymbol{Z}_{\infty}| >0\right).\label{division}
\end{align}
Let us investigate the convergence of $ R^{T-t}(|\Z| - j )$ for $j \in \{0,\dots,k-1\}$. Under Assumptions~\ref{infinite_matrix} and~\ref{R-positive}, using Theorem~\ref{moy}, the limiting behaviour of $R^{T-t} (|\Z| - j)$ is not influenced by the constant $j$, so
\begin{equation*}
    R^{T-t} (|\Z| - j) = R^{T-t} |\Z| - R^{T-t} j \xrightarrow{L^2} W^{(i)}_{t, s}, \qquad T\to\infty.
\end{equation*}
$L^2$ convergence implies convergence in probability, so we can apply the continuous mapping theorem (see, e.g., \cite[Chapter 5, p.\ 103]{kallenberg}) for the ratio, sum, and product of the random variables in (\ref{division}) to obtain
\begin{align*}
    %&\mathbb{E}\left(\lim_{T \rightarrow \infty}\frac{ \sum_{i \in S}\sum_{s=1}^{\z} \prod_{j=0}^{k-1} R^{T-t}(|\Z|-j)}
    % denominator
    %{\prod_{j=0}^{k-1}\sum_{i \in S} \sum_{s=1}^{\z} R^{T-t}(|\Z| - j)}\bigm| |\boldsymbol{Z}_{\infty}| >0\right) \nonumber\\
    \lim_{T \rightarrow \infty}\mathbb{P}(X_{T,k} \ge t\bigm| |\boldsymbol{Z}_T| \ge k) 
    &= \mathbb{E}\left(\frac{ \sum_{i \in S}\sum_{s=1}^{\z}  \prod_{j=0}^{k-1} W^{(i)}_{t, s}}
    % denominator
    {\prod_{j=0}^{k-1}\sum_{i \in S} \sum_{s=1}^{Z_t^{i_0,i}} W^{(i)}_{t, s}}\bigm| |\boldsymbol{Z}_{\infty}| >0\right) \nonumber\\
    &=\mathbb{E}\left(\frac{ \sum_{i \in S}\sum_{s=1}^{\z}   (W^{(i)}_{t, s})^k}
    % denominator
    {(\sum_{i \in S} \sum_{s=1}^{Z_t^{i_0,i}} W^{(i)}_{t, s})^k}\bigm| |\boldsymbol{Z}_{\infty}| >0\right), %\label{eq:continuous_mapping}
\end{align*}
as desired.
%Combining \eqref{eq:limit_expectation}, \eqref{division} and \eqref{eq:continuous_mapping} yields \eqref{equal_expression}, the desired result.
%\begin{align}
%    \lim_{T \rightarrow \infty}\mathbb{P}(X_{T,k} < t\bigm| |\boldsymbol{Z}_T| \ge k) %\nonumber \\
%    & = 1- \mathbb{E}\left(\frac{ \sum_{i \in S}\sum_{s=1}^{\z}   (W^{(i)}_{t, s})^k}
%    % denominator
%    {\left(\sum_{i \in S} \sum_{s=1}^{Z_t^{i_0,i}}W^{(i)}_{t, s}\right)^k}\bigm| |\boldsymbol{Z}_{\infty}| > 0\right). \label{eq:just_above}
%\end{align}
\section{Proof of Theorem \ref{harmonic_thm}}\label{harmonic_section}

The proof is structured as follows. We relate the right-hand side of equation \eqref{equal_expression} to the harmonic moments of $|\boldsymbol{Z}_t|$ by first dividing the numerator and denominator of the ratio inside the expectation by $|\boldsymbol{Z}_t|$. We then consider the event, $\A$, on which the sum of $W^{(i)}_{t,s}$s lies within $\epsilon$ of its expected value, which allows us to bound the coalescence probability in terms of the harmonic moments of $|\boldsymbol{Z}_t|$ and explicit constants. 

Recalling \eqref{equal_expression}, we first rewrite
%We condition on the population size at time $t$ by using the tower rule for conditional expectation, and then divide both the numerator and denominator by the full population size at time $t$:
\begin{align}
    &\mathbb{E}\left(\frac{ \sum_{i \in S}\sum_{s=1}^{\z}   (W^{(i)}_{t, s})^k}
    % denominator
    {(\sum_{i \in S} \sum_{s=1}^{\z} W^{(i)}_{t, s})^k}\bigm| |\boldsymbol{Z}_{\infty}| > 0\right) \nonumber\\
    % &=\mathbb{E}\left(\mathbb{E}\left(\frac{\sum_{i \in S}\sum_{s=1}^{\z} (W_{t,s}^{(i)})^k }
    % denominator
    % {(\sum_{i \in S}\sum_{s=1}^{\z}W_{t,s}^{(i)})^k}\bigm| |\boldsymbol{Z}_{\infty}| > 0, \boldsymbol{Z}_t  \right)  \bigm| |\boldsymbol{Z}_{\infty}| > 0\right) \label{tower_rule}\\
    %&= \mathbb{E}\left(\mathbb{E}\left(\frac{\sum_{i \in S}\sum_{s=1}^{\z} \frac{(W^{(i)}_{t, s})^k}{|\boldsymbol{Z}_t|} }
    % denominator
    %{(|\boldsymbol{Z}_t|^{\frac{k-1}{k}}\sum_{i \in S}\sum_{s=1}^{\z}\frac{W^{(i)}_{t, s}}{|\boldsymbol{Z}_t|})^k}\bigm| |\boldsymbol{Z}_{\infty}| > 0, \boldsymbol{Z}_t  \right)  \bigm| |\boldsymbol{Z}_{\infty}| > 0\right)
    &= \mathbb{E}\left(\mathbb{E}\left(\frac{\V{k}}
    {(|\boldsymbol{Z}_t|^{\frac{k-1}{k}}\V{1})^k}\bigm| |\boldsymbol{Z}_{\infty}| > 0, \boldsymbol{Z}_t  \right)  \bigm| |\boldsymbol{Z}_{\infty}| > 0\right),
    \label{beforeLLN}
\end{align}
where we will make extensive use of the quantities
\begin{equation}
   \V{j} := \sum_{i \in S}\sum_{s=1}^{\z}\frac{(W^{(i)}_{t, s})^j}{|\boldsymbol{Z}_t|}, \qquad j\in\mathbb{N}_+. \label{eq:V}
\end{equation}
Let us now explore the conditional mean and variance of $\V{j}$ given $\boldsymbol{Z}_t$.

\begin{lemma}\label{lemma_moments}
    Under Assumption \ref{infinite_matrix} and \ref{R-positive}, for $j \in \mathbb{N}_+$
      \begin{align} \label{eq:condEidentity}
        %\mathbb{E}\left(\sum_{i \in S}\sum_{s=1}^{\z}\frac{(W^{(i)}_{t, s})^j}{|\boldsymbol{Z}_t|}\bigm|  \boldsymbol{Z}_t, |\boldsymbol{Z}_t| > 0 \right) &= \frac{\sum_{i \in S}\z \mathbb{E}((W^{(i)})^j)}{|\boldsymbol{Z}_t|} ,\\
        %\text{ and } \quad
        %\Var\left(\sum_{i \in S}\sum_{s=1}^{\z}\frac{(W^{(i)}_{t, s})^j}{|\boldsymbol{Z}_t|}\bigm|  \boldsymbol{Z}_t, |\boldsymbol{Z}_t| > 0 \right) &= \frac{\sum_{i \in S} \z \Var((W^{(i)})^j) }{|\boldsymbol{Z}_t|^2}.\label{eq:condVidentity}
        \mathbb{E}\left(\V{j}\bigm|  \boldsymbol{Z}_t, |\boldsymbol{Z}_t| > 0 \right) &= \frac{\sum_{i \in S}\z \mathbb{E}((W^{(i)})^j)}{|\boldsymbol{Z}_t|} ,\\
        \text{ and } \quad
        \Var\left(\V{j}\bigm|  \boldsymbol{Z}_t, |\boldsymbol{Z}_t| > 0 \right) &= \frac{\sum_{i \in S} \z \Var((W^{(i)})^j) }{|\boldsymbol{Z}_t|^2}.\label{eq:condVidentity}
    \end{align}
\end{lemma}
The proof is included in \ref{sums_W}.

Next we relate the probability of an event $B$ conditioned on $\{ |\boldsymbol{Z}_{\infty}| > 0\} $ with the probability of the same event conditioned on $\{ |\boldsymbol{Z}_t| > 0 \}$, which will allow us to take advantage of Chebyshev's inequality.
\begin{lemma}\label{future_present}
 Under Assumptions \ref{infinite_matrix} and \ref{R-positive}, for any event $B$ in the relevant probability space,
    \begin{align*}
         \mathbb{P} \left(B \bigm| |\boldsymbol{Z}_{\infty}| > 0, \boldsymbol{Z}_t\right) &\le \frac{\mathbb{P}(B \bigm||\boldsymbol{Z}_{t}| > 0, \boldsymbol{Z}_t)}
    % denominator
    {1 - \prod_{i \in S}q_i^{\z}}.
    \end{align*}
\end{lemma}
\begin{proof}
Consider the probability %of an event $B$ conditioned on the population configuration at time $t$ and $\{|\boldsymbol{Z}_{\infty}| > 0\}$ 
\begin{align}
    \mathbb{P}\left(B \bigm| |\boldsymbol{Z}_{\infty}| > 0, \boldsymbol{Z}_t\right) %\nonumber \\
    &=\mathbb{P}(B \bigm||\boldsymbol{Z}_{\infty}| > 0,|\boldsymbol{Z}_{t}| > 0, \boldsymbol{Z}_t) \label{more_conditioning} \\
    &=\frac{\mathbb{P}(B,|\boldsymbol{Z}_{\infty}| > 0\bigm||\boldsymbol{Z}_{t}| > 0, \boldsymbol{Z}_t)}
    % denominator
    {\mathbb{P}(|\boldsymbol{Z}_{\infty}| > 0\bigm||\boldsymbol{Z}_{t}| > 0, \boldsymbol{Z}_t)}  \label{eq:fraction},
\end{align}
where (\ref{more_conditioning}) follows from the fact that $\{|\boldsymbol{Z}_{\infty}| > 0\}\subset\{|\boldsymbol{Z}_t| > 0\}$ up to null sets. Now, by the law of total probability %$\mathbb{P}(B|\mathcal{G})\ge \mathbb{P}(A, B| \mathcal{G})$ 
applied to \eqref{eq:fraction},
\begin{align}
\frac{\mathbb{P}(B,|\boldsymbol{Z}_{\infty}| > 0\bigm||\boldsymbol{Z}_{t}| > 0, \boldsymbol{Z}_t)}
    % denominator
    {\mathbb{P}(|\boldsymbol{Z}_{\infty}| > 0\bigm||\boldsymbol{Z}_{t}| > 0, \boldsymbol{Z}_t)} %\nonumber \\
    &\le \frac{\mathbb{P}(B\bigm||\boldsymbol{Z}_{t}| > 0, \boldsymbol{Z}_t)}
    % denominator
    {\mathbb{P}(|\boldsymbol{Z}_{\infty}| > 0\bigm||\boldsymbol{Z}_{t}| > 0, \boldsymbol{Z}_t)} %\notag \\
    = \frac{\mathbb{P}(B\bigm||\boldsymbol{Z}_{t}| > 0, \boldsymbol{Z}_t)}
    % denominator
    {1 - \prod_{i \in S}q_i^{\z}} \label{before_lln},
\end{align}
where the equality in \eqref{before_lln} follows from the definition of extinction probabilities $\boldsymbol{q}$. 
\end{proof}
Setting %$B=\left\{\left|\sum_{i \in S}\sum_{s=1}^{\z}\frac{(W^{(i)}_{t, s})^j}{|\boldsymbol{Z}_t|}-\frac{\sum_{i \in S}\z \mathbb{E}((W^{(i)})^j)}{|\boldsymbol{Z}_t|}\right| \ge\epsilon\right\}$
$B=\left\{\left|\V{j}-\mathbb{E}\left(\V{j}\bigm|  \boldsymbol{Z}_t\right)\right| \ge\epsilon\right\}$, the following corollary is immediate.
\begin{corollary}
For any $\epsilon > 0$,
      \begin{align}
      %   &\mathbb{P} \left(\left|\sum_{i \in S}\sum_{s=1}^{\z}\frac{(W^{(i)}_{t, s})^j}{|\boldsymbol{Z}_t|}-\frac{\sum_{i \in S}\z \mathbb{E}((W^{(i)})^j)}{|\boldsymbol{Z}_t|} \right|\ge\epsilon\bigm\vert |\boldsymbol{Z}_{\infty}| > 0, \boldsymbol{Z}_t\right) \label{eq:z_infinity}\\
      %   &\le \frac{\mathbb{P}(|\sum_{i \in S}\sum_{s=1}^{\z}\frac{(W^{(i)}_{t, s})^j}{|\boldsymbol{Z}_t|}-\frac{\sum_{i \in S}\z \mathbb{E}((W^{(i)})^j)}{|\boldsymbol{Z}_t|} |\ge\epsilon\bigm||\boldsymbol{Z}_{t}| > 0, \boldsymbol{Z}_t)}
    % denominator
    %{1 - \prod_{i \in S}q_i^{\z}}.\label{eq:before_chebyshev}
    \mathbb{P} \left(\left|\V{j}-\mathbb{E}\left(\V{j}\bigm|  \boldsymbol{Z}_t\right) \right|\ge\epsilon\bigm\vert |\boldsymbol{Z}_{\infty}| > 0, \boldsymbol{Z}_t\right) %\label{eq:z_infinity}\\
         &\le \frac{\mathbb{P}(|\V{j}-\mathbb{E}\left(\V{j}\bigm|  \boldsymbol{Z}_t\right) |\ge\epsilon\bigm||\boldsymbol{Z}_{t}| > 0, \boldsymbol{Z}_t)}
    % denominator
    {1 - \prod_{i \in S}q_i^{\z}}.\label{eq:before_chebyshev}
    \end{align}
\end{corollary}
%\subsection{Chebyshev's inequality}
Now we can apply Chebyshev's inequality to the numerator of the right-hand side of \eqref{eq:before_chebyshev} to obtain a further upper bound for the left-hand side. 
\begin{lemma}\label{chebyshev}
    Under Assumptions \assumptionsHS, for any $\epsilon > 0$, $j \in \mathbb{N}_+$, it is true that
    \begin{align*}
    %\mathbb{P}\left(\left|\sum_{i \in S}\sum_{s=1}^{\z}\frac{(W^{(i)}_{t, s})^j}{|\boldsymbol{Z}_t|}-\frac{\sum_{i \in S}\z \mathbb{E}((W^{(i)})^j)}{|\boldsymbol{Z}_t|} \right|\ge\epsilon \bigm| | \boldsymbol{Z}_{\infty}| > 0, \boldsymbol{Z}_t\right) %\\
    %& \le \frac{\sup_{i \in S}{\Var((W^{(i)})^j)}}{|\boldsymbol{Z}_t| \epsilon^2(1 - \sup_{i \in S}q_i)}.
    \mathbb{P}\left(\left|\V{j}- \mathbb{E}\left(\V{j}\bigm|  \boldsymbol{Z}_t\right) \right|\ge\epsilon \bigm| | \boldsymbol{Z}_{\infty}| > 0, \boldsymbol{Z}_t\right) %\\
    & \le \frac{\sup_{i \in S}{\Var((W^{(i)})^j)}}{|\boldsymbol{Z}_t| \epsilon^2(1 - \sup_{i \in S}q_i)}.
\end{align*}
\end{lemma}
\begin{proof}
   For $\epsilon > 0$ consider conditioned Chebyshev's inequality (e.g.\ \cite[Lemma 5.1, p.\ 102]{kallenberg}) for a random variable $V$ with a finite expectation $\mathbb{E}(V)$ and variance $\Var(V)$ and a sub-$\sigma$-algebra $\mathcal{G}$:
\begin{equation*}
    \mathbb{P}(|V-\mathbb{E}(V\vert \mathcal{G})| > \epsilon \bigm| \mathcal{G}) \le \frac{\Var(V \bigm| \mathcal{G})}{\epsilon^2},
\end{equation*}
Under Assumption \ref{moments}, we can set $V = \V{j}$, $\mathcal{G} = \sigma(\boldsymbol{Z}_t)$, and restrict to $\{\vert \boldsymbol{Z}_t\vert > 0\}$. Now, we have by Lemma \ref{lemma_moments}
\begin{align*}
    \mathbb{E}(V \bigm| \mathcal{G}, \{\vert \boldsymbol{Z}_t\vert > 0\}) &= \frac{\sum_{i \in S}\z \mathbb{E}((W^{(i)})^j)}{|\boldsymbol{Z}_t|}  \\
    % variance
    \text{ and } \quad \Var(V\bigm| \mathcal{G}, \{\vert \boldsymbol{Z}_t\vert > 0\}) &= \frac{\sum_{i \in S}\z \Var((W^{(i)})^j) }{|\boldsymbol{Z}_t|^2}.
\end{align*}
Chebyshev's inequality applied to the numerator of expression \eqref{eq:before_chebyshev} yields:
\begin{align}
    &\mathbb{P}(|\V{j}- \mathbb{E}\left(\V{j}\bigm|  \boldsymbol{Z}_t\right) |\ge\epsilon\bigm||\boldsymbol{Z}_{t}| > 0, \boldsymbol{Z}_t)\nonumber \\
    &\le \frac{\sum_{i \in S}\z \Var((W^{(i)})^j)}{|\boldsymbol{Z}_t|^2 \epsilon^2 } 
    \le \frac{\sup_{i \in S}{\Var((W^{(i)})^j)}\sum_{i \in S}\z}{|\boldsymbol{Z}_t|^2 \epsilon^2}
    =\frac{\sup_{i \in S}{\Var((W^{(i)})^j)}}{|\boldsymbol{Z}_t| \epsilon^2}.\label{with_type}
\end{align}
Hence \eqref{eq:before_chebyshev} becomes
\begin{align*}
    %&\mathbb{P}(|\sum_{i \in S}\sum_{s=1}^{\z}\frac{(W^{(i)}_{t,s})^j}{|\boldsymbol{Z}_t|}-\frac{\sum_{i \in S}\z \mathbb{E}((W^{(i)})^j)}{|\boldsymbol{Z}_t|} |\ge\epsilon|\boldsymbol{Z}_{\infty} > 0, \boldsymbol{Z}_t)\nonumber \\
    %&\le\frac{\mathbb{P}(|\sum_{i \in S}\sum_{s=1}^{\z}\frac{(W^{(i)}_{t,s})^j}{|\boldsymbol{Z}_t|}-\frac{\sum_{i \in S}\z \mathbb{E}((W^{(i)})^j)}{|\boldsymbol{Z}_t|} |\ge\epsilon\bigm||\boldsymbol{Z}_{t}| > 0, \boldsymbol{Z}_t)}
    %% denominator
    %{1 - \prod_{i \in S}q_i^{\z}} \nonumber \\
    \mathbb{P}(| \V{j}-\mathbb{E}\left(\V{j}\bigm|  \boldsymbol{Z}_t\right) |\ge\epsilon|\boldsymbol{Z}_{\infty} > 0, \boldsymbol{Z}_t) &\le\frac{\mathbb{P}(|\V{j}- \mathbb{E}\left(\V{j}\bigm|  \boldsymbol{Z}_t\right) |\ge\epsilon\bigm||\boldsymbol{Z}_{t}| > 0, \boldsymbol{Z}_t)}
    {1 - \prod_{i \in S}q_i^{\z}} \nonumber \\
    &\le \frac{\sup_{i \in S}{\Var((W^{(i)})^j)}}{|\boldsymbol{Z}_t| \epsilon^2(1 - \prod_{i \in S}q_i^{\z})}
    \le \frac{\sup_{i \in S}{\Var((W^{(i)})^j)}}{|\boldsymbol{Z}_t| \epsilon^2(1 - \sup_{i \in S}q_i)} \label{eq:dependency_population}
\end{align*}
    which is well defined under Assumptions \ref{ass:extinct} and \ref{ass:E_W}.
%where in \eqref{eq:dependency_population} we removed the dependence on the number of each type by bounding the product of extinction probabilities by the maximum extinction probability.
\end{proof}
For $j \in \{1, k\}$ and $\epsilon > 0$, consider the intersection of events where $\V{j}$ is close to its expected value for both values of $j$:
\begin{equation*}
    \A = \bigcap_{j\in \{1,k\}} \left\{|\V{j} - \mathbb{E}\left(\V{j}\bigm|  \boldsymbol{Z}_t\right) |<\epsilon \right\}.
\end{equation*} 
Let us return to (\ref{beforeLLN}) and decompose the expectation over $\A$ and $\A[C]$:
\begin{align}
    \mathbb{E}\Bigg(\mathbb{E}\bigg(&\frac{\V{k}} 
    {(|\boldsymbol{Z}_t|^{\frac{k-1}{k}}\V{1})^k}\bigm| |\boldsymbol{Z}_{\infty}| > 0, \boldsymbol{Z}_t  \bigg)  \bigm| |\boldsymbol{Z}_{\infty}| > 0\Bigg) \\
    = \mathbb{E}\Bigg(\mathbb{E}&\bigg(\frac{\V{k}}
    {(|\boldsymbol{Z}_t|^{\frac{k-1}{k}}\V{1})^k} \mathbbm{1}_{\A} + % Other event
    \frac{\V{k}}
    {(|\boldsymbol{Z}_t|^{\frac{k-1}{k}}\V{1})^k} \mathbbm{1}_{\A[C]}
    %conditioning
    \bigm| |\boldsymbol{Z}_{\infty}| > 0,\boldsymbol{Z}_t  \bigg)\bigm| |\boldsymbol{Z}_{\infty}| > 0\Bigg) \label{before_bound}.
    \end{align}
As shown in \ref{less1}:
    \begin{equation}
         % Other event
      \frac{\V{k}}
    {(|\boldsymbol{Z}_t|^{\frac{k-1}{k}}\V{1})^k} \le 1 \label{less_than1}
    %conditioning
%    \bigm| |\boldsymbol{Z}_{\infty}| > 0,\boldsymbol{Z}_t  \bigg)\bigm| |\boldsymbol{Z}_{\infty}| > 0\Bigg) 
    \end{equation}
    almost surely on the event $\{|\boldsymbol{Z}_{\infty}| > 0\}$.
    
By Assumption~\ref{ass:E_W}, we can define $\epsilon \in (0, \inf_{i \in S} \mathbb{E}(W^{(i)}))$ and apply inequality \eqref{less_than1} to \eqref{before_bound} to get 
    \begin{align}
         \mathbb{E}\Bigg(\mathbb{E}&\bigg(\frac{\V{k}}
    {(|\boldsymbol{Z}_t|^{\frac{k-1}{k}}\V{1})^k} \mathbbm{1}_{\A} + % Other event
    \frac{\V{k}}
    {(|\boldsymbol{Z}_t|^{\frac{k-1}{k}}\V{1})^k} \mathbbm{1}_{\A[C]}
    %conditioning
    \bigm| |\boldsymbol{Z}_{\infty}| > 0,\boldsymbol{Z}_t  \bigg)\bigm| |\boldsymbol{Z}_{\infty}| > 0\Bigg)\notag\\
    % First inequality
     \le \mathbb{E}&\left(\frac{\mathbb{E}\left(\V{k}\bigm|  \boldsymbol{Z}_t\right)+\epsilon}
    % Denominator
    {\left(|\boldsymbol{Z}_t|^{\frac{k-1}{k}}\left(\mathbb{E}\left(\V{1}\bigm|  \boldsymbol{Z}_t\right)-\epsilon\right)\right)^k}\bigm| |\boldsymbol{Z}_{\infty}| > 0\right)+ 
    % Expectation
    \mathbb{E}\left(\mathbb{E}(\mathbbm{1}_{A^C_t}\bigm| |\boldsymbol{Z}_{\infty}| > 0, \boldsymbol{Z}_t) \bigm| |\boldsymbol{Z}_{\infty}| > 0\right)\label{eq:close_to_mean} \\
    % minimum, maximum replacement
     \le \mathbb{E}&\left(\frac{\sup\limits_{i \in S}\mathbb{E}((W^{(i)})^k)+\epsilon}
    % Denominator
    {\left(|\boldsymbol{Z}_t|^{\frac{k-1}{k}}(\inf\limits_{i \in S}\mathbb{E}(W^{(i)})-\epsilon)\right)^k}\bigm| |\boldsymbol{Z}_{\infty}| > 0\right)+ 
    % Expectation
    \mathbb{E}\left(\mathbb{P}(A^C_t\bigm| |\boldsymbol{Z}_{\infty}| > 0, \boldsymbol{Z}_t)\bigm| |\boldsymbol{Z}_{\infty}| > 0\right), \label{min_max} 
\end{align}%
where (\ref{min_max}) is obtained by upper bounding $\sum_{i \in S}\z \mathbb{E}((W^{(i)})^k)$ by $\sum_{i \in S}\z \sup_{i \in S}\mathbb{E}((W^{(i)})^k)$ and lower bounding $\sum_{i \in S}\z \mathbb{E}(W^{(i)})$ by $\sum_{i \in S}\z \inf_{i \in S}\mathbb{E}(W^{(i)})$.

Now, the second term in \eqref{min_max} can be upper bounded since by Lemma \ref{chebyshev} we have %the probability of the complement of event $\A$ is bounded by
\begin{equation}
    \mathbb{P}(\A[C]|\boldsymbol{Z}_t, |\boldsymbol{Z}_{\infty}| > 0) \le \sum_{j \in \{1,k\}}\frac{\sup_{i \in S}{\Var((W^{(i)})^j)}}{|\boldsymbol{Z}_t| \epsilon^2(1 - \sup_{i \in S}q_i)}.  \label{complement}
\end{equation}
Hence,
    \begin{align}
    % minimum, maximum replacement
    \mathbb{E}\Bigg(&\frac{\sup\limits_{i \in S}\mathbb{E}((W^{(i)})^k)+\epsilon}
    % Denominator
    {\left(|\boldsymbol{Z}_t|^{\frac{k-1}{k}}(\inf\limits_{i \in S}\mathbb{E}(W^{(i)})-\epsilon)\right)^k}\bigm| |\boldsymbol{Z}_{\infty}| > 0\Bigg)+ 
    % Expectation
    \mathbb{E}\left(\mathbb{P}(A^C_t\bigm| |\boldsymbol{Z}_{\infty}| > 0, \boldsymbol{Z}_t)\bigm| |\boldsymbol{Z}_{\infty}| > 0\right) \notag\\
        % Probability unlikely event replacement
    \le \mathbb{E}\Bigg(&\frac{\sup\limits_{i \in S}\mathbb{E}((W^{(i)})^k)+\epsilon}
    % Denominator
    {|\boldsymbol{Z}_t|^{k-1}(\inf\limits_{i \in S}\mathbb{E}(W^{(i)})-\epsilon)^k}\bigm| |\boldsymbol{Z}_{\infty}| > 0\Bigg)
    % unlikely event bit
    + \mathbb{E}\left(\sum_{j \in \{1, k\}}\frac{\sup\limits_{i \in S}{\Var((W^{(i)})^j)}}{|\boldsymbol{Z}_t| \epsilon^2(1 - \sup\limits_{i \in S}q_i)}\bigm| |\boldsymbol{Z}_{\infty}| > 0\right) \notag  \\
    % Simplifying
    = \mathbb{E}\Bigg(&\frac{1}{|\boldsymbol{Z}_t|^{k-1}}\bigm||\boldsymbol{Z}_{\infty}| > 0\Bigg)\left(\frac{\sup_{i \in S}\mathbb{E}((W^{(i)})^k)+\epsilon}
    % Denominator
    {(\inf_{i \in S}\mathbb{E}(W^{(i)})-\epsilon)^k}\right) \notag\\
    &+ \mathbb{E}\Bigg(\frac{1}{|\boldsymbol{Z}_t|}\bigm||\boldsymbol{Z}_{\infty}| > 0\Bigg) \left(
    % bad event bit
    \sum_{j \in \{1,k\}}\frac{\sup_{i \in S}{\Var((W^{(i)})^j)}}{ \epsilon^2(1 - \sup_{i \in S}q_i)}\right).
 \label{last_with_power}
\end{align}
To simplify things further, we can relate the $(k-1)$th harmonic moment of $|\boldsymbol{Z}_t|$ to its first harmonic moment. For $k \geq 2$:
\begin{equation}
    \mathbb{E}\Bigg(\frac{1}{|\boldsymbol{Z}_t|}\bigm||\boldsymbol{Z}_{\infty}| > 0\Bigg) \ge  \mathbb{E}\Bigg(\frac{1}{|\boldsymbol{Z}_t|^{k-1}}\bigm||\boldsymbol{Z}_{\infty}| > 0\Bigg), \label{first_and_kth}
\end{equation}
leading to the simpler bound
%Equation (\ref{first_and_kth}) allows us to simplify (\ref{last_with_power}) as follows
\begin{align}
    \mathbb{E}\Bigg(&\frac{1}{|\boldsymbol{Z}_t|^{k-1}}\bigm||\boldsymbol{Z}_{\infty}| > 0\Bigg)\left(\frac{\sup_{i \in S}\mathbb{E}((W^{(i)})^k)+\epsilon}
    % Denominator
    {(\inf_{i \in S}\mathbb{E}(W^{(i)})-\epsilon)^k}\right) \notag\\
    &+ \mathbb{E}\Bigg(\frac{1}{|\boldsymbol{Z}_t|}\bigm||\boldsymbol{Z}_{\infty}| > 0\Bigg) \left(
    % bad event bit
    \sum_{j \in \{1,k\}}\frac{\sup_{i \in S}{\Var((W^{(i)})^j)}}{ \epsilon^2(1 - \sup_{i \in S}q_i)}\right) \notag\\
    \le \mathbb{E}\Bigg(&\frac{1}{|\boldsymbol{Z}_t|}\bigm||\boldsymbol{Z}_{\infty}| > 0\Bigg)\left(\frac{\sup_{i \in S}\mathbb{E}((W^{(i)})^k)+\epsilon}
    % Denominator
    {(\inf_{i \in S}\mathbb{E}(W^{(i)})-\epsilon)^k}
    % bad event bit
   + \sum_{j \in \{1,k\}}\frac{\sup_{i \in S}{\Var((W^{(i)})^j)}}{ \epsilon^2(1 - \sup_{i \in S}q_i)}\right). \label{eq:final_bound}
\end{align}
Combining \eqref{equal_expression}, \eqref{beforeLLN}, \eqref{before_bound}, \eqref{min_max}, \eqref{last_with_power}, and \eqref{eq:final_bound} leads to the right-hand inequality in the statement of Theorem \ref{harmonic_thm}.

The proof for the lower bound is analogous. In particular,
\begin{align}
  \mathbb{E}\Bigg(\mathbb{E}&\bigg(\frac{\V{k} }
    % denominator
    {(|\boldsymbol{Z}_t|^{\frac{k-1}{k}}\V{1})^k} \mathbbm{1}_{\A} + % Other event
    \frac{\V{k}}
    % denominator
    {(|\boldsymbol{Z}_t|^{\frac{k-1}{k}}\V{1})^k} \mathbbm{1}_{\A[C]}
    %conditioning
    \bigm| |\boldsymbol{Z}_{\infty}| > 0,\boldsymbol{Z}_t  \bigg)\bigm| |\boldsymbol{Z}_{\infty}| > 0\Bigg) \notag\\
    % first inequality
 &\ge \mathbb{E}\Bigg(\mathbb{E}\bigg(\frac{\V{k}}
    % denominator
    {(|\boldsymbol{Z}_t|^{\frac{k-1}{k}}\V{1})^k} \mathbbm{1}_{\A} \bigm| |\boldsymbol{Z}_{\infty}| > 0,\boldsymbol{Z}_t  \bigg)\bigm| |\boldsymbol{Z}_{\infty}| > 0\Bigg) \notag\\
    % After epsilon 
    &\ge \mathbb{E}\left(\mathbb{E}\left(\frac{\mathbb{E}\left(\V{k}\bigm|  \boldsymbol{Z}_t\right)-\epsilon}
    % Denominator
    {\left(|\boldsymbol{Z}_t|^{\frac{k-1}{k}}(\mathbb{E}\left(\V{1}\bigm|  \boldsymbol{Z}_t\right)+\epsilon)\right)^k}\mathbbm{1}_{\A}\bigm| |\boldsymbol{Z}_{\infty}| > 0,\boldsymbol{Z}_t  \right)\bigm| |\boldsymbol{Z}_{\infty}| > 0\right) \label{epsilon}\\
    % probability instead of expected value
    &=\mathbb{E}\left(\frac{\mathbb{E}\left(\V{k}\bigm|  \boldsymbol{Z}_t\right)-\epsilon}
    % Denominator
    {\left(|\boldsymbol{Z}_t|^{\frac{k-1}{k}}(\mathbb{E}\left(\V{1}\bigm|  \boldsymbol{Z}_t\right)+\epsilon)\right)^k}\mathbb{P}\left(\A\bigm| |\boldsymbol{Z}_{\infty}| > 0,\boldsymbol{Z}_t  \right)\bigm| |\boldsymbol{Z}_{\infty}| > 0\right), \notag
    \end{align}
where (\ref{epsilon}) holds for $\epsilon \in (0, \inf_{i \in S} \mathbb{E}((W^{(i)})^k))$.
    By subtracting \eqref{complement} from $1$, we obtain the lower bound %for the probability of the event $\A$ 
    \begin{equation}
        \mathbb{P}(\A|\boldsymbol{Z}_t, |\boldsymbol{Z}_{\infty}| > 0) \ge 1 -  \sum_{j \in \{1,k\}}\frac{\sup_{i \in S}{\Var((W^{(i)})^j)}}{|\boldsymbol{Z}_t| \epsilon^2(1 - \sup_{i \in S}q_i)}.
    \end{equation}
    In addition, we lower bound $\sum_{i \in S}\z \mathbb{E}((W^{(i)})^k)$ by $\sum_{i \in S}\z \inf_{i \in S}\mathbb{E}((W^{(i)})^k)$ and upper bound $\sum_{i \in S}\z \mathbb{E}(W^{(i)})$ by $\sum_{i \in S}\z \sup_{i \in S}\mathbb{E}(W^{(i)})$ to obtain
    \begin{align*}
    &\mathbb{E}\left(\frac{\mathbb{E}\left(\V{k}\bigm|  \boldsymbol{Z}_t\right)-\epsilon}
    % Denominator
    {\left(|\boldsymbol{Z}_t|^{\frac{k-1}{k}}(\mathbb{E}\left(\V{1}\bigm|  \boldsymbol{Z}_t\right)+\epsilon)\right)^k}\mathbb{P}\left(\A\bigm| |\boldsymbol{Z}_{\infty}| > 0,\boldsymbol{Z}_t  \right)\bigm| |\boldsymbol{Z}_{\infty}| > 0\right) \\
    &\ge \mathbb{E}\left(\frac{\inf_{i \in S}\mathbb{E}((W^{(i)})^k)-\epsilon}
    % Denominator
    {\left(|\boldsymbol{Z}_t|^{\frac{k-1}{k}}(\sup_{i \in S}\mathbb{E}(W^{(i)})+\epsilon)\right)^k}\left(1 -  \sum_{j \in \{1,k\}}\frac{\sup_{i \in S}{\Var((W^{(i)})^j)}}{|\boldsymbol{Z}_t| \epsilon^2(1 - \sup_{i \in S}q_i)}\right)\bigm| |\boldsymbol{Z}_{\infty}| > 0\right) \\
    &= \frac{\inf_{i \in S}\mathbb{E}((W^{(i)})^k)-\epsilon}{(\sup_{i \in S}\mathbb{E}(W^{(i)})+\epsilon)^k}\Bigg(\mathbb{E}\left(\frac{1}
    % Denominator
    {|\boldsymbol{Z}_t|^{k-1}} \bigm| |\boldsymbol{Z}_{\infty}| > 0\right) \\
    & \phantom{{}= \frac{\inf_{i \in S}\mathbb{E}((W^{(i)})^k)-\epsilon}{(\sup_{i \in S}\mathbb{E}(W^{(i)})+\epsilon)^k}} {}-  \sum_{j \in \{1,k\}}\frac{\sup_{i \in S}{\Var((W^{(i)})^j)}}{ \epsilon^2(1 - \sup_{i \in S}q_i)}\mathbb{E}\left(\frac{1}{|\boldsymbol{Z}_t|^k}\bigm| |\boldsymbol{Z}_{\infty}| > 0\right)\Bigg).
\end{align*}
\section{Proof of Theorem \ref{hs_thm}} \label{hs_section}
\subsection{Multi-type Harris--Sevastyanov transformation}
In Theorem \ref{harmonic_thm}, the coalescence time distribution depends on the harmonic moments of $|\boldsymbol{Z}_t|$. For large $t$, the computation of harmonic moments of the total population size is likely to be unfeasible. This motivates the derivation of a relationship between the harmonic moments of $|\boldsymbol{Z}_t|$ and those of the Harris--Sevastyanov transformation of the process at time $1$, which can be more easily computed numerically. 

In the rest of this section, for any $r \in \mathbb{N}_+$, the proof of Theorem \ref{hs_thm} proceeds as follows:
    \begin{enumerate}
    \item In Lemma \ref{t_and_infinity}, we relate the harmonic mean of $|\boldsymbol{Z}_t|^r$ conditioned on $|\boldsymbol{Z}_{\infty}| >0$ to the same quantity conditioned on $|\boldsymbol{Z}_{t}| >0$. 
    \item In Lemma \ref{generating}, we express the harmonic mean of $|\boldsymbol{Z}_t|^r$ in terms of the generating functions, $\boldsymbol{f}(\boldsymbol{s})$, of the process.
    \item In Lemma \ref{z_and_y_upper}, we upper bound the harmonic mean of $|\boldsymbol{Z}_t|^r$ in terms of the harmonic mean of  $|\boldsymbol{Y}_t|^r$ using Lemma \ref{generating}, where we recall that $\boldsymbol{Y}$ is the Harris--Sevastyanov transformed process. Analogously, in Lemma \ref{z_and_y_lower}, we will lower bound the harmonic mean of $|\boldsymbol{Z}_t|^r$ in terms of the reciprocal of the expected value of $|\boldsymbol{Y}_t|^{r}$.
    \item In Lemmata \ref{y_upper} and \ref{y_lower}, we derive the relationship between $|\boldsymbol{Y}_t|$ and $|\boldsymbol{Y}_1|$ .
\end{enumerate}
\begin{remark}
\cite[Inequalities (22)--(25)]{heyde1971} prove that, in the single type case, for the process $(Y_t)$ obtained by the Harris--Sevastyanov transformation:
\begin{equation*}
    \mathbb{E}\left(\frac{1}{Y_t}\right) \le \left(\mathbb{E}\left(\frac{1}{Y_1}\right)\right)^t.
\end{equation*}
We generalise this result to the multi-type case and the $r$th harmonic moment in Lemma \ref{y_upper}.
\end{remark}
\subsection{Conditioning on non-extinction at generation t}
First we relate the harmonic moment of $|\boldsymbol{Z}_t|$ conditioned on $|\boldsymbol{Z}_t| >0$ to that conditioned on $|\boldsymbol{Z}_{\infty}| > 0$.
\begin{lemma}\label{t_and_infinity}
For $r \in \mathbb{N}_+$, under Assumptions \assumptionsHS,
\begin{equation*}
   (1-\sup_{i \in S}q_i) \mathbb{E}\left(\frac{1}{|\boldsymbol{Z}_t|^{r}}\bigm| |\boldsymbol{Z}_{t}|>0\right) \le \mathbb{E}\left(\frac{1}{|\boldsymbol{Z}_t|^{r}}\bigm| |\boldsymbol{Z}_{\infty}|>0\right)\le \frac{\mathbb{E}\left(\frac{1}{|\boldsymbol{Z}_t|^r}\bigm| |\boldsymbol{Z}_t| >0\right)}{1 - \sup_{i \in S}q_i}.  \label{infinity_vs_t}
\end{equation*}
\end{lemma}
\begin{proof}
Since $\{|\boldsymbol{Z}_{\infty}| >0\} \subset \{|\boldsymbol{Z}_{t}| >0\} $ up to null sets, we have
\begin{equation*}
    \mathbb{E}(\cdot\bigm| |\boldsymbol{Z}_{\infty}| >0) = \mathbb{E}(\cdot\bigm| |\boldsymbol{Z}_{\infty}| >0, |\boldsymbol{Z}_{t}| >0).
\end{equation*}
Let us condition the $r$th harmonic moment of $|\boldsymbol{Z}_t|$ additionally on non-extinction at time $t$:
    \begin{align}
        \mathbb{E}\left(\frac{1}{|\boldsymbol{Z}_t|^r}\bigm| |\boldsymbol{Z}_{\infty}| >0 \right) %\nonumber \\
        & = \mathbb{E}\left(\frac{1}{|\boldsymbol{Z}_t|^r}\bigm| |\boldsymbol{Z}_{\infty}| >0, |\boldsymbol{Z}_{t}| >0 \right)\nonumber  \\
        &=  \frac{1}{\mathbb{P}(|\boldsymbol{Z}_{\infty}| > 0 \bigm| |\boldsymbol{Z}_t |>0)}\mathbb{E}\left(\frac{1}{|\boldsymbol{Z}_t|^r} \mathbbm{1}(|\boldsymbol{Z}_{\infty}| >0)\bigm| |\boldsymbol{Z}_{t}| >0 \right)\label{eq:hmequality}  \\
        &\le\frac{1}{\mathbb{P}(|\boldsymbol{Z}_{\infty}| >0\bigm| |\boldsymbol{Z}_{t}| >0 )} \mathbb{E}\left(\frac{1}{|\boldsymbol{Z}_t|^r}\bigm| |\boldsymbol{Z}_{t}| >0 \right). \label{eq:hard_prob}
    \end{align}
%where \eqref{eq:hard_prob} follows from upper bounding the indicator function by 1. 
%Let us explore the first factor of \eqref{eq:hard_prob}, $\mathbb{P}(|\boldsymbol{Z}_{\infty}| >0\bigm| |\boldsymbol{Z}_{t}| >0 )$, further. Note that 
Further note that
% \begin{align}
%     &\mathbb{P}(|\boldsymbol{Z}_{\infty}| >0\bigm| |\boldsymbol{Z}_{t}| >0 ) \nonumber \\
%     & \sum_{l=0}^{\infty} \mathbb{P}(|\boldsymbol{Z}_t|=l \bigm| |\boldsymbol{Z}_t| > 0 ) \mathbb{P}(|\boldsymbol{Z}_{\infty}| > 0 \bigm| |\boldsymbol{Z}_t| = l, |\boldsymbol{Z}_t| > 0)  \nonumber\\
%     &=\sum_{l=0}^{\infty} \mathbb{P}(|\boldsymbol{Z}_t|=l \bigm| |\boldsymbol{Z}_t| > 0 ) (1 - \mathbb{P}(|\boldsymbol{Z}_{\infty}| > 0 \bigm| |\boldsymbol{Z}_t| = l, |\boldsymbol{Z}_t| > 0) )\\
%     &=\sum_{l=0}^{\infty} \mathbb{P}(|\boldsymbol{Z}_t|=l \bigm| |\boldsymbol{Z}_t| > 0 ) (1 - \mathbb{P}(|\boldsymbol{Z}_{\infty}| > 0 \bigm| |\boldsymbol{Z}_t| = 1, |\boldsymbol{Z}_t| > 0)^l )  \nonumber\\
%     &\ge\sum_{l=0}^{\infty} \mathbb{P}(|\boldsymbol{Z}_t|=l \bigm| |\boldsymbol{Z}_t| > 0 ) (1 - \sup_{i \in S} q_i^l )  \nonumber \\
%      &\ge\sum_{l=0}^{\infty} \mathbb{P}(|\boldsymbol{Z}_t|=l \bigm| |\boldsymbol{Z}_t| > 0 ) (1 - \sup_{i \in S} q_i )  \nonumber\\
%     &= 1 - \sup_{i \in S}q_i \label{less_than_hard_prob}  \nonumber
% \end{align}
\begin{equation}
 \mathbb{P}(|\boldsymbol{Z}_{\infty}| >0\bigm| |\boldsymbol{Z}_{t}| >0 ) = 1 - \mathbb{P}(|\boldsymbol{Z}_{\infty}| = 0\bigm| |\boldsymbol{Z}_{t}| >0 ) \ge 1 - \sup_{i \in S}q_i \label{less_than_hard_prob} 
 \end{equation}
since, for any configuration $\boldsymbol{Z}_t$, the probability of extinction is the product of at least one $q_i$ and every such product is smaller than the single largest extinction probability. Now combining (\ref{eq:hard_prob}) with (\ref{less_than_hard_prob}), we conclude that
    \begin{equation*}
        \mathbb{E}\left(\frac{1}{|\boldsymbol{Z}_t|^r}\bigm| |\boldsymbol{Z}_{\infty}| >0 \right) \le \frac{1}{1 -\sup_{i \in S}q_i} \mathbb{E}\left(\frac{1}{|\boldsymbol{Z}_t|^r }\bigm| |\boldsymbol{Z}_t| >0\right),
    \end{equation*}
    which is the upper bound in the statement of the lemma, which is well defined under Assumption~\ref{ass:extinct}.
    The proof for the lower bound is analogous. From \eqref{eq:hmequality} and the tower law:
  \begin{align*}
     \mathbb{E}\left(\frac{1}{|\boldsymbol{Z}_t|^{r}}\bigm| |\boldsymbol{Z}_{\infty}| >0 \right) %\nonumber\\
%        & = \mathbb{E}\left(\frac{1}{|\boldsymbol{Z}_t|^{r}}\bigm| |\boldsymbol{Z}_{\infty}| >0, |\boldsymbol{Z}_{t}| >0 \right)\nonumber  \\
 %       &=\frac{\mathbb{E}\left(\frac{1}{|\boldsymbol{Z}_t|^{r}}\mathbbm{1}_{\{|\boldsymbol{Z}_\infty| > 0\}}\bigm| |\boldsymbol{Z}_{t}| >0 \right)}{\mathbb{P}( |\boldsymbol{Z}_{\infty}| >0\bigm| |\boldsymbol{Z}_{t}| >0 )} \nonumber\\
        &= \frac{\mathbb{E}\left(\frac{1}{|\boldsymbol{Z}_t|^{r}}\mathbb{P}(|\boldsymbol{Z}_\infty| > 0\bigm| \boldsymbol{Z}_{t},|\boldsymbol{Z}_{t}| >0)\bigm| |\boldsymbol{Z}_{t}| >0 \right)}{\mathbb{P}( |\boldsymbol{Z}_{\infty}| >0\bigm| |\boldsymbol{Z}_{t}| >0 )}. \label{still_equal} 
        \end{align*}
  To lower bound $\mathbb{P}( |\boldsymbol{Z}_{\infty}| >0\bigm| \boldsymbol{Z}_t, |\boldsymbol{Z}_{t}| >0)$ we note that
        \begin{equation*}
            1 \ge \mathbb{P}( |\boldsymbol{Z}_{\infty}| >0\bigm| \boldsymbol{Z}_t,|\boldsymbol{Z}_t| > 0) = 1 - \mathbb{P}( |\boldsymbol{Z}_{\infty}| =0\bigm| \boldsymbol{Z}_t,|\boldsymbol{Z}_t| > 0) \ge 1-\sup_{i \in S}q_i.
        \end{equation*}
        Hence,
        \begin{align*}
     \mathbb{E}\left(\frac{1}{|\boldsymbol{Z}_t|^{r}}\bigm| |\boldsymbol{Z}_{\infty}| >0 \right) & \ge (1-\sup_{i \in S}q_i)\mathbb{E}\left(\frac{1}{|\boldsymbol{Z}_t|^{r}}\bigm| |\boldsymbol{Z}_{t}|>0\right).
\end{align*}

\end{proof}
\subsection{Population size after the Harris--Sevastyanov transformation}
First, let us relate the expectation of the total population size to its generating function. For $r \in \mathbb{N}_+$, consider the generating functions for the $r$th moments of entries of $\boldsymbol{Z}_t^{i_0}$ and $\boldsymbol{Y}_t^{i_0}$, denoted respectively by
\begin{equation*}
    f_{t,r}^{i_0}(\boldsymbol{s}) := \sum_{\substack{\boldsymbol{v} \in \mathbb{N}_0^{d} {}}} \mathbb{P}(\boldsymbol{Z}_t = \boldsymbol{v}) \prod_{i \in S} s_i^{(v_i^r)} \qquad \text{ and } 
    \qquad F_{t,r}^{i_0}(\boldsymbol{s}) := \sum_{\substack{\boldsymbol{v} \in \mathbb{N}_0^{d} {}}} \mathbb{P}(\boldsymbol{Y}_t = \boldsymbol{v}) \prod_{i \in S} s_i^{(v_i^r)}.
\end{equation*}
\begin{remark}
    With regard to the extinction probability, similarly to the generating functions $\boldsymbol{f}(\boldsymbol{s})$ of $\boldsymbol{Z}_1$ (see Proposition \ref{prop:compute_extinct}), we observe that for any $r \in \mathbb{N}_+$
    \begin{equation} \label{eq:fixed_point}
        f_{t,r}^{i_0}(\boldsymbol{q}) = q_{i_0}.
    \end{equation}
    The proof can be found in \ref{f_and_q}.
\end{remark}
\begin{lemma} \label{generating}
% Under Assumptions \ref{infinite_matrix}, \ref{R-positive} and \ref{moments}, f
For $r \in \mathbb{N}_+$, it holds that
\begin{equation*}
 \mathbb{E}\left(\frac{1}{|\boldsymbol{Z}_t|^r}\bigm| |\boldsymbol{Z}_t| > 0\right) = \int_0^1 \frac{f_{t, r}^{i_0}(s\boldsymbol{1})-f_t^{i_0}(\boldsymbol{0})}{1-f_t^{i_0}(\boldsymbol{0})} \frac{1}{s}ds.
\end{equation*}
\end{lemma}
\begin{proof}
For $r \in \mathbb{N}_+$, observe that the $r$th harmonic moment of $|\boldsymbol{Z}_t|$ can be expressed in terms of an integral of its generating functions:
\begin{align}
  \mathbb{E}\left(\frac{1}{|\boldsymbol{Z}_t|^r}\bigm| |\boldsymbol{Z}_{t}| > 0\right)%\nonumber \\
  &= \sum_{v =1}^{\infty} \mathbb{P}(|\boldsymbol{Z}_t| = v\bigm| |\boldsymbol{Z}_{t}| > 0) \frac{1}{v^r}\nonumber \\
  &=    \sum_{v =1}^{\infty} \mathbb{P}(|\boldsymbol{Z}_t| = v\bigm| |\boldsymbol{Z}_{t}| > 0) \int_0^1s^{v^r-1}  ds\nonumber \\
  &= \int_0^1   \sum_{v =1}^{\infty}  \frac{\mathbb{P}(|\boldsymbol{Z}_t| = v, |\boldsymbol{Z}_{t}| > 0)}{\mathbb{P}(|\boldsymbol{Z}_{t}| > 0)} s^{v^r} \frac{1}{s}  ds \label{eq:int_exchange}\\
    & =\int_0^1 \frac{1}{1 - \mathbb{P}(|\boldsymbol{Z}_{t}| = 0)} \frac{1}{s}\left( \left(\sum_{v =0}^{\infty} \mathbb{P}(|\boldsymbol{Z}_t| = v) s^{v^r} \right) - \mathbb{P}(|\boldsymbol{Z}_{t}| = 0)\right) ds \nonumber\\
  &= \int_0^1 \frac{f_{t, r}^{i_0}(s\boldsymbol{1})-f_t^{i_0}(\boldsymbol{0})}{1-f_t^{i_0}(\boldsymbol{0})} \frac{1}{s} ds, \label{eq:definition}
\end{align}
where in \eqref{eq:int_exchange} the sum and integral can be exchanged due to Tonelli's theorem (e.g., \cite[Theorem 1.29, p.\ 25]{kallenberg}) %\cite[Chapter 2]{tonelli}
as all quantities in the expression are positive, and \eqref{eq:definition} follows from the definition of $f_{t,r}^{i_0}$ and $f_t^{i_0}$. 
\end{proof}
Now Lemma \ref{generating} will be an important tool in upper bounding the harmonic moment of $|\boldsymbol{Z}_t|^r$ by that of $|\boldsymbol{Y}_t|^r$.
\begin{lemma} \label{z_and_y_upper}
Under Assumptions \ref{infinite_matrix}, \ref{R-positive}, and \ref{ass:extinct}, for $r \in \mathbb{N}_+$ it holds that
\begin{equation*}
  \mathbb{E}\left(\frac{1}{|\boldsymbol{Z}_t|^r}\bigm| |\boldsymbol{Z}_t| > 0\right) \le \frac{1-q_{i_0}}{    (1-\sup_{i \in S}q_i)(1-f_t^{i_0}(\boldsymbol{0}))  }\mathbb{E}\left(\frac{1}{|\boldsymbol{Y}_t|^r}\right).
   \end{equation*}
\end{lemma}
\begin{proof}
  To compare the generating functions of $|\boldsymbol{Z}_t|^r$ and $|\boldsymbol{Y}_t|^r$, for each $t \in \mathbb{N}_+$ we define $U_{t,r}^{i_0}:[0,1]\to\mathbb{R}$ as %for any $r \in \mathbb{N}_+$
  %\paul{Should we specify $U_{t,r}^{i_0}:[0,1]\to\mathbb{R}$?}
\begin{equation*}
    U^{i_0}_{t, r}(s) := f_{t, r}^{i_0}( s(\boldsymbol{1}-\boldsymbol{q}) +\boldsymbol{q}) - (1-\sup_{i \in S}q_i)f_{t,r}^{i_0}(s\boldsymbol{1}).
\end{equation*}
Observe that 
\begin{equation*}
    U^{i_0}_{t,r}(0) = f_{t,r}^{i_0}(\boldsymbol{q}) -(1-\sup_{i \in S}q_i) f_{t}^{i_0}(\boldsymbol{0}) = q_{i_0} - (1-\sup_{i \in S}q_i) f_{t}^{i_0}(\boldsymbol{0}).
\end{equation*}
The derivative of $U^{i_0}_{t,r}(s)$ with respect to $s$ can be expressed in terms of a sum of partial derivatives of $f_{t,r}^{i_0}(\boldsymbol{s})$. Note that for all $i \in S$,%the partial derivative of $f_{t,r}^{i_0}(\boldsymbol{s})$ with respect to $s_i$ is non-negative 
 \begin{equation*}
     \frac{\partial f_{t,r}^{i_0}}{\partial_i}(\boldsymbol{s})=\sum_{\substack{\boldsymbol{w} \in \mathbb{N}_0^d{}}}
     % Probability
     \left(\mathbb{P}(\boldsymbol{Z}_t=\boldsymbol{w}) w_i^r \frac{1}{s_i} \prod_{j \in S} s_j^{(w_j^r)}\right) \ge 0,
 \end{equation*} 
 where $\partial f/\partial_i$ denotes the partial derivative of a function $f$ with respect to its $i$th argument. ${\partial f_{t,r}^{i_0}}/{\partial_i}$ is also non-decreasing in each of its arguments, which can be verified by considering the second partial derivative
 \begin{equation*}
     \frac{\partial^2 f_{t,r}^{i_0}}{\partial_i^2}(\boldsymbol{s})=\sum_{\substack{\boldsymbol{w} \in \mathbb{N}_0^d{}}}
     % Probability
     \left(\mathbb{P}(\boldsymbol{Z}_t=\boldsymbol{w}) w_i^r (w_i^r - 1) \frac{1}{s_i^2} \prod_{j \in S} s_j^{(w_j^r)}\right) \ge 0.
 \end{equation*}
 Now define $\boldsymbol{x}:[0,1]^d\to[0,1]^d$ by $\boldsymbol{s} \mapsto \boldsymbol{x}(\boldsymbol{s}) = \boldsymbol{s} \odot (\boldsymbol{1}-\boldsymbol{q}) + \boldsymbol{q}$. By the chain rule, 
\begin{align}
 \frac{dU_{t,r}^{i_0}(s)}{ds} &= (\boldsymbol{1} - \boldsymbol{q}) \cdot \nabla f_{t,r}^{i_0} (\boldsymbol{x}(s\boldsymbol{1})) - (1-\sup_{j \in S} q_j) (\nabla \cdot f_{t,r}^{i_0})(s\boldsymbol 1) \label{eq:sup_qi}\\
 &\geq (\boldsymbol{1} - \boldsymbol{q}) \cdot \nabla f_{t,r}^{i_0} (\boldsymbol{x}(s\boldsymbol{1})) - (\boldsymbol{1} - \boldsymbol{q}) \cdot \nabla f_{t,r}^{i_0} (s\boldsymbol{1})\label{eq:expect_Z}\\
 &= (\boldsymbol{1} - \boldsymbol{q}) \cdot (\nabla f_{t,r}^{i_0} (\boldsymbol{x}(s\boldsymbol{1})) -  \nabla f_{t,r}^{i_0} (s\boldsymbol{1}))\notag%\label{eq:derivatives_sum}
\end{align}
%Now define $x_i:=s_i(1-q_i)+q_i$ for each $i \in S$ and let $\boldsymbol{x} := (x_i)_{i \in S}$. By the multivariate chain rule, 
 %Consider the first derivative of $U_{t,r}^{i_0}(s)$ with respect to $s$
% \begin{align}
%    \frac{dU_{t,r}^{i_0}(s)}{ds}
%    &=\sum_{i \in S} (1-q_i)\frac{\partial f_{t,r}^{i_0}(\boldsymbol{x})}{\partial x_i} - \sum_{i \in S}(1-\sup_{j \in S}q_j) \left.\frac{\partial f_{t,r}^{i_0}(\boldsymbol{s})}{\partial s_i}\right|_{\boldsymbol{s}=s\boldsymbol{1}}\label{eq:sup_qi} \\
%    &\ge \sum_{i \in S} (1-q_i)\frac{\partial f_{t,r}^{i_0}(\boldsymbol{x})}{\partial x_i} - \sum_{i \in S} (1-q_i)\left.\frac{\partial f_{t,r}^{i_0}(\boldsymbol{s})}{\partial s_i}\right|_{\boldsymbol{s}=s\boldsymbol{1}} \label{eq:expect_Z}\\
%    &=\sum_{i \in S} (1-q_i)\left[\frac{\partial f_{t,r}^{i_0}(\boldsymbol{x})}{\partial x_i} -\frac{\partial f_{t,r}^{i_0}(\boldsymbol{s})}{\partial s_i}\right]_{\boldsymbol{s}=s\boldsymbol{1}}. \label{eq:derivatives_sum}
% \end{align}
where \eqref{eq:expect_Z} comes from noting that each entry of $\nabla f_{t,r}^{i_0} (s\boldsymbol{1})$ %partial derivative in the second sum
is non-negative and that, for each $i\in S$, $1-\sup_{j \in S}q_j \le 1-q_i$.

Since we have that $x_i = s_i(1-q_{i}) +q_{i} = s_i + q_{i}(1-s_i)\ge s_i$ for $s_i \in [0,1]$, we can use the monotonicity of the derivatives to compare $f_{t,r}^{i_0}(\boldsymbol{x}(\boldsymbol{s}))$ and $f_{t,r}^{i_0}(\boldsymbol{s})$. For each $i \in S$,
 \begin{equation}
     \nabla_i f_{t,r}^{i_0}(\boldsymbol{x}(\boldsymbol{s})) \geq 
     \nabla_i f_{t,r}^{i_0}(\boldsymbol{s}).\label{eq:derivatives_comparison}
 \end{equation}
% \begin{equation}
%     \left.\frac{\partial f_{t,r}^{i_0}(\boldsymbol{x})}{\partial x_i}\right|_{\boldsymbol{x}=\boldsymbol{s} \odot (\boldsymbol{1}-\boldsymbol{q}) + \boldsymbol{q}} \ge\frac{\partial f_{t,r}^{i_0}(\boldsymbol{s})}{\partial s_i}  \label{eq:derivatives_comparison}.
% \end{equation}
Now using \eqref{eq:derivatives_comparison}, we can say that
%\begin{align}
    $\frac{dU_{t,r}^{i_0}(s)}{ds} %&\ge \sum_{i \in S} (1-q_i)\left[\frac{\partial f_{t,r}^{i_0}(\boldsymbol{x})}{\partial x_i} -\frac{\partial f_{t,r}^{i_0}(\boldsymbol{s})}{\partial s_i}\right]_{\boldsymbol{s}=s\boldsymbol{1}}
    \ge 0$, %\label{eq:U_increasing}
%\end{align}
i.e.\ that $U_{t,r}^{i_0}(s)$ is increasing in $s$,
%\begin{equation}
    $U_{t,r}^{i_0}(s) \ge U_{t,r}^{i_0}(0)$, and thus:
%\end{equation}
%Expanding \eqref{eq:U_inequality}, we get
 \begin{equation}
     f_{t,r}^{i_0}(\boldsymbol{x}(s\boldsymbol{1})) - (1-\sup_{i \in S}q_i)f_{t,r}^{i_0}(s\boldsymbol{1}) \ge q_{i_0} - (1-\sup_{i \in S}q_i) f_{t}^{i_0}(\boldsymbol{0}). \label{eq:U_expanded}
 \end{equation}
 Rearranging \eqref{eq:U_expanded} and using the definition of $F_{t,r}^{i_0}$, we obtain an upper bound for the generating function of $\boldsymbol{Z}_t$ in terms of the generating function of $\boldsymbol{Y}_t$:
 \begin{equation}
     F_{t,r}^{i_0}(s\boldsymbol{1}) \ge \frac{1-\sup_{i \in S} q_i}{1-q_{i_0}}(f_{t,r}^{i_0}(s\boldsymbol{1})-f_{t}^{i_0}(\boldsymbol{0})). \label{eq:generating_functions}
 \end{equation}
 Finally, using Lemma \ref{generating}, we express the conditional harmonic mean of $|\boldsymbol{Z}_t|^r$ as
 \begin{align}
     \mathbb{E}\left(\frac{1}{|\boldsymbol{Z}_t|^r}\bigm| |\boldsymbol{Z}_t| >0\right) %\nonumber\\
&=    \int_0^1 \frac{f_{t,r}^{i_0}(s\boldsymbol{1})-f_{t}^{i_0}(\boldsymbol{0})}{1-f_{t}^{i_0}(\boldsymbol{0})} \frac{1}{s}ds \nonumber \\
&\le \frac{1-q_{i_0}}{    (1-\sup_{i \in S}q_i)(1-f_{t}^{i_0}(\boldsymbol{0}))  } \int_0^1 \frac{F_t^{i_0}(s\boldsymbol{1})}{s} ds \label{eq:generating_inequality} \\
&= \frac{1-q_{i_0}}{    (1-\sup_{i \in S}q_i)(1-f_{t}^{i_0}(\boldsymbol{0}))  } \int_0^1 \frac{\sum_{w=1}^{\infty} \mathbb{P}(|\boldsymbol{Y}_t|=w) s^{(w^r)}}{s} ds \nonumber \\
&=\frac{1-q_{i_0}}{    (1-\sup_{i \in S}q_i)(1-f_{t}^{i_0}(\boldsymbol{0}))  }  \sum_{w=1}^{\infty} \mathbb{P}(|\boldsymbol{Y}_t|=w) \frac{1}{w^r} \nonumber \\
&=  \frac{1-q_{i_0}}{     (1-\sup_{i \in S}q_i)(1-f_{t}^{i_0}(\boldsymbol{0}))  }\mathbb{E}\left(\frac{1}{|\boldsymbol{Y}_t|^r}\right), \nonumber
 \end{align}
 where \eqref{eq:generating_inequality} comes from \eqref{eq:generating_functions} and is well defined under Assumption~\ref{ass:extinct}. 
\end{proof}
Let us now proceed to a lower bound for the conditional harmonic moment of $|\boldsymbol{Z}_t|^r$ in terms of $\mathbb{E}(|\boldsymbol{Y}_t|^r)$.
\begin{lemma}\label{z_and_y_lower}
Under Assumptions \assumptionsHS, for $r \in \mathbb{N}_+$,
       \begin{equation*}
       \frac{(1-\sup_{i \in S}q_i)^{r}}{\mathbb{E}(|\boldsymbol{Y}_t|^{r})} 
   \le\mathbb{E}\left(\frac{1}{|\boldsymbol{Z}_t|^{r}}\bigm| |\boldsymbol{Z}_{t}| >0\right).
    \end{equation*}
\end{lemma}
\begin{proof}
Firstly, we will relate the expected value of the falling factorials $(|\boldsymbol{Y}_t|)_{r}$ and $(|\boldsymbol{Z}_t|)_{r}$. Note that, with $\boldsymbol{x}$ defined as in the proof of Lemma \ref{z_and_y_upper},
        \begin{align}
            \mathbb{E}((|\boldsymbol{Y}_t|)_{r})
            = \sum_{j_1, \dots,j_r \in S}\frac{\partial^r F_t^{i_0}}{\partial_{j_1} \dots \partial_{j_r}}(\boldsymbol{1})
             &=  \frac{\prod_{l=1}^r 1-q_{j_l}}{1-q_{i_0}}\sum_{j_1,\dots,j_r \in S}\frac{\partial^r f_t^{i_0}}{\partial_{j_1}\dots \partial_{j_r}}(\boldsymbol{x}(\boldsymbol{1}))\label{eq:by_def} \\
            &\ge \frac{(1-\sup_{i \in S}q_i)^r}{1-q_{i_0}}\sum_{j_1,\dots,j_r \in S} \frac{\partial^r f_t^{i_0}}{\partial_{j_1}\dots \partial_{j_r}}(\boldsymbol{x}(\boldsymbol{1}))\nonumber \\
            %&=\frac{(1-\sup_{i \in S}q_i)^r}{1-q_{i_0}}\sum_{j_1,\dots,j_r \in S} \left.\frac{\partial^r f_t^{i_0}(\boldsymbol{x})}{\partial s_{j_1} \dots \partial s_{j_r}}\frac{\partial s_{j_1} \dots \partial s_{j_r}}{\partial x_{j_1} \dots \partial x_{j_r}}\right|_{\boldsymbol{s}=\boldsymbol{1}}\nonumber \\
            %&=\frac{(1-\sup_{i \in S}q_i)^r}{1-q_{i_0}}\sum_{j_1,\dots,j_r \in S} 
            % Possible values of Z
            %\sum_{\boldsymbol{w}\in \mathbb{N}_0^{d}}\left( p_{i_0}(\boldsymbol{w}) \prod_{j \in S}((1-q_j)s_j +q_j)^{w_j}\right) \notag\\
            % The part that simplifies
            %&\phantom{{}=\frac{(1-\sup_{i \in S}q_i)^r}{1-q_{i_0}}\sum_{j_1,\dots,j_r \in S}}\times\left.\frac{\prod_{l=1}^{r}(1-q_{j_l})w_{j_l}}
            % Denominator
            %{\prod_{l=1}^{r}(1-q_{j_l}) ((1-q_{j_l})s_{j_l}+q_{j_l})}\right|_{\boldsymbol{s}=\boldsymbol{1}}\nonumber \\
            %&=\frac{(1-\sup_{i \in S}q_i)^r}{1-q_{i_0}}\sum_{j_1,\dots,j_r \in S}
            % Possible values of Z
            %\sum_{\boldsymbol{w}\in \mathbb{N}_0^{d}}\left( p_{i_0}(\boldsymbol{w})\right)
            % The part that simplifies
            %\prod_{l=1}^{r}w_{j_l}
            %\nonumber \\
             &= \frac{(1-\sup_{i \in S}q_i)^r}{1-q_{i_0}}\sum_{j_1,\dots,j_r \in S} \frac{\partial^r f_t^{i_0}}{\partial_{j_1}\dots \partial_{j_r}}(\boldsymbol{1}) \nonumber \\
             &= \frac{(1-\sup_{i \in S}q_i)^r}{1-q_{i_0}} \mathbb{E}((|\boldsymbol{Z}_t|)_{r}),  \label{expectations}
        \end{align}
where the right of \eqref{eq:by_def} follows from the connection between the expected value of $\boldsymbol{Y}_t$ and $\boldsymbol{Z}_t$ (refer to Lemma \ref{lemma:expected}).
        
In order to relate this inequality involving falling factorials to the $r$th moments of $|\boldsymbol{Y}_t|$ and $|\boldsymbol{Z}_t|$, consider Stirling numbers of the second kind $\begin{Bsmallmatrix}
r \\
j
\end{Bsmallmatrix}$, which describe the number of partitions of $[r]$ into $j$ blocks. It holds \cite[Eq.\ (2.3), p.\ 87]{combinatorics} that 
\begin{equation}
    x^{r} = \sum_{j=0}^{r} \begin{Bmatrix}
r \\
j
\end{Bmatrix} (x)_j \label{falling_factorial}.
\end{equation}
Let us set $x=|\boldsymbol{Y}_t|$ and take the expectation of both sides in (\ref{falling_factorial}) to get
\begin{align}
    \mathbb{E}(|\boldsymbol{Y}_t|^{r}) %\nonumber \\
    = \sum_{j=0}^{r} \begin{Bmatrix}
r \\
j
\end{Bmatrix} \mathbb{E}((|\boldsymbol{Y}_t|)_{j})%\label{applied_def} \\
&\ge \sum_{j=0}^{r} \begin{Bmatrix}
r \\
j
\end{Bmatrix}\frac{(1-\sup_{i \in S}q_i)^j}{1-q_{i_0}} \mathbb{E}((|\boldsymbol{Z}_t|)_{j}) \label{partial_derivatives}\\
&\ge \frac{(1-\sup_{i \in S}q_i)^{r}}{1-q_{i_0}} \sum_{j=0}^{r} \begin{Bmatrix}
r \\
j
\end{Bmatrix}\mathbb{E}((|\boldsymbol{Z}_t|)_{j}) \label{maximum_power} \\
& = \frac{(1-\sup_{i \in S}q_i)^{r}}{1-q_{i_0}} \mathbb{E}(|\boldsymbol{Z}_t|^{r}), \label{k_th_moment}
\end{align}
 where the inequality in \eqref{partial_derivatives} uses the bound derived in \eqref{expectations}, and \eqref{maximum_power} bounds $(1-\sup_{i \in S}q_i)^j$ below by $(1-\sup_{i \in S}q_i)^r$ for each $j \in \{0,1,\dots,r\}$. For $r \in \mathbb{N}_+$:
 % Let us now relate the expectation of $|\boldsymbol{Z}_t|^r$ to the expectation of $|\boldsymbol{Z}_t|^r$ conditioned on non-extinction.
        \begin{equation}
            \mathbb{E}(|\boldsymbol{Z}_t|^r) = (1- \mathbb{P}(|\boldsymbol{Z}_t|=0))\mathbb{E}(|\boldsymbol{Z}_t|^r\bigm| |\boldsymbol{Z}_t| > 0) \ge  (1- q_{i_0})\mathbb{E}(|\boldsymbol{Z}_t|^r\bigm| |\boldsymbol{Z}_t| > 0). \label{EZandEZ0}
        \end{equation}
        Combining \eqref{k_th_moment} and \eqref{EZandEZ0}), we arrive at
        \begin{equation*}
            \mathbb{E}(|\boldsymbol{Y}_t|^{r})   \ge (1-\sup_{i \in S}q_i)^{r} \mathbb{E}(|\boldsymbol{Z}_t|^{r}\bigm| |\boldsymbol{Z}_t| > 0).  \label{denominator}
        \end{equation*}
        Finally, we can apply Jensen's inequality to obtain %relate the $r$th moment of $|\boldsymbol{Y}_t|$ to the corresponding harmonic mean of $|\boldsymbol{Z}_t|$:
    \begin{align*}
        \mathbb{E}\left(\frac{1}{|\boldsymbol{Z}_t|^{r}}\bigm| |\boldsymbol{Z}_t| > 0\right) %\nonumber \\
        & \ge \frac{1}{\mathbb{E}(|\boldsymbol{Z}_t|^{r}\bigm| |\boldsymbol{Z}_t| > 0)}  %\nonumber \\
         \ge \frac{(1-\sup_{i \in S}q_i)^{r}}{\mathbb{E}(|\boldsymbol{Y}_t|^{r})}. \label{just_above}
    \end{align*}
    %where (\ref{just_above}) comes from (\ref{denominator}).
\end{proof}
\subsection{Relationship between harmonic mean at generation 1 and t for the transformed process}

We upper bound the harmonic mean of $|\boldsymbol{Y}_t|$ in terms of the harmonic mean of the same process at time $1$.
\begin{lemma}\label{y_upper}
Under Assumptions \ref{infinite_matrix} and \ref{R-positive}, for $r\in\mathbb{N}_+$, %\paul{Really not allowing $r=1$ here?}
    \begin{equation*}
       \mathbb{E}\left(\frac{1}{|\boldsymbol{Y}_t|^r}\right) \le   \left(\mathbb{E}\left(\sup_{i \in S}\frac{1}{|\boldsymbol{Y}_{1}^{i}|}\right)\right)^{tr} .
    \end{equation*}
\end{lemma}
\begin{proof}
Denote by $\boldsymbol{Y}_{t-1, 1, s}^i$ the transformed branching process started from an ancestor of type $i$ with index $s$ in generation $t-1$, and allowed to evolve for one generation. Using the tower rule, let us condition on the population size at generation $t-1$. For $r \in \mathbb{N}_+$,
\begin{align}
    \mathbb{E}\left(\frac{1}{|\boldsymbol{Y}_t|^r}\right) %\nonumber \\
    %&= \mathbb{E}\left(\mathbb{E}\left(\frac{|\boldsymbol{Y}_{t-1}|^r}{|\boldsymbol{Y}_{t-1}|^r}\frac{1}{|\boldsymbol{Y}_t|^r}\bigm| \boldsymbol{Y}_{t-1}\right)\right) \nonumber \\
    &= \mathbb{E}\left(\frac{1}{|\boldsymbol{Y}_{t-1}|^r}\mathbb{E}\left(\frac{|\boldsymbol{Y}_{t-1}|^r}{|\boldsymbol{Y}_t|^r}\bigm| \boldsymbol{Y}_{t-1}\right)\right) \nonumber \\
    &= \mathbb{E}\left(\frac{1}{|\boldsymbol{Y}_{t-1}|^r}\mathbb{E}\left(\frac{|\boldsymbol{Y}_{t-1}|^r}{(\sum_{i \in S}\sum_{s=1}^{Y_{t-1}^i}|\Y|)^r}\bigm| \boldsymbol{Y}_{t-1}\right)\right) \label{eq:arithmetic_mean} \\
    &= \mathbb{E}\left(\frac{1}{|\boldsymbol{Y}_{t-1}|^r}\mathbb{E}\left(\frac{|\boldsymbol{Y}_{t-1}|^r}{\sum_{i_1 \in S}\sum_{s_1=1}^{Y_{t-1}^{i_1}} \dots \sum_{i_r \in S}\sum_{s_r=1}^{Y_{t-1}^{i_r}}\prod_{j=1}^r|\Y[j]|}\bigm| \boldsymbol{Y}_{t-1}\right)\right), \label{eq:expanded} 
\end{align}
where \eqref{eq:arithmetic_mean} uses a decomposition analogous to \eqref{eq:Zdecomposition}, i.e.\ 
%comes from expressing the full population size in terms of families started at generation $t-1$ and ran for one generation, i.e. 
$
\boldsymbol{Y}_t = \sum_{i \in S}\sum_{s=1}^{Y_{t-1}^i}\Y
$.
%and \eqref{eq:expanded} expands the $r$th power of $\sum_{i \in S}\sum_{s=1}^{Y_{t-1}^i}|\Y|$.
The expression in the inner expectation in \eqref{eq:expanded} is the reciprocal of the arithmetic mean of the collection of $\prod_{j=1}^r|\Y[j]\vert$.
%\begin{equation*}
%    \frac{\sum_{i_1 \in S}\sum_{s_1=1}^{Y_{t-1}^{i_1}} \dots \sum_{i_r \in S}\sum_{s_r=1}^{Y_{t-1}^{i_r}}\prod_{j=1}^r|\Y[j]|}{|\boldsymbol{Y}_{t-1}|^r}
%\end{equation*}
We will use the fact that the arithmetic mean is always greater than or equal to the harmonic mean (and hence the inverse of harmonic mean is greater than or equal to the inverse of the arithmetic mean);
here
\begin{equation}
   \frac{|\boldsymbol{Y}_{t-1}|^r}{\sum_{i_1 \in S}\sum_{s_1=1}^{Y_{t-1}^{i_1}} \dots \sum_{i_r \in S}\sum_{s_r=1}^{Y_{t-1}^{i_r}}\prod_{j=1}^r|\Y[j]|} 
   \le \frac{\sum_{i_1 \in S}\sum_{s_1=1}^{Y_{t-1}^{i_1}} \dots \sum_{i_r \in S}\sum_{s_r=1}^{Y_{t-1}^{i_r}}\frac{1}{\prod_{j=1}^r|\Y[j]|}}{|\boldsymbol{Y}_{t-1}|^r}. \label{eq:harmonic_arithmetic}
\end{equation}
The harmonic mean is defined since we know $|\Y[j]| > 0 $ for any $s_j \in \mathbb{N}_+$ and $i_j \in S$ as by construction of the Harris--Savastyanov transformation the extinction probability is 0.

Let us use inequality \eqref{eq:harmonic_arithmetic} to bound \eqref{eq:expanded}:
\begin{align}
\mathbb{E}\left(\frac{1}{|\boldsymbol{Y}_t|^r}\right)
 %&=\mathbb{E}\left(\frac{1}{|\boldsymbol{Y}_{t-1}|^r}\mathbb{E}\left(\frac{|\boldsymbol{Y}_{t-1}|^r}{\sum_{i_1 \in S}\sum_{s_1=1}^{Y_{t-1}^{i_1}} \dots \sum_{i_r \in S}\sum_{s_r=1}^{Y_{t-1}^{i_r}}\prod_{j=1}^r|\Y[j]|}| \boldsymbol{Y}_{t-1}\right)\right) \nonumber \\
&\le \mathbb{E}\left(\frac{1}{|\boldsymbol{Y}_{t-1}|^r}\mathbb{E}\left(\frac{\sum_{i_1 \in S}\sum_{s_1=1}^{Y_{t-1}^{i_1}} \dots \sum_{i_r \in S}\sum_{s_r=1}^{Y_{t-1}^{i_r}}\frac{1}{\prod_{j=1}^r|\Y[j]|}}{|\boldsymbol{Y}_{t-1}|^r}| \boldsymbol{Y}_{t-1}\right)\right) \nonumber \\
& =\mathbb{E}\left(|\boldsymbol{Y}_{t-1}|^{-2r}\sum_{i_1 \in S}\sum_{s_1=1}^{Y_{t-1}^{i_1}} \dots \sum_{i_r \in S}\sum_{s_r=1}^{Y_{t-1}^{i_r}}\mathbb{E}\left(\frac{1}{\prod_{j=1}^r|\Y[j]|} \mid \boldsymbol{Y}_{t-1}\right)\right) \label{eq:tonelli}\\ 
& =\mathbb{E}\left(|\boldsymbol{Y}_{t-1}|^{-2r}\sum_{i_1 \in S}\sum_{s_1=1}^{Y_{t-1}^{i_1}} \dots \sum_{i_r \in S}\sum_{s_r=1}^{Y_{t-1}^{i_r}}\prod_{j=1}^r\mathbb{E}\left(|\Y[j]|^{-1} \mid \boldsymbol{Y}_{t-1}\right)\right)\label{eq:independenceY}
\end{align}%
where \eqref{eq:tonelli} follows by Tonelli's Theorem and \eqref{eq:independenceY} is due to the independence of the branching processes started from each individual at generation $t-1$. Observe that the branching processes started at generation $t-1$, $\{\Y[j]\}_{j \in [r]}$, are independent of the population size at time $t-1$ and thus the conditional expectation in \eqref{eq:independenceY} is equal to the unconditional expectation:
\begin{align}
&\mathbb{E}\left(|\boldsymbol{Y}_{t-1}|^{-2r}\sum_{i_1 \in S}\sum_{s_1=1}^{Y_{t-1}^{i_1}} \dots \sum_{i_r \in S}\sum_{s_r=1}^{Y_{t-1}^{i_r}}\prod_{j=1}^r\mathbb{E}\left(|\Y[j]|^{-1} \mid \boldsymbol{Y}_{t-1}\right)\right) \nonumber\\ 
& =\mathbb{E}\left(|\boldsymbol{Y}_{t-1}|^{-2r}\sum_{i_1 \in S}\sum_{s_1=1}^{Y_{t-1}^{i_1}} \dots \sum_{i_r \in S}\sum_{s_r=1}^{Y_{t-1}^{i_r}}\prod_{j=1}^r\mathbb{E}\left(|\boldsymbol{Y}_{t-1,1,1}^{i_j}|^{-1}\right)\right) \label{eq:drop_index}\\
& \le\mathbb{E}\left(|\boldsymbol{Y}_{t-1}|^{-2r}\sum_{i_1 \in S}\sum_{s_1=1}^{Y_{t-1}^{i_1}} \dots \sum_{i_r \in S}\sum_{s_r=1}^{Y_{t-1}^{i_r}}\left(\mathbb{E}\left(\sup_{j \in [r]}|\boldsymbol{Y}_{t-1,1,1}^{i_j}|^{-1}\right)\right)^r\right), \label{eq:sup} 
\end{align}
where \eqref{eq:drop_index} follows from the fact that for a given type $i_j$, $\mathbb{E}(|\Y[j]|^{-1}) = \mathbb{E}(|\boldsymbol{Y}_{t-1,1,1}^{i_j}|^{-1})$ and \eqref{eq:sup} upper bounds each element of the product by its supremum over all types.

Note that the supremum over types $\{i_j\}_{j=1}^r$ in \eqref{eq:sup} can be upper bounded by the supremum over \emph{all} possible types:
\begin{align}
&\mathbb{E}\left(|\boldsymbol{Y}_{t-1}|^{-2r}\sum_{i_1 \in S}\sum_{s_1=1}^{Y_{t-1}^{i_1}} \dots \sum_{i_r \in S}\sum_{s_r=1}^{Y_{t-1}^{i_r}}\left(\mathbb{E}\left(\sup_{j \in [r]}|\boldsymbol{Y}_{t-1,1,1}^{i_j}|^{-1}\right)\right)^r\right) \nonumber\\
    & \le  \mathbb{E}\left(|\boldsymbol{Y}_{t-1}|^{-2r}\sum_{i_1 \in S}\sum_{s_1=1}^{Y_{t-1}^{i_1}} \dots \sum_{i_r \in S}\sum_{s_r=1}^{Y_{t-1}^{i_r}}\left(\mathbb{E}\left(\sup_{j \in S}|\boldsymbol{Y}^j_{t-1, 1, 1}|^{-1}\right)\right)^r\right)  \nonumber \\
    & = \mathbb{E}\left(|\boldsymbol{Y}_{t-1}|^{-r}\left(\mathbb{E}\left(\sup_{j \in S}|\boldsymbol{Y}^j_{t-1, 1, 1}|^{-1}\right)\right)^r\right) \nonumber\\
    & = \mathbb{E}\left(|\boldsymbol{Y}_{t-1}|^{-r}\right)\left(\mathbb{E}\left(\sup_{j \in S}|\boldsymbol{Y}^j_{t-1, 1, 1}|^{-1}\right)\right)^r = \mathbb{E}\left(|\boldsymbol{Y}_{t-1}|^{-r}\right)\left(\mathbb{E}\left(\sup_{j \in S}|\boldsymbol{Y}_1^j|^{-1}\right)\right)^r  \label{eq:split2}%\label{eq:gent1_vs1}
    \end{align}
where the first equality in \eqref{eq:split2} follows by the independence of the two quantities and the second is due to the equality in law of $|\boldsymbol{Y}^j_{t-1, 1, 1}|^{-1}$ and $|\boldsymbol{Y}_1^j|^{-1}$.

Iterating this procedure over $t$, we obtain the desired inequality. %get an upper bound for the $r$th harmonic moment of $|\boldsymbol{Y}_t|$ that depends only upon the properties of the process at generation $1$ 
%\begin{equation*}
%    \mathbb{E}\left(\frac{1}{|\boldsymbol{Y}_t|^r}\right) \le   \left(\mathbb{E}\left(\sup_{i \in S}\frac{1}{|\boldsymbol{Y}_{1}^{i}|^r}\right)\right)^t.
%\end{equation*}
\end{proof}
For the lower bound of Theorem \ref{hs_thm}, we find an analogous lower bound on $\mathbb{E}(|\boldsymbol{Y}_1|^r)^{-1}$.
\begin{lemma}\label{y_lower}
Under Assumptions \ref{infinite_matrix} and \ref{R-positive}, for 
$r \in \mathbb{N}_+$, 
\begin{equation*}
    \left(\mathbb{E}\left(\sup_{i \in S}|\boldsymbol{Y}_{1}^i|\right)\right)^{-tr}   \le \frac{1}{\mathbb{E}(|\boldsymbol{Y}_t|^r)}.
    \end{equation*}
\end{lemma}
\begin{proof}
This proof is analogous to the upper bound of Lemma~\ref{y_upper}. We use the tower rule to express $|\boldsymbol{Y}_t|^r$ in terms of the offspring of families from time $t-1$:
   \begin{align*}
        \mathbb{E}\left(|\boldsymbol{Y}_t|^r\right) %\\
        &= \mathbb{E}\left(\mathbb{E}\left(|\boldsymbol{Y}_t|^r\bigm|\boldsymbol{Y}_{t-1}\right)\right) \\
        % Y_t is a sum of Y_1
        & = \mathbb{E}\left(\mathbb{E}\left(\left(\sum_{i \in S}\sum_{s=1}^{Y_{t-1}^i} |\Y|\right)^r \bigm|\boldsymbol{Y}_{t-1}\right)\right) \\
        % expanded power
        & = \mathbb{E}\left(\mathbb{E}\left(\sum_{i_1 \in S}\sum_{s_1=1}^{Y_{t-1}^{i_1}} \dots \sum_{i_r \in S}\sum_{s_r=1}^{Y_{t-1}^{i_r}} \prod_{j=1}^r |\boldsymbol{Y}_{t-1, 1,s_j}^{i_j}| \bigm|\boldsymbol{Y}_{t-1}\right)\right) \\
        % moved expectation
  &  =\mathbb{E}\left(\sum_{i_1 \in S}\sum_{s_1=1}^{Y_{t-1}^{i_1}} \dots \sum_{i_r \in S}\sum_{s_r=1}^{Y_{t-1}^{i_r}}\prod_{j=1}^r \mathbb{E}\left( |\boldsymbol{Y}_{t-1, 1,s_j}^{i_j}|\bigm|\boldsymbol{Y}_{t-1} \right)\right) \\
    &  =\mathbb{E}\left(\sum_{i_1 \in S}\sum_{s_1=1}^{Y_{t-1}^{i_1}} \dots \sum_{i_r \in S}\sum_{s_r=1}^{Y_{t-1}^{i_r}}\prod_{j=1}^r \mathbb{E}\left( |\boldsymbol{Y}_{t-1, 1,1}^{i_j}|\bigm|\boldsymbol{Y}_{t-1} \right)\right) \\
  % supremum
   &  \le \mathbb{E}\left(\sum_{i_1 \in S}\sum_{s_1=1}^{Y_{t-1}^{i_1}} \dots \sum_{i_r \in S}\sum_{s_r=1}^{Y_{t-1}^{i_r}}\left(\mathbb{E}\left( \sup_{i \in S}  |\boldsymbol{Y}_{t-1,1,1}^{i}| \bigm|\boldsymbol{Y}_{t-1}\right)\right)^r\right) \\
   % no sums
        & = \mathbb{E}\left(|\boldsymbol{Y}_{t-1}|^r\right)\left(\mathbb{E}\left(\sup_{i \in S}|\boldsymbol{Y}_{t-1,1,1}^i|\right)\right)^r  %\label{eq:independence_t1} \\
        = \mathbb{E}\left(|\boldsymbol{Y}_{t-1}|^r\right)\left(\mathbb{E}\left(\sup_{i \in S}|\boldsymbol{Y}_{1}^i|\right)\right)^r        .
   \end{align*}
    Iterating this procedure, we get
    \begin{equation*}
        \mathbb{E}(|\boldsymbol{Y}_t|^r) \le \left(\mathbb{E}(\sup_{i \in S}|\boldsymbol{Y}_{1}^i|)\right)^{tr},
    \end{equation*}
    and the result follows immediately.
    %% Which is equivalent to 
    %% \begin{equation*}
    %%   \left(\mathbb{E}(\sup_{i \in S}|\boldsymbol{Y}_{1}^i|^r)\right)^{-t}   \le \frac{1}{\mathbb{E}(|\boldsymbol{Y}_t|^r)}.
    %% \end{equation*}
\end{proof}
Finally, combining Lemmata \ref{t_and_infinity}, \ref{z_and_y_upper}, \ref{z_and_y_lower}, \ref{y_upper}, and \ref{y_lower}, we get a lower bound for the $r$th harmonic moment of $|\boldsymbol{Z}_t|$:
\begin{equation*}
    (1-\sup_{i \in S} q_i)^r 
\left(  \mathbb{E}(\sup_{i \in S}|\boldsymbol{Y}_1^i|) \right)^{-tr} \le 
    % middle 
    \mathbb{E}\left(\frac{1}{|\boldsymbol{Z}_t|^{r}}\bigm| |\boldsymbol{Z}_{\infty}| > 0\right),
\end{equation*}
and an upper bound
\begin{equation}
    % right
      \mathbb{E}\left(\frac{1}{|\boldsymbol{Z}_t|^r}\bigm| |\boldsymbol{Z}_{\infty}| > 0\right)  \le  \frac{1-q_{i_0}}{    (1-\sup_{i \in S}q_i)^2(1-f_t^{i_0}(\boldsymbol{0}))  } \left( \mathbb{E}\left(\sup_{i \in S}\frac{1}{|\boldsymbol{Y}_1^i|}\right)\right)^{tr}. \label{eq:with_f}
\end{equation}
To remove the dependence of \eqref{eq:with_f} on iterates of the generating functions, we can use the fact that $f_t^{i_0}(\boldsymbol{0}) \le q_{i_0}$ and thus $(1-q_{i_0})(1-f_t^{i_0}(0))^{-1} \leq 1$,  to arrive at the final expression 
\begin{equation*}
    % middle 
    \mathbb{E}\left(\frac{1}{|\boldsymbol{Z}_t|^r}\bigm| |\boldsymbol{Z}_{\infty}| > 0\right) 
    % right
    \le (1-\sup_{i \in S}q_i)^{-2} \left(\mathbb{E}\left(\sup_{i \in S} \frac{1}{|\boldsymbol{Y}_1^i|}\right)\right)^{tr}.
\end{equation*}
\section{Numerical Results} \label{sec:numerical}
We now show that the theoretical developments above allow us to approximate quantities relating to the branching process computationally. Python and Julia code to reproduce the results presented here is available at \url{https://github.com/mandarynka033/coalescence_supercritical}. Numerical results will be presented for a finite number of types, $d < \infty$. In this case, denote by $\lambda$ the leading eigenvalue of $\boldsymbol{M}$. By the Perron--Frobenius theorem for non-negative matrices, $\lambda > 0$ with an associated strictly positive left, $\boldsymbol{\nu}$, and right, $\boldsymbol{u}$, eigenvector \cite[Chapter 1.1]{seneta}. Eigenvectors are normalised so that  $\boldsymbol{u}\cdot \boldsymbol{1} =1 $ and $ \boldsymbol{u} \cdot \boldsymbol{\nu} =1$. The process is supercritical if $\lambda > 1$. 

Theorems \ref{exact_thm}, \ref{harmonic_thm}, and \ref{hs_thm} still hold for a finite number of types under their respective assumptions summarised in Section \ref{main_results}.
However, certain requirements of Assumption \ref{R-positive} are redundant in the finite case, such as the requirement for the finiteness of $\boldsymbol{u} \cdot \boldsymbol{\nu}$ and $|\boldsymbol{\nu}|$. Hence, for clarity, we state sufficient conditions for Theorem \ref{exact_thm} to hold in the case $d<\infty$. 
\begin{assumption} \label{finite_types}
The process $\boldsymbol{Z}$ is supercritical, $\boldsymbol{M}$ is irreducible, and for all $i,j  \in [d]$
\begin{equation*}
    \mathbb{E}(Z_1^{i,j} \log^+ Z_1^{i,j})< \infty.
\end{equation*}
\end{assumption}
Under Assumption \ref{finite_types}, when $d <\infty$, Assumptions \ref{infinite_matrix} and \ref{R-positive} are verified, and hence by Theorem~\ref{exact_thm}:
\begin{equation}
 \lim_{T \rightarrow \infty}\mathbb{P}(X_{T,k} < t\bigm| |\boldsymbol{Z}_T| \ge k ) = 
     1 - \mathbb{E}\left(\frac{\sum_{i=1}^d\sum_{s=1}^{\z} (W_{t,s}^{(i)})^k }
    % denominator
    {(\sum_{i=1}^d\sum_{s=1}^{\z}W_{t,s}^{(i)})^k}\bigm| |\boldsymbol{Z}_{\infty}| > 0  \right) \label{eq:thm_finite}.
\end{equation}
The right-hand side of \eqref{eq:thm_finite} will be approximated numerically by Monte Carlo sampling of $W^{(i)}$.
\subsection{Approximate sampling from the law of $W^{(j)}$}
The imaginary unit $\sqrt{-1}$ will be denoted by $\iota$. For $j \in [d]$, let us introduce the following integral transform $\phi_j : \{s \in \mathbb{C}: \operatorname{Re}(s)\ge 0\} \rightarrow \mathbb{C}$:
\begin{equation*}
    \phi_j(s) := \mathbb{E}(e^{-sW^{(j)}}),
\end{equation*}
and let $\boldsymbol{\phi}(s):= (\phi_j(s))_{j=1}^d$. 
%\subsubsection{Moments of $W$}
The following identity \cite[Chapter 5.6]{athreya_ney} relating $\boldsymbol\phi$ and the generating function $\boldsymbol{f}$, for $s$ with $\operatorname{Re}(s) \ge 0$, will be instrumental in evaluating moments of $W^{(j)}$, 
\begin{equation}
    \phi_j(\lambda s) = f^j(\boldsymbol{\phi}(s)) \label{eq:phi_and_f}.
\end{equation}
Under Assumption \ref{moments}, that $2k$ finite moments exist for $Z_1^{i,j}$, for any $j \in [d]$, we approximate $\phi_j(s)$ with its $2k$th order Taylor polynomial approximation around 0:
\begin{equation}
    \phi_j(s) =\mathbb{E}(e^{-sW^{(j)}}) \approx  \sum_{n=0}^{2k} \frac{d^n \phi_j}{d s^n}(0) \frac{1}{n!}(-s)^n = \sum_{n=0}^{2k} \frac{\mathbb{E}((W^{(j)})^n)}{n!} (-s)^n \label{eq:approximation}.
\end{equation}
By Theorem \ref{moy}, $\mathbb{E}(W^{(j)})$ is known for all $j \in S$. For higher moments, $n \in \{2,\dots,2k\}$, $\mathbb{E}((W^{(j)})^n)$ can be approximated by recursively calculating derivatives of \eqref{eq:phi_and_f} symbolically using the Python library SymPy \cite{software:sympy} and evaluating them at 0, which is described step by step below. 
\begin{enumerate}
    \item For each $n \in \{2,\dots,2k\}$ and $j \in [d]$, compute the $n$th derivative of \eqref{eq:phi_and_f} by applying the multivariate Faa di Bruno formula \cite{faa} to the right hand-side
\begin{align}
\frac{d^n \phi_j(\lambda s)}{d s^n} &= \frac{d^n f^j(\boldsymbol{\phi}(s))}{d s^n}  \label{eq:n_derivative}
\end{align}
   and evaluate the derivative in \eqref{eq:n_derivative} at 0 to obtain
    \begin{equation}
    \lambda^n \mathbb{E}((W^{(j)})^n) = \frac{d^n f^j(\boldsymbol{\phi}(s))}{d s^n}\Biggm|_{s=0} \label{eq:set_equations}.
\end{equation}
\item Having previously computed the first $n-1$ moments $\mathbb{E}((W^{(j)})^l)$ for each $j\in[d]$ and each $l \in\{1,\dots,n-1\}$, we can solve the set of equations \eqref{eq:set_equations} to obtain $\mathbb{E}((W^{(j)})^n)$ for each $j\in[d]$.
\end{enumerate}
%\subsubsection{Density of $W$}
Under Assumption \ref{finite_types}, for each $j \in [d]$, there exists a strictly positive function $w_j$ on $(0, \infty)$ \cite[Chapter 5.6]{athreya_ney} such that for any $0 < a<b< \infty$
\begin{equation*}
    \mathbb{P}(a < W^{(j)} \le b) = \int_a^b w_j(x) dx.
\end{equation*}
 Observe that for all $j \in [d]$ the law of $W^{(j)}$ has an atom at 0:
\begin{equation}
    \mathbb{P}(W^{(j)} = 0) = q_j.
\end{equation}
By Lebesgue's decomposition theorem \cite[Theorem 2.10, p.\ 40]{kallenberg}, the law of $W^{(j)}$ can be decomposed into a density $w_j$ on $(0, \infty)$ and an atom at $0$.
It is more numerically efficient to consider the density of $W^{(j)}$ conditioned on non-extinction separately from the case where the population becomes extinct: %For any $j \in [d]$
\begin{align}
    \mathbb{E}(e^{-\iota zW^{(j)}}) &= q_j +(1-q_j)  \mathbb{E}(e^{-\iota zW^{(j)}} \bigm| W^{(j)} > 0) \notag\\
\Longleftrightarrow\qquad    (1-q_j)  \mathbb{E}(e^{- \iota zW^{(j)}}\bigm| W^{(j)} > 0) &= \mathbb{E}(e^{-\iota zW^{(j)}})  - q_j \label{eq:non_extinction}.
\end{align}

The right-hand side of \eqref{eq:non_extinction}, will be inverted with the inverse discrete Fourier transform to obtain the relevant density. The procedure is described in detail in Algorithm \ref{alg:density}. 
\begin{definition}[Inverse discrete Fourier transform, \cite{fourier}]
     For $M \in \mathbb{N}_+$ equally spaced evaluations of the characteristic function $\{\hat{x}_j\}_{j=0}^{M-1}$ on $[-a, a]$, $a \in \mathbb{R}_+$, we obtain $M$ equally spaced values from the density $\{x_j\}_{j=0}^{M-1}$ on $[-\frac{M-1}{2a}, \frac{M-1}{2a}]$ by computing for $j \in \{0,\dots,M-1\}$
\begin{equation}
    x_j = \frac{1}{M} \sum_{l=0}^{M-1} \hat{x}_l \exp{\left(\frac{2\pi \iota}{M} lj\right)}. \label{eq:fourier}
\end{equation}
\end{definition}
\begin{algorithm}[t]
\caption{$W^{(j)}$ density approximation}
\label{alg:density}
\begin{algorithmic}[1]
\REQUIRE small $z > 0$, $L \in \mathbb{N}_+$ large enough to observe the decay of the characteristic function to 0, $N \in \mathbb{N}_+$ large enough cover the non-trivial regions of the density, $M \in \mathbb{N}_+$ large enough to have sufficient granularity of the approximated characteristic function, $\lambda$ the leading eigenvalue of $\boldsymbol{M}$.
\FOR{$j=1$ \TO $d$}
\STATE Approximate the value of the characteristic function, $\phi_j(iz)$, near 0 with the Taylor approximation \eqref{eq:approximation} evaluated on $N$ equally spaced points on $[z, \lambda z]$ and $N$ points on $[ -\lambda z, -z]$. 
\ENDFOR
\FOR{$l = 1$ \TO $L-1$}
    \STATE Use the $2N$ approximate values of $\phi_j(\iota z), j \in [d]$, on $\pm[\lambda^{l-1}, \lambda^{l} z]$ to calculate $N$ values of the characteristic function on $\pm[\lambda^{l}z, \lambda^{l+1} z]$ using the identity \eqref{eq:phi_and_f}. 
\ENDFOR
\FOR{$j=1$ \TO $d$}
\STATE Interpolate the characteristic function approximation $\phi_j(\iota z)$ linearly to get $M$ equally spaced points $\{\hat{x}_l\}_{l=0}^{M-1}$ on $[-\lambda^{L}z, \lambda^{L} z]$.
\FOR{$l=0$ \TO $M-1$} 
\STATE Compute the characteristic function conditioned on non-extinction, $\hat{x}_l \leftarrow \hat{x}_l  - q_j$.
\ENDFOR
\STATE Using $\{\hat{x}_l\}_{l=0}^{M-1}$, compute the discrete inverse Fourier transform \eqref{eq:fourier} to obtain $M$ values of the density approximation $\{x_l\}_{l=0}^{M-1}$ of $W^{(j)}$ on [$-\frac{M-1}{\lambda^L z}, \frac{M-1}{\lambda^L z}$] spaced $\frac{2}{\lambda^L z}$ apart.
\ENDFOR
\end{algorithmic}
\end{algorithm}

% Inverting the characteristic function will give us the density of W. A standard inverse fourier transform requires equally spaced points, whereas our points get exponentially further away from each other - $z, \lambda z, \lambda^2 z, \dots$. 
Let $\boldsymbol{w}:= (\tilde{w}_j)_{j=1}^d$ be the set of approximate densities computed using Algorithm \ref{alg:density}. Each $W^{(j)}$ can be sampled using composition sampling (see Algorithm~\ref{alg:samples} for a more detailed description).  

\begin{example}\label{poisson}
    Let us demonstrate the performance of Algorithm \ref{alg:density} on a $d=2$-type system with Poisson distributed offspring (Table \ref{tab:poisson}). The maximal eigenvalue is $\lambda \approx 1.81$ and $\boldsymbol{q} \approx \{0.36, 0.23\}$.
        \begin{table}[h]
    \centering
    \begin{tabular}{ccc}
         & \bf Offspring of type 1 & \bf Offspring of type 2 \\ 
        \hline
        \bf Parent of type 1 & Poisson($1$) & Poisson($0.5$) \\
        \bf Parent of type 2 &  Poisson($0.5$) &  Poisson($1.5$)
    \end{tabular}
    \caption{\label{tab:poisson}Offspring distributions for a slightly supercritical system.}
\end{table}
For $j \in [2]$, the quantity $\phi_j(\iota z) - q_j$ and the density of $W^{(j)}$ conditioned on non-extinction, obtained with Algorithm \ref{alg:density}, are illustrated in Figures \ref{fig:characteristic} and \ref{fig:density} respectively. 
\begin{figure}[t]
    \centering
    \includegraphics[width=0.7\linewidth]{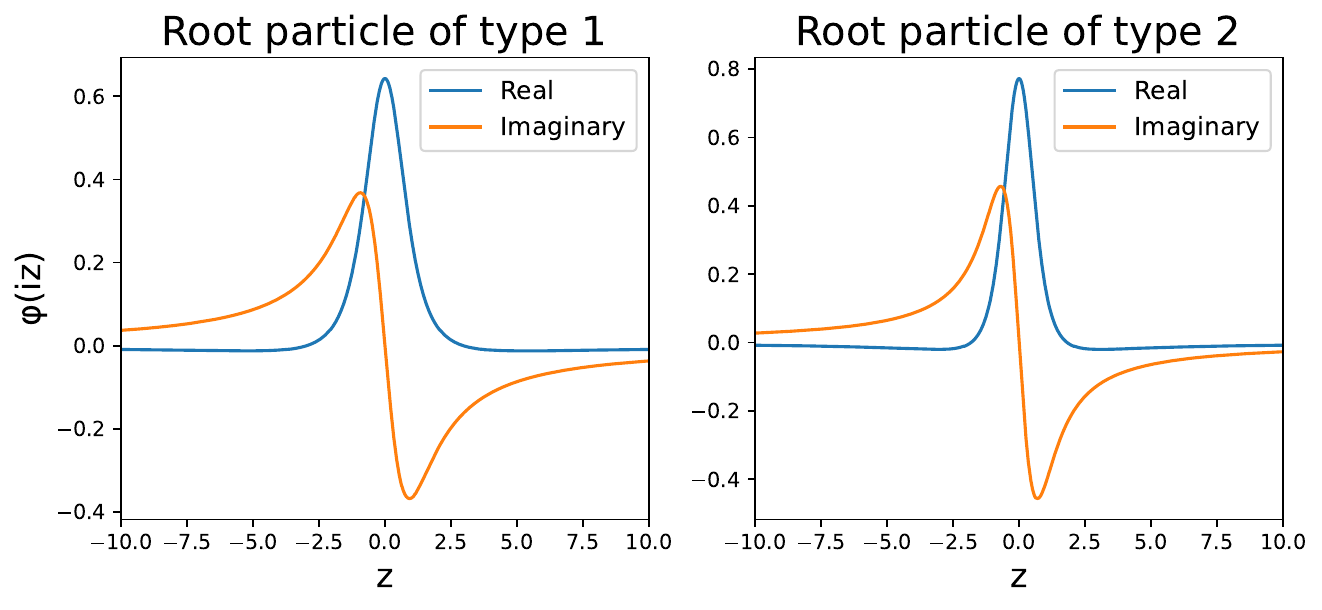}
    \caption{The characteristic function of $W^{(j)}, j \in [2],$ for the system in Example \ref{poisson}.}
    \label{fig:characteristic}
\end{figure}
\begin{figure}[t]
    \centering
    \includegraphics[width=0.7\linewidth]{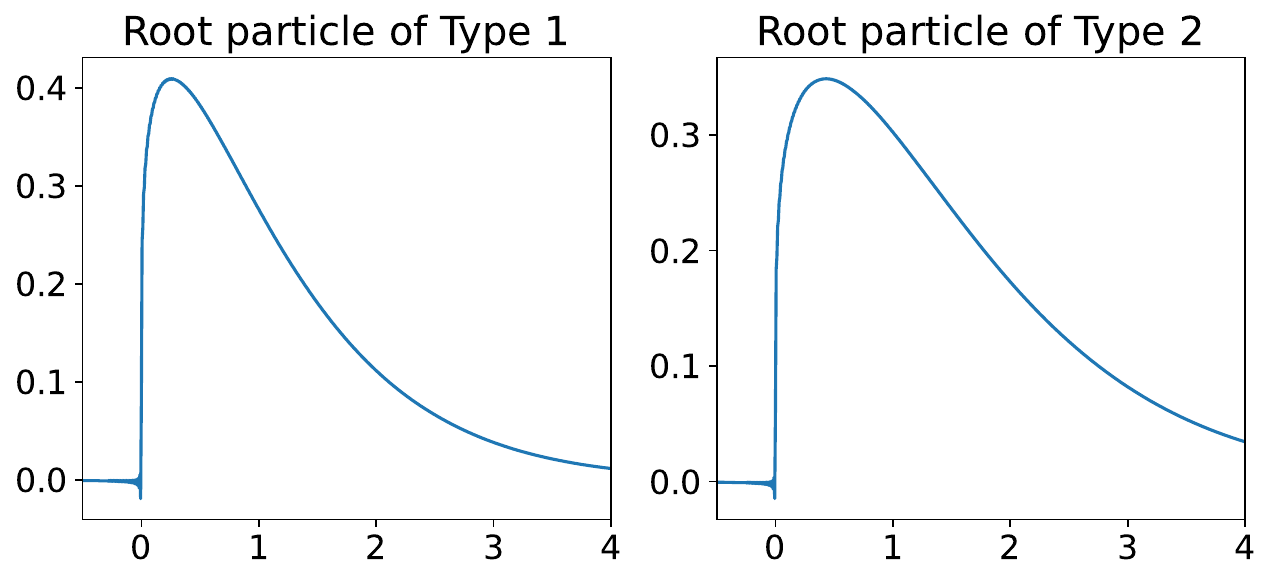}
    \caption{Estimated density of $W^{(j)}, j \in [2],$ for the system in Example \ref{poisson}.}
    \label{fig:density}
\end{figure}
%\end{example}
%\subsection{Comparison the density of $W$ by inversion and by simulation}
 We compare samples of $W^{(1)}$ from our approximation of its density with direct simulation of the branching process. With regard to the latter, we sample $\boldsymbol{Z}_T^1/\lambda^T$ for some large $T$, say $T=20$. For a larger number of generations, direct simulation becomes difficult---this issue will be described in more detail in Subsection \ref{computation}.
%\begin{example}
%    Consider again the system from Example \ref{poisson}.
The comparison of the inverted density of $W^{(1)}$ with a direct simulation of the normalised population is presented in Figure \ref{fig:histo}. As is clear from the Figure, there is good agreement between the two approaches, confirming that Algorithm \ref{alg:density} provides an accurate and approximately unbiased approximation to the density $w_j$. 
    % \begin{figure}
    %     \centering
    %     \includegraphics[width=0.9\linewidth]{images/comparison.pdf}
    %     \caption{Comparison of the inverted density and directly simulated density for 1000 samples. For 5000 samples they are virtually undistinguishable.\label{fig:comparison}
    % \end{figure}
    \begin{figure}[t]
      \centering
        \includegraphics[width=0.99\linewidth]{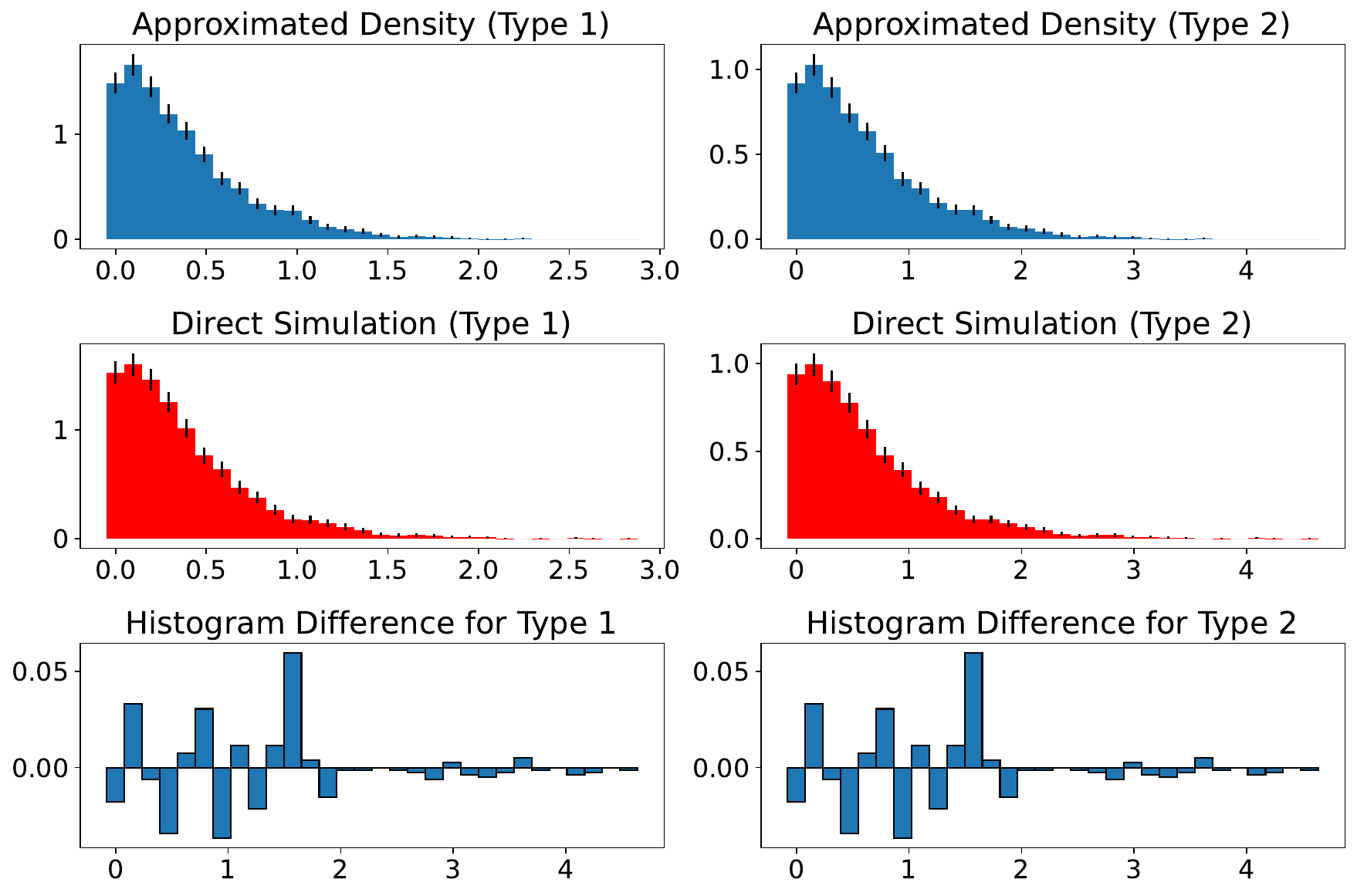}
        \caption{Histograms with their difference for 5000 samples from the approximated density and 5000 renormalised population sizes after 20 generations of a simulated branching process from Example \ref{poisson} started from a root ancestor of type 1. Only non-zero values are plotted. }\label{fig:histo}
    \end{figure}
\end{example}
\subsection{Estimation of Coalescence Probability}\label{composite_sampling}

To estimate coalescence probability using \eqref{eq:thm_finite}, we simulate the system until generation $T$ and them sample from $W^{(j)}$, $j \in [d]$, according to the types of individuals found in generation $T$, as described in more detail in Algorithm \ref{alg:probability}. The performance of this approach, which is based on Theorem \ref{exact_thm}, can be compared against a direct simulation of genealogy, which is described in more detail in Algorithm \ref{alg:direct}.

\begin{example}\label{supercritical}
We analyse the performance of our method on two systems---one nearly critical (Example \ref{poisson}) and one significantly supercritical (maximal eigenvalue of $\lambda = 5$) introduced below (Table \ref{tab:poisson2}). Results are shown in Figure \ref{fig:coalescence}. 

Based on Figure \ref{fig:coalescence}, the coalescence probability estimated using \eqref{eq:thm_finite} agrees closely with the estimate from direct simulation. Moreover, as can be expected from a theoretical perspective, our experiments visualise how the bounds become tighter as the system becomes more supercritical. This is particularly useful since highly supercritical systems are much more difficult to simulate due to their large population sizes.
  \begin{table}[h]
    \centering
    \begin{tabular}{ccc}
         & \bf Offspring of type 1 & \bf Offspring of type 2 \\ 
        \hline
        \bf Parent of type 1 & Poisson($4$) & Poisson($2$) \\
        \bf Parent of type 2 & Poisson($1$) &  Poisson($3$) 
    \end{tabular}
    \caption{\label{tab:poisson2}Offspring distributions for a significantly supercritical system.}
\end{table}
\begin{figure}[ht!]
    \centering
    % First image
    \begin{subfigure}{0.45\textwidth}
        \centering
        \includegraphics[width=\textwidth]{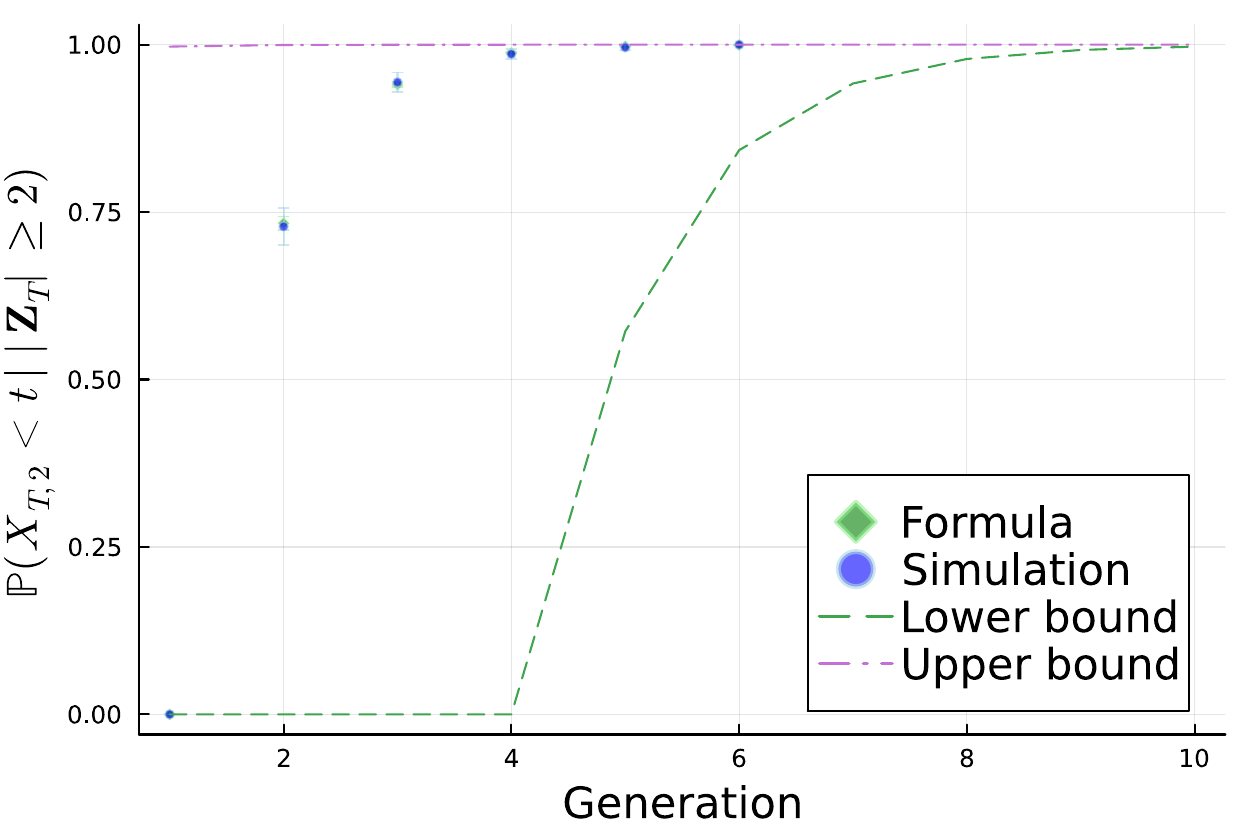}
        \caption{Significantly supercritical system from Example \ref{supercritical}.}
    \end{subfigure}
    \hfill
    % Second image
    \begin{subfigure}{0.45\textwidth}
        \centering
        \includegraphics[width=\textwidth]{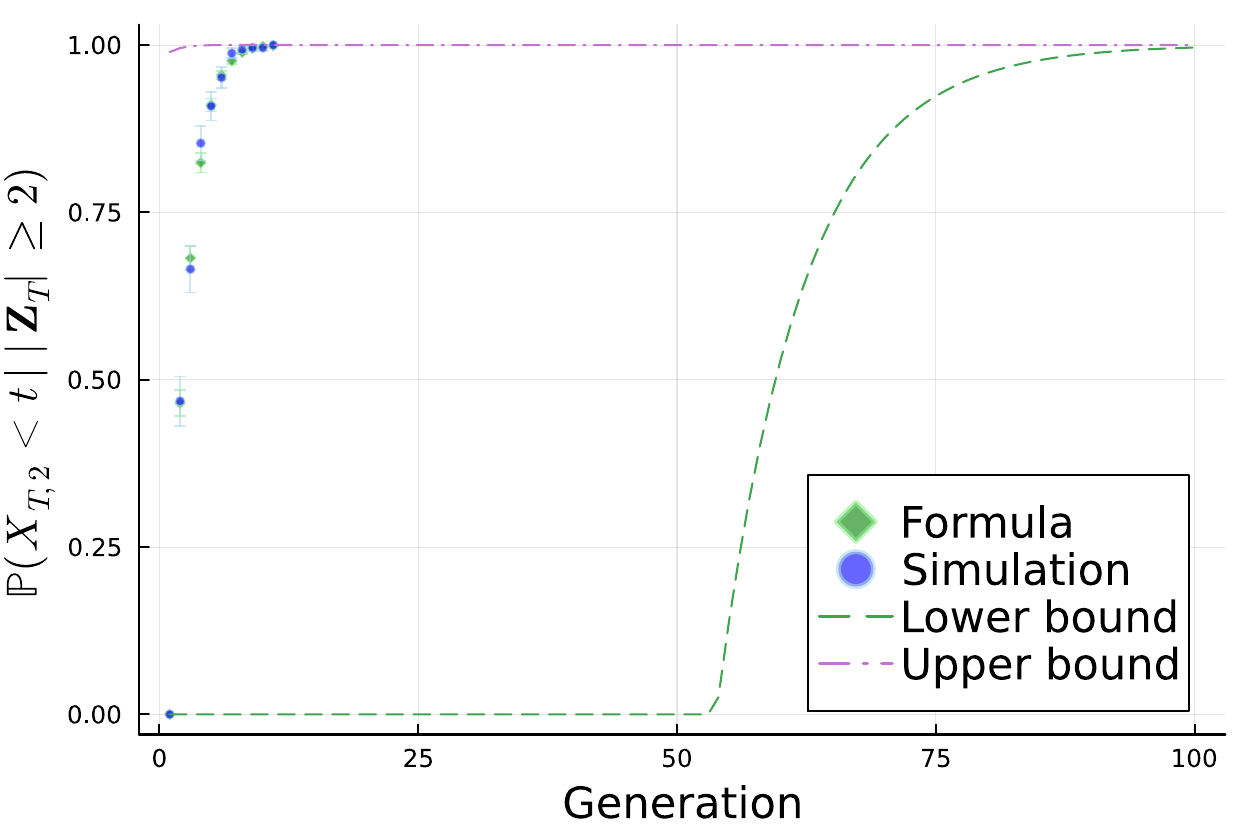}
        \caption{Slightly supercritical system from Example \ref{poisson}.}
    \end{subfigure}
    \caption{Coalescence probability estimation and bounds (see Theorems \ref{exact_thm} and \ref{hs_thm} respectively).}
    \label{fig:coalescence}
\end{figure}
\end{example}
\subsection{Computational cost} \label{computation}
Computing the coalescence probability using Theorem \ref{exact_thm} is much faster than keeping track of the genealogy of a branching process ran for a large number of generations. The actual length of computation is presented in Table~\ref{table_time}. We ran 1000 simulations for both the slightly supercritical system (Example \ref{poisson}) and the very supercritical system (Example \ref{supercritical}). Concerning the direct simulation of genealogy, for Example \ref{poisson} we considered the probability of coalescing in the first 10 generations, $t \in [1, 10]$, with the simulation finishing 10 generations later, $T = t + 10$. For Example \ref{supercritical}, we set $t \in [1, 5]$ and $T = t + 5$.

All the experiments have been executed in serial using Python 3.12.3 on a single core of Intel(R) Core(TM) Ultra 7 256V.

\begin{table}[ht!]
    \centering
    \caption{Computation time (in seconds) for coalescence probability estimation using Theorem \ref{exact_thm} (which consists of three sub-algorithms) and via direct simulation.}
    \label{table_time}
    \begin{tabular}{l r r} % column alignment: l=left, c=center, r=right
       &  \bf Slightly & \bf Very\\
      \bf Algorithm & \bf Supercritical & \bf Supercritical  \\
      \hline
      \bf Theorem \ref{exact_thm} estimation & & \\
        \quad Moments of W (Solving equation \eqref{eq:set_equations})  & 12   & 16   \\ 
        \quad Inverse Fourier transform (Algorithm  \ref{alg:density})  & 0.07   & 0.013   \\ 
        \quad Simulating the process until time $t$ (Algorithm \ref{alg:probability}) & 0.8   & 0.2   \\ 
        \bf Direct simulation (\ref{genealogy_simulation}) & 17,972 & 469 \\ % %4h59m32s & 7m49s
    \end{tabular}
\end{table}
\begin{remark}
    For the direct simulation of a very supercritical system, the process is evolved for fewer generations (10 for the slightly supercritical system versus 5 for the very supercritical system) due to the high memory usage needed for the larger population. For instance, for a population size of $10^9$, the vector storing their ancestors requires approximately $5$ GB of memory. For Example \ref{supercritical}, the population is expected to reach this size at around generation $12$, which makes simulating the process for a larger number of generations challenging. 
\end{remark}
\section{Discussion} \label{sec:discussion}
In Theorem \ref{exact_thm}, we derived a formula for the probability of all individuals from a sample of size $k$ from generation $T$ coalescing on the interval $[0, t)$ from a countably multi-type supercritical branching process as $T \rightarrow \infty$. This expression is given in terms of the limiting random variable describing the rescaled population size and the population size at time $t$, thereby extending the result of \cite{Hong2015} to countably infinite set of types and a process that can become extinct. In Theorem \ref{harmonic_thm} we show that the probability of interest can be upper- and lower-bounded in terms of the harmonic moments of the total population size $|\boldsymbol{Z}_t|$, and further in Theorem \ref{hs_thm} that those harmonic moments can in turn be bounded in terms of expectations of quantities relating to the Harris--Sevastyanov transformation of the original process at time $1$.  By combining these observations (Corollary \ref{corr:mainresult}), one obtains explicitly computable upper and lower bounds on the probability of interest which enable it to be approximated when $t$ becomes large enough to preclude direct simulation of $\boldsymbol{Z}_t$.

Our proof relies on Assumption \ref{moments} that $\boldsymbol{Z}_1^{i,j}$ possesses $2k$ moments---an assumption that holds for many distributions commonly used in applications. We anticipate that this condition can be relaxed with the use of more refined methods to bound the expected deviation of $\boldsymbol{Z}_T$ from its expected value at time $t$. Progress in this direction could be supported by the development of large-deviation results specific to multi-type branching processes.

In addition, one could analyse the genealogical process associated with $(\boldsymbol{Z}_t)$ in further detail to obtain a distribution on more than simply $X_{T,k}$. For example, for a critical multi-type branching process, \cite{osvaldo} studies among other things the distribution of the joint collection of all $k-1$ split times, the number of individuals that coalesce at any specific split time, as well as the type of the ancestor at each splitting event.

Our numerical results demonstrate that the bounds we obtain become tighter as the maximal eigenvalue increases. This is particularly useful because significantly supercritical processes are extremely challenging to simulate.

\section*{Acknowledgements}
AMJ acknowledges the financial support of the by United Kingdom Research and Innovation (UKRI) via grant number EP/Y014650/1, as part of the ERC Synergy project OCEAN. 

JK is supported by EPSRC grant: Project Reference 2733781.

\paragraph{Data access statement} No new data was generated or analysed in this study. Python and Julia code for reproducing the numerical results can be found at \url{https://github.com/mandarynka033/coalescence_supercritical}.

\appendix
\section{Technical Results}

\subsection{A modified conditional dominated convergence theorem}
\label{dominated}
The following proposition is used in line \eqref{eq:limit_expectation} of the proof of Theorem \ref{exact_thm}.
\begin{proposition}
% For any sequence of random variables $(\xi_i)_{i \in S}$ with $\xi_i \rightarrow \xi$ a.s. and a random variable $\eta$ with $\mathbb{E}(\eta) < \infty$ 
For any sequence of random variables $(\xi_T)_{T=1}^{\infty}$ satisfying $|\xi_T| < \eta$ for all $T \in \mathbb{N}_0$ for some random variable $\eta$ with $\mathbb{E}(\eta) < \infty$, and a supercritical, positively regular multi-type branching process $(\boldsymbol{Z}_t)_{t\in\mathbb{N}_0}$, it holds that
      \begin{equation}
        \lim_{T \rightarrow \infty} \mathbb{E}(\xi_{T} \bigm| |\boldsymbol{Z}_T| \ge k)  =  \mathbb{E}\left(\lim_{T \rightarrow \infty}\xi_{T} \bigm| |\boldsymbol{Z}_{\infty}| > 0\right).
    \end{equation}
\end{proposition}
\begin{proof}
For any $c \in \mathbb{N}_0$, denote $\{|\boldsymbol{Z}_{\infty}|>c\}:=\{\lim\inf_{T \rightarrow \infty}|\boldsymbol{Z}_{T}|>c\}$. We first show equivalence of the events $\{|\boldsymbol{Z}_T| \rightarrow \infty\}, \{|\boldsymbol{Z}_{\infty}| \ge k\}$, and $\{|\boldsymbol{Z}_{\infty}| > 0\}$ up to null sets. According to \cite[Chapter 2.6]{harris1963}, under the stated conditions on $(\boldsymbol{Z}_t)_{t\in\mathbb{N}_0}$ and for any non-zero $\boldsymbol{z} \in \mathbb{N}_0^{d}$:
\begin{equation*}
    \mathbb{P}(\boldsymbol{Z}_T = \boldsymbol{z} \text{ infinitely often}) = 0,
\end{equation*}
and hence $\mathbb{P}(|\boldsymbol{Z}_{\infty}| = 0)+\mathbb{P}(|\boldsymbol{Z}_T|\rightarrow \infty) = 1$. Furthermore, for any $j_1, j_2 \in \mathbb{N}_0$, the events $\{|\boldsymbol{Z}_{\infty}| > j_1\}$ and $\{|\boldsymbol{Z}_{\infty}| > j_2\}$ are equal almost surely, which can be proven by considering the probability of their symmetric difference:
\begin{align*}
    \mathbb{P}(\{|\boldsymbol{Z}_{\infty}| > j_1\} \setminus \{|\boldsymbol{Z}_{\infty}| > j_2\}) + \mathbb{P}(\{|\boldsymbol{Z}_{\infty}| > j_2\} \setminus \{|\boldsymbol{Z}_{\infty}| > j_1\}) &=\\ \mathbb{P}(j_1 <|\boldsymbol{Z}_{\infty}| \leq j_2) + \mathbb{P}(j_2 <|\boldsymbol{Z}_{\infty}| \leq j_1) &=0.
\end{align*}
The proof of the equivalence of $\{|\boldsymbol{Z}_T| \rightarrow \infty\}$ and $\{|\boldsymbol{Z}_{\infty}|> j\}$ for any $j \in \mathbb{N}_0$ is analogous.

Setting $j_1 = 0$ and $j_2 = k-1$, we get the equivalence of $\{|\boldsymbol{Z}_{\infty}| > 0\}$, $\{|\boldsymbol{Z}_{\infty}| \ge k\}$, and $\{|\boldsymbol{Z}_T| \rightarrow \infty\}$ up to null sets. %Hence, the expected values and probabilities conditioned on them are equal.
% \begin{equation}
    % \mathbb{E}(. \bigm| \{|\boldsymbol{Z}_{\infty}| > 0\}) = \mathbb{E}(. \bigm| \{|\boldsymbol{Z}_{\infty}| \ge k\}) \quad ; \quad \mathbb{P}(. \bigm| \{|\boldsymbol{Z}_{\infty}| > 0\}) = \mathbb{P}(. \bigm| \{|\boldsymbol{Z}_{\infty}| \ge k\}). \label{eq:equivalence}
% \end{equation}
Hence, for any fixed $k$,
\begin{equation*}
    0 < 1- q_{i_0} = \mathbb{P}(|\boldsymbol{Z}_{\infty}| > 0) = \mathbb{P}(|\boldsymbol{Z}_T| \rightarrow \infty) = \mathbb{P}(|\boldsymbol{Z}_{\infty}| \ge k).
\end{equation*}
Next, note that the quantity $\mathbb{E}(\xi_{T} \mathbbm{1}(|\boldsymbol{Z}_T| \ge k))$ is bounded:
\begin{equation*}
   -\infty < -\mathbb{E}(\eta) \le  \mathbb{E}(\xi_{T} \mathbbm{1}(|\boldsymbol{Z}_T| \ge k)) \le \mathbb{E}(\xi_{T} )) \le \mathbb{E}(\eta) < \infty.
\end{equation*}
Consequently, we can apply the standard dominated convergence theorem to show: %we can represent \ref{eq:limit_product} as a product of limits:
\begin{align}
        \lim_{T \rightarrow \infty} \mathbb{E}(\xi_{T} \bigm| |\boldsymbol{Z}_T| \ge k)
&= \lim_{T \rightarrow \infty} \frac{\mathbb{E}(\xi_{T} \mathbbm{1}(|\boldsymbol{Z}_T| \ge k)) }{\mathbb{P}(|\boldsymbol{Z}_T| \ge k)} \nonumber\\
    &= \lim_{T \rightarrow \infty} \frac{1}{\mathbb{P}(|\boldsymbol{Z}_T| \ge k)} \lim_{T \rightarrow \infty} \mathbb{E}(\xi_{T} \mathbbm{1}(|\boldsymbol{Z}_T| \ge k)) \nonumber\\
    & = \frac{1}{\lim_{T \rightarrow \infty}\mathbb{P}(|\boldsymbol{Z}_{T}| \ge k)}  \mathbb{E}\left(\lim_{T \rightarrow \infty}\xi_{T} \mathbbm{1}(|\boldsymbol{Z}_T| \ge k)\right) \nonumber\\ %\label{normal_dominated}\\
    &= \frac{\mathbb{E}(\lim_{T \rightarrow \infty}\xi_{T} \lim_{T \rightarrow \infty}\mathbbm{1}(|\boldsymbol{Z}_T| \ge k))}{\lim_{T \rightarrow \infty}\mathbb{P}(|\boldsymbol{Z}_{T}| \ge k)}.   \label{eq:indicator}
\end{align}
%Consider the limit of the indicator function in \eqref{eq:indicator}
%\begin{equation}
%    \lim_{T \rightarrow \infty} \mathbbm{1}(|\boldsymbol{Z}_T| \ge k) \label{eq:extended_mapping}.
%\end{equation}
The argument of the expectation in \eqref{eq:indicator} is continuous on $\mathbb{R}$ except at $k$. 

Since $\mathbb{P}(\lim_{T \rightarrow \infty}|\boldsymbol{Z}_T| = k) = 0$, we can apply the continuous mapping theorem (see for example \cite[Lemma 5.3, p.\ 103]{kallenberg}) %to \eqref{eq:indicator}
to obtain 
\begin{equation}
    \lim_{T \rightarrow \infty} \mathbbm{1}(|\boldsymbol{Z}_T| \ge k) = \mathbbm{1}\left(\lim_{T \rightarrow \infty} |\boldsymbol{Z}_T| \ge k\right)  = \mathbbm{1}( |\boldsymbol{Z}_{\infty}| \ge k). \label{eq:indicator_equivalence}
\end{equation}
Returning to \eqref{eq:indicator} and substituting \eqref{eq:indicator_equivalence}, we get
\begin{align*}
\frac{\mathbb{E}(\lim_{T \rightarrow \infty}\xi_{T} \lim_{T \rightarrow \infty}\mathbbm{1}(|\boldsymbol{Z}_T| \ge k))}{\lim_{T \rightarrow \infty}\mathbb{P}(|\boldsymbol{Z}_{T}| \ge k)} %\nonumber \\
    & = \mathbb{E}\left(\lim_{T \rightarrow \infty}\xi_{T}\bigm| |\boldsymbol{Z}_{\infty}| \ge k\right)
    = \mathbb{E}\left(\lim_{T \rightarrow \infty}\xi_{T}\bigm| |\boldsymbol{Z}_{\infty}| > 0\right) %\label{eq:equivalent}
\end{align*}
where the latter equality follows from the equivalence of $\{|\boldsymbol{Z}_{\infty}| \ge k\}$ and $\{|\boldsymbol{Z}_{\infty}| > 0\}$.
\end{proof}
\subsection{Upper bound on a quantity from equation \eqref{less_than1}}\label{less1}
% \adaml{Should we give this result a name?}
We will prove that the following inequality, which is used in \eqref{eq:close_to_mean} in the proof of Theorem \ref{harmonic_thm}, holds almost surely on the event $\{|\boldsymbol{Z}_t| > 0\}$:
        \begin{equation}
         % Other event
    \frac{\V{k}}
    % denominator
    {(|\boldsymbol{Z}_t|^{\frac{k-1}{k}}\V{1})^k} 
    %conditioning
    %\bigm| |\boldsymbol{Z}_{\infty}| > 0,\boldsymbol{Z}_t  \bigg)\bigm| |\boldsymbol{Z}_{\infty}| > 0\Bigg)
    \le 1,
        \end{equation}
    where $\V{j}$, $j\in\mathbb{N}_+$, is given by \eqref{eq:V}.
        \begin{proof}
          We have
    \begin{align}
    % \frac{\sum_{i\in S}\sum_{s=1}^{\z} \frac{(W_{t,s}^{(i)})^k}{|\boldsymbol{Z}_t|} }    {(|\boldsymbol{Z}_t|^{\frac{k-1}{k}}\sum_{i\in S}\sum_{s=1}^{\z}\frac{W_{t,s}^{(i)}}{|\boldsymbol{Z}_t|})^k}
    %&=  \frac{\sum_{i\in S}\sum_{s=1}^{\z} (W_{t,s}^{(i)})^k }{(\sum_{i\in S}\sum_{s=1}^{\z}W_{t,s}^{(i)})^k} \nonumber\\
    %&= \frac{\sum_{i \in S}\sum_{s=1}^{\z} (W_{t,s}^{(i)})^k }{\sum_{\substack{\{(k_{i,s}):\: k_{i,s} \geq 0,\\\sum k_{i,s} = k\}}} \prod_{i \in S} \prod_{s=1}^{\z}(W_{t,s}^{(i)})^{k_{i,s}}}\label{eq:multinomial}
    \frac{\V{k} }
    {(|\boldsymbol{Z}_t|^{\frac{k-1}{k}}\V{1})^k}
    &=  \frac{\sum_{i\in S}\sum_{s=1}^{\z} (W_{t,s}^{(i)})^k }
    {(\sum_{i\in S}\sum_{s=1}^{\z}W_{t,s}^{(i)})^k}
    = \frac{\sum_{i \in S}\sum_{s=1}^{\z} (W_{t,s}^{(i)})^k }
    {\sum_{\substack{\{(k_{i,s}):\: k_{i,s} \geq 0,\\\sum k_{i,s} = k\}}} \prod_{i \in S} \prod_{s=1}^{\z}(W_{t,s}^{(i)})^{k_{i,s}}}
    \label{eq:multinomial}
    \end{align}
    where \eqref{eq:multinomial} follows by expanding the $k$th power in the denominator using the multinomial theorem. Now observe that in \eqref{eq:multinomial} the numerator is a component of the denominator:
    \begin{align*}
      %&\mathbb{E}\Bigg(\mathbb{E}\bigg(
      \frac{\sum_{i \in S}\sum_{s=1}^{\z} (W_{t,s}^{(i)})^k }
    % denominator
    {\sum_{\substack{\{(k_{i,s}):\: k_{i,s} \geq 0,\\\sum k_{i,s} = k\}}} \prod_{i \in S} \prod_{s=1}^{\z}(W_{t,s}^{(i)})^{k_{i,s}}}
    %conditioning
    %\bigm| |\boldsymbol{Z}_{\infty}| > 0,\boldsymbol{Z}_t  %\bigg)\bigm| |\boldsymbol{Z}_{\infty}| > 0\Bigg)
    %\nonumber\\
    &= \frac{\sum_{i \in S}\sum_{s=1}^{\z} (W_{t,s}^{(i)})^k }
    % denominator
    {\sum_{i \in S}\sum_{s=1}^{\z} (W_{t,s}^{(i)})^k +\sum_{\substack{\{(k_{i,s}):\: k_{i,s} \geq 0,\\\sum k_{i,s} = k,\\ \max k_{i,s} \neq k\}}} \prod_{i \in S} \prod_{s=1}^{\z}(W_{t,s}^{(i)})^{k_{i,s}}}
    %conditioning
    %\bigm| |\boldsymbol{Z}_{\infty}| > 0,\boldsymbol{Z}_t  \bigg)\bigm| |\boldsymbol{Z}_{\infty}| > 0\Bigg)
    %\label{eq:1over1}
    \end{align*}
    which has the form  
    \begin{equation*}
        \frac{\sum_{i \in S}\sum_{s=1}^{\z} (W_{t,s}^{(i)})^k }{\sum_{i \in S}\sum_{s=1}^{\z} (W_{t,s}^{(i)})^k  + C} = \left(1 + \frac{C}{\sum_{i \in S}\sum_{s=1}^{\z} (W_{t,s}^{(i)})^k} \right)^{-1}
    \end{equation*}
    for some $C \ge 0$. As $W_{t,s}^{(i)} \geq 0$ this quantity is clearly bounded above by $1$.
\end{proof}
\section{Branching Processes Properties}
\subsection{Radius of Convergence for Example \ref{deterministic2}} \label{radius_convergence}

The following is an equivalent definition of the radius of convergence of $\boldsymbol{M}$ \cite[Corollary of Theorem 6.1]{seneta}: For an element-wise finite matrix $\boldsymbol{M} = (m_{i,j})_{i,j \in S}$ that is irreducible with a common finite period $d \ge 1$,
\begin{equation*}
    \lim_{n \rightarrow \infty} (m_{i,i}^{(nd)})^{1/nd} = R^{-1},
\end{equation*}
where $m_{i,i}^{(nd)}$ is the $i$th diagonal element of $\boldsymbol{M}^{nd}$.

For any stochastic matrix $\boldsymbol{T}=(t_{i,j})_{i,j \in S}$, the radius of convergence is $R=1$. In Example \ref{deterministic2}, we have $\boldsymbol{M}=2 \boldsymbol{T}$, for a stochastic matrix $\boldsymbol{T}$. Hence
\begin{align*}
    \lim_{n \rightarrow \infty} (m_{i,i}^{(nd)})^{1/nd} &= \lim_{n \rightarrow \infty} (2^{nd} t_{i,i}^{(nd)})^{1/nd} = 2\lim_{n \rightarrow \infty} (t_{i,i}^{(nd)})^{1/nd} = 2 R^{-1} = 2,
\end{align*}
so the radius of convergence of $\boldsymbol{M}$ is $1/2$.
\subsection{Proof of Lemma \ref{gamma_lemma}}\label{gamma}
%We are going to prove Lemma \ref{gamma_lemma}:
%\begin{equation}
% \mathbb{E}\left(\frac{1}{|\boldsymbol{Z}_t|^k} \bigm| |\boldsymbol{Z}_t| > 0\right) = \frac{1}{\Gamma(k)}\int_0^{\infty} u^{k-1} \frac{f_t^{i_0}(e^{-u}\boldsymbol{1})}{1-f_t^{i_0}(\boldsymbol{0})} du  \label{harmonic_moment_formula}.
%\end{equation}
%\begin{proof}
%By definition of the gamma function $\Gamma(k)$, for $k >0$:
%\begin{equation}
%\Gamma(k) = \int_0^{\infty} z^{k-1} e^{-z} dz.
%\end{equation}
Substitute $z=u|\boldsymbol{Z}_t|$ into the definition of the Gamma function $\Gamma(r)$ as an integral with respect to $z$ to obtain
\begin{align*}
   \Gamma(r) = \int_0^{\infty} z^{r-1} e^{-z} dz %\\
   & = \int_0^{\infty} (u|\boldsymbol{Z}_t|)^{r-1} e^{-u|\boldsymbol{Z}_t|} |\boldsymbol{Z}_t| du 
   =|\boldsymbol{Z}_t|^r \int_0^{\infty} u^{r-1} e^{-u|\boldsymbol{Z}_t|} du.
   \end{align*} 
   Rearranging, we get
   \begin{equation*}
       \frac{1}{|\boldsymbol{Z}_t|^r} = \frac{1}{\Gamma(r)} \int_0^{\infty} u^{r-1} e^{-u|\boldsymbol{Z}_t|} du,
   \end{equation*}
   from which we can consider the expected value
   \begin{equation}
        \mathbb{E}\left(\frac{1}{|\boldsymbol{Z}_t^{i_0}|^r} \bigm| |\boldsymbol{Z}_t| > 0\right)  = \mathbb{E}\left(\frac{1}{\Gamma(r)}\int_0^{\infty} u^{r-1} e^{-u|\boldsymbol{Z}_t|} du \bigm| |\boldsymbol{Z}_t| > 0 \right). \label{no_generating}
   \end{equation}
   Note that the generating function of the original process $|\boldsymbol{Z}_t^{i_0}|$ is related to the generating function of the same process conditioned on non-extinction via
\begin{equation}
    \mathbb{E}(s^{|\boldsymbol{Z}_t|} \bigm| |\boldsymbol{Z}_t| > 0) = \frac{f_t^{i_0}(s\boldsymbol{1})- f_t^{i_0}(\boldsymbol{0})}{1-f_t^{i_0}(\boldsymbol{0})}. \label{generating0}
\end{equation}
 Since the expected value and integration order can be exchanged due to Tonelli's theorem (e.g.\ \cite[Theorem 1.29, p.\ 25]{kallenberg}), we can substitute (\ref{generating0}) into (\ref{no_generating}) to arrive at equation \eqref{eq:harmonic_moment_formula}
%\begin{equation*}
%    \mathbb{E}\left(\frac{1}{|\boldsymbol{Z}_t|^k} \bigm| |\boldsymbol{Z}_t| > 0\right) = \frac{1}{\Gamma(k)}\int_0^{\infty} u^{k-1} \frac{f_t^{i_0}(e^{-u}\boldsymbol{1})- f_t^{i_0}(\boldsymbol{0})}{1-f_t^{i_0}(\boldsymbol{0})} du.
%\end{equation*}
as required. \qed
%\end{proof}

\subsection{Proof of Lemma \ref{lemma:iterates}}\label{iterates}
%We will prove Lemma \ref{lemma:iterates}:
%\begin{equation}
%    \boldsymbol{F}_t(\boldsymbol{s}) = (\boldsymbol{f}_t(\boldsymbol{s} \odot (\boldsymbol{1}-\boldsymbol{q})+\boldsymbol{q}) - \boldsymbol{q})\oslash(\boldsymbol{1}-\boldsymbol{q}).
%\end{equation}
%\begin{proof}
  Define $\boldsymbol{x} = \boldsymbol{s} \odot (\boldsymbol{1} - \boldsymbol{q}) + \boldsymbol{q}$.
  %  $x_i:=s_i(1-q_i)+q_i$ for all $i \in S$ and $\boldsymbol{x} := (x_i)_{i \in S}$.
  Now
\begin{align*}
  \boldsymbol{F}_t(\boldsymbol{s}) = \boldsymbol{F}_{t-1}(\boldsymbol{F}(\boldsymbol{s})) &= \boldsymbol{F}_{t-1}((\boldsymbol{f}(\boldsymbol{x}) - \boldsymbol{q}) \oslash (\boldsymbol{1} - \boldsymbol{q}))\\
  &= \boldsymbol{F}_{t-2}(\boldsymbol{F}((\boldsymbol{f}(\boldsymbol{x}) - \boldsymbol{q}) \oslash (\boldsymbol{1} - \boldsymbol{q})))\\
  &= \boldsymbol{F}_{t-2}(\boldsymbol{f}((\boldsymbol{f}(\boldsymbol{x}) - \boldsymbol{q}) \oslash (\boldsymbol{1} - \boldsymbol{q}) \odot(\boldsymbol{1} - \boldsymbol{q}) + \boldsymbol{q})-\boldsymbol{q}) \oslash (\boldsymbol{1} - \boldsymbol{q})\\
  &= \boldsymbol{F}_{t-1}(\boldsymbol{f}_2(\boldsymbol{x}) - \boldsymbol{q}) \oslash (\boldsymbol{1} - \boldsymbol{q}).
\end{align*}
Iterating this procedure, we get
\begin{equation*}
  \boldsymbol{F}_t(\boldsymbol{s}) = (\boldsymbol{f}_t(\boldsymbol{x}) - \boldsymbol{q}) \oslash (\boldsymbol{1} - \boldsymbol{q}).
\end{equation*}
\qed
%\end{proof}

\subsection{Proof of Lemma \ref{lemma:expected}}\label{expect_hs}
      Define $\boldsymbol{x} = \boldsymbol{s} \odot (\boldsymbol{1} - \boldsymbol{q}) + \boldsymbol{q}$, and note by the chain rule that
\begin{align*}
    \mathbb{E}(\boldsymbol{Y}^{i_0}_t) %\\
    = \nabla F^{i_0}_t (\boldsymbol{1}) 
    %&= \frac{\partial \left( \frac{f^{i_0}_t(\boldsymbol{s} \odot (\boldsymbol{1}-\boldsymbol{q})+\boldsymbol{q}) - q_{i_0}}{1 - q_{i_0}} \right)}{\partial s_i} \bigg|_{\boldsymbol{s} = \boldsymbol{1}} \\%\label{no_x}
%\end{align}
%  Proceeding from \ref{no_x}, we find
%\begin{align*}
%\frac{\partial \left( \frac{f^i_t(\boldsymbol{s} \odot (\boldsymbol{1}-\boldsymbol{q})+\boldsymbol{q}) - q_{i_0}}{1 - q_{i_0}} \right)}{\partial s_i} \bigg|_{\boldsymbol{s} = \boldsymbol{1}} %\\
    &=\frac{\boldsymbol{1}-\boldsymbol{q}}{1-q_{i_0}}\odot\nabla f^{i_0}_t (\boldsymbol{x}(\boldsymbol{1})) 
    =\frac{\boldsymbol{1}-\boldsymbol{q}}{1-q_{i_0}}\odot\nabla f^{i_0}_t(\boldsymbol{1})
    = \frac{\boldsymbol{1}-\boldsymbol{q}}{1-q_{i_0}}\odot\mathbb{E}(\boldsymbol{Z}_t^{i_0}).
\end{align*}\qed
%\end{proof}

\subsection{Proof of Lemma \ref{lemma_moments}}\label{sums_W}
%We will prove Lemma \ref{lemma_moments} --- that under Assumption \ref{infinite_matrix}, \ref{R-positive} and \ref{moments}, for $j \in [k]$ 
%      \begin{equation*}
%        \mathbb{E}\left(\sum_{i \in S}\sum_{s=1}^{\z}\frac{(W_{t, s}^{(i)})^j}{|\boldsymbol{Z}_t|}\bigm|  \boldsymbol{Z}_t, |\boldsymbol{Z}_t| > 0 \right) = \frac{\sum_{i \in S}\z \mathbb{E}((W^{(i)})^j)}{|\boldsymbol{Z}_t|} 
%    \end{equation*}
%    and 
%    \begin{equation*}
%        \Var\left(\sum_{i\in S}\sum_{s=1}^{\z}\frac{(W_{t, s}^{(i)})^j}{|\boldsymbol{Z}_t|}\bigm|  \boldsymbol{Z}_t, |\boldsymbol{Z}_t| > 0 \right) = \frac{\sum_{i \in S} \z \Var((W^{(i)})^j) }{|\boldsymbol{Z}_t|^2}
%    \end{equation*}
%    \begin{proof}
For $j \in \mathbb{N}_+$, note that %consider the expected value of the sum of $W$s divided by the population size
    \begin{align}
        \mathbb{E}\left(\sum_{i \in S}\sum_{s=1}^{\z}  \frac{W_{t, s}^{(j)}}{|\boldsymbol{Z}_t|}\bigm|  \boldsymbol{Z}_t, |\boldsymbol{Z}_t| > 0 \right) 
        &= \sum_{i \in S}\mathbb{E}\left(\sum_{s=1}^{\z}\frac{W_{t, s}^{(j)}}{|\boldsymbol{Z}_t|}\bigm|  \boldsymbol{Z}_t, |\boldsymbol{Z}_t| > 0\right) \label{eq:monotone_convergence} \\
        &= \frac{1}{\vert\boldsymbol{Z}_t\vert}\sum_{i \in S}\sum_{s=1}^{\z}\mathbb{E}(W_{t, s}^{(j)}\bigm|  \boldsymbol{Z}_t, |\boldsymbol{Z}_t| > 0), \label{eq:independent}
      \end{align}
      where \eqref{eq:monotone_convergence} follows by the monotone convergence theorem for a sum of non-negative random variables (refer to \cite[Thm.\ 8.1, p.\ 164]{kallenberg} for example).

      For any $i \in S$ and $s \in [\z]$, $(W_{t,s}^{(i)})^j$ is independent of $\boldsymbol{Z}_t$ so we have that
      \begin{equation}
      \mathbb{E}((W_{t,s}^{(i)})^j\bigm|  \boldsymbol{Z}_t, |\boldsymbol{Z}_t| > 0) = \mathbb{E}((W_{t,s}^{(i)})^j) \label{eq:drop_condition}.
      \end{equation}
      Returning to \eqref{eq:independent} and using \eqref{eq:drop_condition}, we get
      \begin{align*}
       \frac{1}{\vert \boldsymbol{Z}_t\vert}\sum_{i \in S}\sum_{s=1}^{\z}\mathbb{E}(W_{t, s}^{(j)}\bigm|  \boldsymbol{Z}_t, |\boldsymbol{Z}_t| > 0) 
        &= \frac{1}{\vert \boldsymbol{Z}_t\vert}\sum_{i \in S}\sum_{s=1}^{\z}\mathbb{E}(W_{t, s}^{(j)}) 
        =  \sum_{i \in S}\frac{\z\mathbb{E}( (W^{(i)})^j)}{\vert\boldsymbol{Z}_t\vert},
    \end{align*} 
verifying \eqref{eq:condEidentity}. The proof for the conditional variance identity \eqref{eq:condVidentity} is analogous. \qed
%\end{proof}


\subsection{Proof of Proposition \ref{prop:W_vs_Z1}}\label{W_vs_Z1}

The proof relies on the upper bound from Rosenthal's inequality, which is stated below. 
\begin{theorem}[Theorem 2.12 in \cite{hall}]
Let $(S_t: t=1,\dots,T)$ be a $(\mathcal{F}_t)$-martingale and define $D_1 = S_1$ and $D_{t} = S_t - S_{t-1}$. Then for $2 \le p < \infty$, there exists a constant $C$ dependent only on $p$, such that 
    \begin{equation}
        \mathbb{E}(\lvert S_T\rvert^p)
\le C \left\{
\mathbb{E}\left[\left(\sum_{t=1}^T \mathbb{E}\!\left(D_t^2 \mid \mathcal{F}_{t-1}\right)\right)^{p/2}\right]
+ \sum_{t=1}^T \mathbb{E} (\lvert D_t\rvert^p)
\right\}. \label{eq:martingale_rosenthal}
    \end{equation}
\end{theorem}

\begin{remark}
An analogous statement is true for a sequence of independent random variables $(\xi_i)_{i=1}^T$ and a sub-$\sigma$-algebra $\mathcal{G}$ where for any $i \in [T]$, it holds that $\mathbb{E}(\xi_i \mid \mathcal{G}) = 0$ and $\mathbb{E}(|\xi_i|^p \mid \mathcal{G}) < \infty$, $p \ge 2$ \cite{rosenthal}. Let $M_T = \sum_{i=1}^T \xi_i$, then there exists a positive, finite constant $C$  dependent only on $p$ such that
\begin{equation}
    \mathbb{E}(|M_T|^p\mid \mathcal{G}) \le C \left(  \left(\sum_{i=1}^T \mathbb{E}(\xi_i^2 \mid \mathcal{G})\right)^{p/2}  + \sum_{i=1}^T \mathbb{E}(|\xi_i|^p \mid \mathcal{G})\right). \label{eq:iid_rosenthal}
\end{equation}

\end{remark}

For $T \in \mathbb{N}_+$, consider $W_T^{(i_0)} = \boldsymbol{u} \cdot \boldsymbol{Z}_T^{i_0}/\lambda^T$, a non-negative $\mathcal{F}_T$-martingale, where $\mathcal{F}_T$ denotes the $\sigma$-algebra generated by $\{Z_t\}_{t=1}^T$. The limit $\lim_{T \rightarrow \infty}{W_T^{(i_0)}} = W^{(i_0)}$ exists \cite[Theorem 4, Chapter V]{athreya_ney}. Define $D_{T+1}^{(i_0)} := W_{T+1}^{(i_0)} - W_T^{(i_0)}$.

\begin{lemma}\label{lemma:E_D}
For $p \in [2k]$ and $d < \infty$, there exist a finite, positive constant $a_1$ only dependent on $p$ such that
    \begin{equation} \label{eq:only_D}
         \mathbb{E}(|D_{T+1}^{(i_0)}|^p\mid \mathcal{F}_T) \le a_1  \lambda^{-p(T+1)} |\boldsymbol{u}|^p\left(|\boldsymbol{Z}_T^{i_0}|  + |\boldsymbol{Z}_T^{i_0}|^{p/2}\right).
    \end{equation}
\end{lemma}
\begin{proof}
Observe that, using \eqref{eq:Zdecomposition}, $W_{T+1}^{(i_0)}$ can be expressed as %a sum of families from individuals at generation $T$
\begin{align}
    W_{T+1}^{(i_0)} = \frac{\boldsymbol{u} \cdot \boldsymbol{Z}_{T+1}^{i_0}}{\lambda^{T+1}} = \frac{\boldsymbol{u} \cdot \sum_{i=1}^d \sum_{s=1}^{Z_T^{i_0, i}} \boldsymbol{Z}_{T,1,s}^{i}}{\lambda^{T+1}} \label{W_k1}.
\end{align}
Using \eqref{W_k1}, let us express $D_{T+1}$ in the following way:
\begin{align*}
    D_{T+1}^{(i_0)} &= \frac{\boldsymbol{u} \cdot \sum_{i=1}^d \sum_{s=1}^{Z_T^{i_0, i}} \boldsymbol{Z}_{T,1,s}^{i}}{\lambda^{T+1}} - \frac{\boldsymbol{u} \cdot \boldsymbol{Z}_T^{i_0}}{\lambda^T}\\
    &= \frac{\sum_{i=1}^d \sum_{s=1}^{Z_T^{i_0, i}} \boldsymbol{u} \cdot \boldsymbol{Z}_{T,1,s}^{i}}{\lambda^{T+1}} - \frac{\boldsymbol{u}  \cdot \sum_{i=1}^d \sum_{s=1}^{Z_T^{i_0, i}}\boldsymbol{e}_{i}}{\lambda^t} 
    = \frac{\sum_{i=1}^d \sum_{s=1}^{Z_T^{i_0, i}} (\boldsymbol{u} \cdot \boldsymbol{Z}_{T,1,s}^{i} - \lambda u_{i})}{\lambda^{T+1}}. 
\end{align*}

Note that for any index $s$ and any $i \in S$, $\mathbb{E}( \boldsymbol{u} \cdot \boldsymbol{Z}_{T,1,s}^{i}) = \lambda u_{i}$, which can be proven by first evaluating \eqref{eq:set_equations} at $n=1$ (the first derivative) to get
\begin{equation}
    \lambda u_{i} = \sum_{j=1}^d u_j \mathbb{E}(Z_1^{i, j}) \label{eq:derivative1}
\end{equation}
and then substituting \eqref{eq:derivative1} in the calculation
\begin{equation}
     \mathbb{E}(\boldsymbol{u} \cdot \boldsymbol{Z}_{T,1,s}^i) = \boldsymbol{u} \cdot \mathbb{E}(\boldsymbol{Z}_1^i) = \boldsymbol{u} \cdot (\boldsymbol{e}_{i}^{\top} \boldsymbol{M}) = \sum_{j=1}^d u_j \mathbb{E}(Z_1^{i,j}) = \lambda u_i \label{eq:expectation_minus}.
\end{equation}
We will now use this result to bound %the expected value of $|D_{T+1}^{(i_0)}|^p$ with respect to filtration $\mathcal{F}_{T}$
\begin{align}
    \mathbb{E}(|D_{T+1}^{(i_0)}|^p\mid\mathcal{F}_{T}) %&= \mathbb{E}\left(\left|\frac{\sum_{i=1}^d \sum_{s=1}^{Z_{T}^{i_0, i}} (\boldsymbol{u} \cdot \boldsymbol{Z}_{T,1,s}^{i} - \lambda u_{i})}{\lambda^{T+1}}\right|^p \mid \mathcal{F}_{T}\right) \label{eq:conditioned}\\
    &=  \lambda^{-p(T+1)}\mathbb{E}\left(\left|\sum_{i=1}^d \sum_{s=1}^{Z_T^{i_0, i}} (\boldsymbol{u} \cdot \boldsymbol{Z}_{T,1,s}^{i} - \lambda u_{i})\right|^p\mid \mathcal{F}_{T}\right)\nonumber \\
    &=  \lambda^{-p(T+1)}\mathbb{E}\left(\left|\sum_{i=1}^d \sum_{s=1}^{Z_T^{i_0, i}} (\boldsymbol{u} \cdot \boldsymbol{Z}_{T,1,s}^{i} - \mathbb{E}(\boldsymbol{u} \cdot \boldsymbol{Z}_{T,1,s}^{i}))\right|^p\mid \mathcal{F}_{T}\right). \label{eq:before_rosenthal}
\end{align}

In order to apply inequality \eqref{eq:iid_rosenthal} to \eqref{eq:before_rosenthal} we need to prove that for any $i \in [d]$ and $s \in [Z_T^{i_0,i}]$, $\mathbb{E}(|\boldsymbol{u} \cdot \boldsymbol{Z}_{T,1,s}^{i} - \lambda u_{i}|^p\mid\mathcal{F}_{T}) < \infty$. To do so, we will use the $c_r$ inequality \cite[p.~157]{loeve1977}
\begin{align}
    \mathbb{E}(|\boldsymbol{u} \cdot \boldsymbol{Z}_{T,1,s}^{i} - \mathbb{E}(\boldsymbol{u} \cdot \boldsymbol{Z}_{T,1,s}^{i})|^p\mid\mathcal{F}_{T}) &\le 2^p \mathbb{E}((\boldsymbol{u} \cdot \boldsymbol{Z}_{T,1,s}^{i})^p\mid\mathcal{F}_{T}) \label{eq:binomial} \\
    &= 2^p\mathbb{E}((\boldsymbol{u} \cdot \boldsymbol{Z}_1^{i} )^p) \nonumber \\
    &= 2^p\sum_{j_1=1}^d \cdots \sum_{j_p=1}^d \left(\prod_{s=1}^p u_{j_s}\right) \mathbb{E}\left(\prod_{s=1}^p  Z_1^{i, j_s}\right) \nonumber \\
    &\le 2^p\sum_{j_1=1}^d \cdots \sum_{j_p=1}^d \left(\prod_{s=1}^p u_{j_s}\right) \prod_{s=1}^p  \left(\mathbb{E}\left(\left(Z_1^{i, j_s}\right)^p\right)\right)^{1/p}, \label{eq:holder}
    \end{align}
    where \eqref{eq:holder} follows from H\"{o}lder's inequality.
    
    If $d < \infty$, under condition \eqref{eq:finite_moments} we can say
    that there exists a positive, finite $C_1$ such that for any $i,j \in S$ and for any $p \in [2k]$
    \begin{equation}
        \mathbb{E}((Z_1^{i,j})^p) < C_1. \label{eq:offspring_bound} 
    \end{equation}
Using \eqref{eq:offspring_bound} in \eqref{eq:holder}, we get a finite bound %for $\mathbb{E}(|\boldsymbol{u} \cdot \boldsymbol{Z}_{T,1,s}^{i} - \mathbb{E}(\boldsymbol{u} \cdot \boldsymbol{Z}_{T,1,s}^{i})|^p)$ for any $i \in [d]$ and $s \in [Z_T^{i_0,i}]$
    \begin{align}
     \mathbb{E}(|\boldsymbol{u} \cdot \boldsymbol{Z}_{T,1,s}^{i} - \mathbb{E}(\boldsymbol{u} \cdot \boldsymbol{Z}_{T,1,s}^{i})|^p) \leq 2^p\mathbb{E}((\boldsymbol{u} \cdot \boldsymbol{Z}_1^{i} )^p) < C_1(2 |\boldsymbol{u}|)^p. \label{eq:finite_p}
\end{align}
Now, applying inequality \eqref{eq:iid_rosenthal} to \eqref{eq:before_rosenthal},

  % \mathbb{E}(\boldsymbol{u} \cdot \boldsymbol{Z}_1^{i})^2 < C_1 ||\boldsymbol{u}||^2 \quad \text{ and }
\begin{align}
\mathbb{E}(|D_{T+1}^{(i_0)}|^p\mid\mathcal{F}_{T}) \nonumber\\
 %\lambda^{-p(T+1)}\mathbb{E}\bigg(| & \sum_{i=1}^d \sum_{s=1}^{Z_T^{i_0, i}} \boldsymbol{u} \cdot \boldsymbol{Z}_{T,1,s}^{i} - \lambda u_{i}|^p\mid\mathcal{F}_{T} \bigg) \nonumber \\
    {}\le c_p \lambda^{-p(T+1)}\Bigg( & \sum_{i=1}^d \sum_{s=1}^{Z_T^{i_0, i}} \mathbb{E}(|\boldsymbol{u} \cdot \boldsymbol{Z}_{T,1,s}^{i} - \mathbb{E}(\boldsymbol{u} \cdot \boldsymbol{Z}_{T,1,s}^{i})|^p\mid\mathcal{F}_{T})\nonumber\\
    &{}+ \left(\sum_{i=1}^d \sum_{s=1}^{Z_T^{i_0, i}} \mathbb{E}((\boldsymbol{u} \cdot \boldsymbol{Z}_{T,1,s}^{i} - \mathbb{E}(\boldsymbol{u} \cdot \boldsymbol{Z}_{T,1,s}^{i}))^2\mid\mathcal{F}_{T})\right)^{p/2}\Bigg) \nonumber\\
     \le c_p \lambda^{-p(T+1)}\Bigg(&\sum_{i=1}^d \sum_{s=1}^{Z_T^{i_0, i}} 2^p\mathbb{E}((\boldsymbol{u} \cdot \boldsymbol{Z}_{T,1,s}^{i})^p \mid\mathcal{F}_{T}) 
    + \left(\sum_{i=1}^d \sum_{s=1}^{Z_T^{i_0, i}} 4\mathbb{E}((\boldsymbol{u} \cdot \boldsymbol{Z}_{T,1,s}^{i})^2\mid\mathcal{F}_{T})\right)^{p/2}\Bigg) \label{eq:non_centered} \\
    {}= c_p \lambda^{-p(T+1)}\Bigg(&\sum_{i=1}^d  2^p Z_T^{i_0, i}\mathbb{E}((\boldsymbol{u} \cdot \boldsymbol{Z}_{1}^{i})^p) 
    + \left(\sum_{i=1}^d 4 Z_T^{i_0, i} \mathbb{E}((\boldsymbol{u} \cdot \boldsymbol{Z}_{1}^{i})^2)\right)^{p/2}\Bigg) \\%\label{eq:last_moment}
     {}\le c_p \lambda^{-p(T+1)}\big(&2^p C_1|\boldsymbol{Z}_T^{i_0}|  |\boldsymbol{u}| + \left(4 C_1|\boldsymbol{Z}_T^{i_0}|  |\boldsymbol{u}|\right)^{p/2}\big), \label{eq:upper_bound}
\end{align}
 where \eqref{eq:non_centered} follows by \eqref{eq:binomial} and \ref{eq:upper_bound} requires the substitution of \eqref{eq:finite_p}.

Setting $a_1 := 2^p c_p \max\{C_1, C_1^{p/2}\}$ yields the desired result.
\end{proof}
\begin{lemma} \label{lemma:E_D-unconditional}
%Based on Lemma \ref{lemma:E_D}, we can prove that for $p \in [2k]$ and some positive, finite constants $a_1, a_2$, we have that
Under Assumptions \ref{infinite_matrix} and \ref{R-positive}, for $p \in [2k]$, there exists a finite, positive constant $a_2$ such that
    \begin{equation}
        \mathbb{E}(|D_{T+1}^{(i_0)}|^p) \le a_2 |\boldsymbol{u}|^p\left(\lambda^{-p(T+1)+T} + \lambda^{-p(T/2+1)}\right).
    \end{equation}
\end{lemma}
\begin{proof}
Using the tower property for %\eqref{eq:before_rosenthal} and combining with
\eqref{eq:only_D}, we get
\begin{align}
   \mathbb{E}(|D_{T+1}^{(i_0)}|^p) &\le %\mathbb{E}\left(c_p \lambda^{-p(T+1)}\left(2^p C_1|\boldsymbol{Z}_T^{i_0}| |\boldsymbol{u}| + \left(4 C_1|\boldsymbol{Z}_T^{i_0}| |\boldsymbol{u}|\right)^{p/2}\right)\right) \nonumber \\
   %&=
   a_1\lambda^{-p(T+1)}|\boldsymbol{u}|^{p}\left(\mathbb{E}(|\boldsymbol{Z}_T^{i_0}|) + \mathbb{E}\left(|\boldsymbol{Z}_T^{i_0}|^{p/2} \right)\right). \label{eq:before_perron}
\end{align}
For a matrix $\boldsymbol{M}$, let $|\boldsymbol{M}|$ denote the sum of its elements. Under Assumptions \ref{infinite_matrix} and \ref{R-positive}, by the Perron--Frobenius theorem (see e.g.\ \cite[Theorem II.5.1]{harris1963}), we can express $\boldsymbol{M}^T$ as
\begin{equation*}
    \boldsymbol{M}^T = \lambda^T \boldsymbol{M}_1 +  \boldsymbol{M}_2^T,
\end{equation*}
where $\boldsymbol{M}_1 = (u_i \nu_j)_{i,j \in [d]}$ and $|\boldsymbol{M}_2^T| = O(\alpha^T)$ as $T\to\infty$ for some $\alpha \in (0,\lambda)$. Further, one can express the expected population size at generation $T$ in terms of the $T$th power of $\boldsymbol{M}$:
\begin{equation*}
\mathbb{E}(\boldsymbol{Z}_T^{i_0}) = (\boldsymbol{e}_{i_0}^{\top} \boldsymbol{M}^T)^{\top},
\end{equation*}
and so under condition \eqref{eq:finite_moments} we have for some positive, finite constant $C_2$ that 
\begin{equation}
    \mathbb{E}(|\boldsymbol{Z}_T^{i_0}|) \le C_2 \lambda^T  \label{eq:frobenius}
\end{equation}
and $\mathbb{E}(|\boldsymbol{Z}_T^{i_0}|^p)$ grows at most at the same rate as $\lambda^{Tp}$,  i.e.\ $\mathbb{E}(|\boldsymbol{Z}_t^{i_0}|^p)< C_3 \lambda^{tp}$ for some positive, finite $C_3$. Hence \eqref{eq:before_perron} becomes %and using \eqref{eq:frobenius}:
\begin{align*}
\mathbb{E}(|D_{T+1}^{(i_0)}|^p)
%&c_p \lambda^{-p(T+1)}\left( 2^p C_1|\boldsymbol{u}|\mathbb{E}(|\boldsymbol{Z}_T^{i_0}|) + 4^{p/2} C_1^{p/2}\mathbb{E}\left(|\boldsymbol{Z}_T^{i_0}| \right)^{p/2}\right) \\
    &\le a_1 \lambda^{-p(T+1)}|\boldsymbol{u}|^p\left( C_2 \lambda^T   +   C_3 \lambda^{Tp/2}\right) 
    %= \frac{a_1 \lambda^T}{\lambda^{p(T+1)}} + \frac{a_2 \lambda^{pT/2}}{\lambda^{p(T+1)}}
    = a_1 |\boldsymbol{u}|^p\left(\frac{ C_2}{\lambda^{p(T+1)-T}} + \frac{C_3}{\lambda^{p(T/2+1)}}\right).
\end{align*}
Setting $a_2 := a_1 \max\{C_2,  C_3 \}$ yields the desired result.
\end{proof}

Now that we have an upper bound for the $p$th moment of $D_T^{(i_0)}$, we apply the inequality \eqref{eq:martingale_rosenthal} for the martingale $(W^{(i_0)}_T)$, applying Lemmata \ref{lemma:E_D} and \ref{lemma:E_D-unconditional} to bound the respective expectations that appear:
\begin{align}
&\mathbb{E}(\lvert W_T^{(i_0)}\rvert^p)  \nonumber\\
&\le C \left\{
\mathbb{E}\left[\left(\sum_{t=1}^T \mathbb{E}\!\left((D_t^{(i_0)})^2 \mid \mathcal{F}_{t-1}\right)\right)^{p/2}\right]
+ \sum_{t=1}^T \mathbb{E}\left(\lvert D_t^{(i_0)}\rvert^p\right)
\right\} \nonumber \\
% Rosenthal
&\le C \left\{
\mathbb{E}\left[\left(\sum_{t=1}^T 2 a_1 \lambda^{-2t} |\boldsymbol{u}|^2 |\boldsymbol{Z}_{t-1}^{i_0}| \right)^{p/2}\right]+
\sum_{t=0}^{T-1} a_2 |\boldsymbol{u}|^p\left(\lambda^{-p(t+1)+t} + \lambda^{-p(t/2+1)}\right)
\right\} \nonumber \\
&= C |\boldsymbol{u}|^{p}\left\{
(2a_1)^{p/2}\mathbb{E}\left(\left( \sum_{t=1}^T \lambda^{-2t}|\boldsymbol{Z}_{t-1}^{i_0}| \right)^{p/2}\right)  + \frac{a_2  (1-\lambda^{-T(p-1)})}{\lambda^{p}-\lambda} + \frac{a_2 (1-\lambda^{-pT/2})}{\lambda^p(1-\lambda^{-p/2})}
\right\} \label{eq:needs_minkowski}.
\end{align}
Let us apply Minkowski's inequality \cite[e.g.][Theorem 1.31]{kallenberg} to the first term in \eqref{eq:needs_minkowski}:
\begin{equation}
\mathbb{E}\left(\left(\sum_{t=1}^T \lambda^{-2t}|\boldsymbol{Z}_{t-1}^{i_0}| \right)^{p/2}\right) \le   
\left(\sum_{t=1}^T \left(\mathbb{E}\left((\lambda^{-2t}|\boldsymbol{Z}_{t-1}^{i_0}|)^{p/2}\right) \right)^{2/p}\right)^{p/2} \label{eq:last_Z}.
\end{equation}
We now upper bound $\mathbb{E}\left(|\boldsymbol{Z}_{t-1}^{i_0}|^{p/2}\right)$ in \eqref{eq:last_Z} by $C_4 \lambda^{(t-1)p/2}$, following the reasoning that led to \eqref{eq:frobenius}:
\begin{align}
\mathbb{E}\left(\left(\sum_{t=1}^T \lambda^{-2t}|\boldsymbol{Z}_{t-1}^{i_0}| \right)^{p/2}\right) 
& \le \left(\sum_{t=1}^T \lambda^{-2t}(C_4 \lambda^{(t-1)p/2} )^{2/p}\right)^{p/2} \nonumber\\
&=  \left(\sum_{t=1}^T \lambda^{-t-1}C_4^{2/p} \right)^{p/2}
= C_4\left(\frac{1-\lambda^{-T}}{\lambda^2(1-\lambda^{-1})} \right)^{p/2} \label{eq:after_minkowski}.
\end{align}
Let us substitute \eqref{eq:after_minkowski} into \eqref{eq:needs_minkowski} to obtain:
\begin{align}
   & \mathbb{E}\left(\lvert W_T^{(i_0)}\rvert^p\right)\nonumber \\
   %&\le C c_p  (2 C_1 |\boldsymbol{u}|)^{p}\left\{c_p^{(p-2)/2}\mathbb{E}\left(\left(\sum_{t=1}^T \lambda^{-2t}|\boldsymbol{Z}_{t-1}^{i_0}| \right)^{p/2}\right)  + \frac{C_2  (1-\lambda^{-T(p+1)})}{\lambda(\lambda^{p+1}-1)} + \frac{C_3 (1-\lambda^{-pT/2})}{\lambda(1-\lambda^{-p/2})}\right\} \nonumber \\
& \le C |\boldsymbol{u}|^{p} \left\{ (2a_1)^{p/2} C_4
\left(\frac{1-\lambda^{-T}}{\lambda^2(1-\lambda^{-1})} \right)^{p/2}  +  \frac{a_2  (1-\lambda^{-T(p-1)})}{\lambda^{p}-\lambda} + \frac{a_2 (1-\lambda^{-pT/2})}{\lambda^p(1-\lambda^{-p/2})}
\right\}\label{eq:many_constants}.
\end{align}
By Fatou's Lemma \cite[Chapter 6, Theorem 2]{shiryaev}, for any random variables $\eta, \xi_1, \xi_2,\dots$ that satisfy for all $n \in \mathbb{N}_+$$, \xi_n > \eta$ and $\mathbb{\eta} > -\infty$, we have that 
\begin{equation}
    \mathbb{E}\left(\liminf_{n\to\infty} \xi_n\right)\le \liminf_{n\to\infty} \mathbb{E}(\xi_n) \label{eq:fatou}.
\end{equation}
Now, we will upper bound our quantity of interest, $\mathbb{E}(|W^{(i_0)}|^p)$, using \eqref{eq:many_constants}:
\begin{align}
 \mathbb{E}(\lvert W^{(i_0)}\rvert^p) &=\mathbb{E}\left(\lim_{T \rightarrow \infty}\lvert W_T^{(i_0)}\rvert^p\right) =\mathbb{E}\left(\liminf_{T \rightarrow \infty}\lvert W_T^{(i_0)}\rvert^p\right) \nonumber \\
    &\le \liminf_{T \rightarrow \infty}\mathbb{E}(\lvert W_T^{(i_0)}\rvert^p)\label{eq:exchange}  \\
 %   & \le \liminf_{T \rightarrow \infty} C c_p (2 C_1 |\boldsymbol{u}|)^{p} \left\{ c_p^{(p-2)/2} C_4 \left(\frac{1-\lambda^{-T}}{\lambda(1-\lambda)} \right)^{p/2}  + \frac{C_2  (1-\lambda^{-T(p+1)})}{\lambda(\lambda^{p+1}-1)} + \frac{C_3 (1-\lambda^{-pT/2})}{\lambda(1-\lambda^{-p/2})}\right\} \nonumber \\
& \le C |\boldsymbol{u}|^{p} \left\{ (2a_1)^{p/2} C_4 (\lambda^2(1-\lambda^{-1}))^{-p/2}  +  \frac{a_2}{\lambda^{p}-\lambda} + \frac{a_2}{\lambda^p(1-\lambda^{-p/2})}
\right\}, \label{eq:not_infty}
\end{align}
where \eqref{eq:exchange} follows from \eqref{eq:fatou}. If $d < \infty$, \eqref{eq:not_infty} is bounded for any $p \in [2k]$.

\subsection{Extinction Probability and Generating Functions} \label{f_and_q}
We will prove \eqref{eq:fixed_point}: that for any $t \in \mathbb{N}_+$ and $r \in [2k]$, we have $f_{t,r}^{i_0}(\boldsymbol{q}) = q_{i_0}$.
\begin{proof}
Observe that for any $i \in S$,
\begin{equation}
    f^{i}_{t,r}(\boldsymbol{0}) = f_t^{i}(\boldsymbol{0}). \label{eq:generating0}
\end{equation}
By the definition of the probability of extinction,
    \begin{align}
        q_{i_0} &= \lim_{t \rightarrow \infty} f_t^{i_0}(\boldsymbol{0}) %\nonumber \\
         = \lim_{t \rightarrow \infty} f_{t+1}^{i_0}(\boldsymbol{0}) %\nonumber \\
         = \lim_{t \rightarrow \infty} f_{t+1, r}^{i_0}(\boldsymbol{0}) %\nonumber \\
         = \lim_{t \rightarrow \infty} f_{1, r}^{i_0} (\boldsymbol{f}_{t, r}(\boldsymbol{0})) \label{eq:lim_f}
        \end{align}
    where $\boldsymbol{f}_{t, r}(\boldsymbol{0}) := \{f_{t, r}^{i}(\boldsymbol{0})\}_{i \in S}$. By the continuity of $\boldsymbol{f}_{t,r}$ it is possible to put the limit on the right-hand side of \eqref{eq:lim_f} inside the function:
    \begin{align*}
     \lim_{t \rightarrow \infty} f_{1, r}^{i_0} (\boldsymbol{f}_{t, r}(\boldsymbol{0}))
        & =  f_{1, r}^{i_0} (\lim_{t \rightarrow \infty}\boldsymbol{f}_{t, r}(\boldsymbol{0}))
        = f_{1, r}^{i_0} (\boldsymbol{q}), %\label{eq:lim_q}
        \end{align*}
where the last equality follows from \eqref{eq:generating0}. Iterating the procedure described above we obtain \eqref{eq:fixed_point}. %, we can obtain, for any $t \in \mathbb{N}_+$,  
%\begin{equation*}
%    f_{t,r}^{i_0}(\boldsymbol{q}) = q_{i_0}.
%\end{equation*}
\end{proof}
\section{Additional Algorithms}
\subsection{Sampling from a mixture distribution} 
Algorithm~\ref{alg:samples} describes in more detail the composition sampling method from Section \ref{composite_sampling} for the purpose of sampling $W^{(j)}$, $j \in S$, for a process that either becomes extinct or goes to infinity.
\begin{algorithm}
\caption{Sampling from $W^{(j)}$}
\label{alg:samples}
\begin{algorithmic}[1] % [1] = line numbers
\REQUIRE Density approximation $\tilde{w}_j$, number of samples $n \in \mathbb{N}_+$, extinction probability $q_j$
\FOR{$l= 1$ \TO $n$}
    \STATE With probability $q_j$, sample 0.
    \STATE Otherwise sample $W^{(j)}$ by choosing one of the values $k\frac{2}{\lambda^L z}, k \in \pm[M]$ with probability proportional to $\tilde{w}_j(k\frac{2}{\lambda^L z})$.
\ENDFOR
\end{algorithmic}
\end{algorithm}
\subsection{Estimating coalescence probability through Theorem \ref{exact_thm}}
Algorithm~\ref{alg:probability} describes in more detail the procedure of approximating the coalescence time distribution from Theorem \ref{exact_thm}, which was used in Figure \ref{fig:coalescence} in Example \ref{supercritical}.
\begin{algorithm}
\caption{Simulation procedure with densities $\boldsymbol{w}$}
\label{alg:probability}
\begin{algorithmic}[1]
\REQUIRE Number of successful simulations $n$, set of approximated densities $\boldsymbol{w}$, generation of interest $t$
\STATE Set $N = 0$.
\WHILE{$N < n$}
    \STATE Simulate $\boldsymbol{Z}_{t}$ from generation 0 to generation $t$.
    \FOR{$i \in [d]$}
        \STATE Sample $\z$ copies of $W^{(i)}$ using Algorithm \ref{alg:samples}.
    \ENDFOR
    \IF{(for all $i \in [d]: W^{(i)} = 0 )\lor( \z = 0)$}
    \STATE Discard the simulation.
    \ELSE
        \STATE Compute and save
        \begin{equation}
       a_N :=     \frac{\sum_{i=1}^d\sum_{s=1}^{\z} (W_{t, s}^{(i)})^k}
            {\left(\sum_{i=1}^d\sum_{s=1}^{\z} W_{t, s}^{(i)}\right)^k}.
        \end{equation}
        \STATE Set $N = N+1$.
    \ENDIF
\ENDWHILE
\STATE Compute the coalescence probability estimate by averaging all the stored values $(a_j)_{j=0}^{n-1}$
\begin{equation}
    1 - \frac{1}{n}\sum_{j=0}^{n-1} a_j.
\end{equation}
\end{algorithmic}
\end{algorithm}
\subsection{Algorithm for direct genealogy simulation}\label{genealogy_simulation}
To assess the results of our method (Theorems \ref{exact_thm} and \ref{hs_thm}), we perform a direct simulation of the genealogy of a branching process using Algorithm~\ref{alg:direct}. The output is plotted in Figure \ref{fig:coalescence} in Example \ref{supercritical}.
\begin{algorithm}
\caption{Direct simulation of genealogy}
\label{alg:direct}
\begin{algorithmic}[1]
\REQUIRE Number of simulations $n \in \mathbb{N}_+$, number of generations simulated $T \in \mathbb{N}_+$, generation of interest for the distribution function $t \in [0, T)$, root ancestor $\boldsymbol{Z}_0 = \boldsymbol{e}_{i_0}$.
\FOR{$j = 1$ \TO $n$}
    \FOR{$l=1$ \TO $t$}
    \STATE Knowing $\boldsymbol{Z}_{l-1}$, simulate $\boldsymbol{Z}_l$ according to the dynamics of the process.
\ENDFOR
    \STATE Index individuals at generation $t$ with $[|\boldsymbol{Z}_t|]$ and store the indices in a vector $\boldsymbol{I}_t$.
    \FOR{$j=t+1$ \TO $T$}
        \STATE Simulate $\boldsymbol{Z}_j$ based on the population at generation $j-1$, $\boldsymbol{Z}_{j-1}$. Copy the index of the parent from $\boldsymbol{I}_{j-1}$ to their offspring at generation $j$. Store the new indices of the offspring in a new vector $\boldsymbol{I}_{j}$. 
    \ENDFOR
    \STATE Sample $k$ individuals at generation $T$ and compare their corresponding elements of $\boldsymbol{I}_T$.
\ENDFOR
\end{algorithmic}
\end{algorithm}
%% else use the following coding to input the
% \printbibliography
\bibliographystyle{elsarticle-num}
\bibliography{biblio}

\end{document}